\documentclass[12pt]{amsart}
\usepackage{geometry}
\usepackage[plainpages,urlcolor=blue]{hyperref}
\usepackage{amscd}
\usepackage{amssymb}
\usepackage{mathrsfs}
\usepackage{amsmath}
\usepackage{amsthm}
\usepackage[all,cmtip]{xy}
\usepackage{enumerate}
\usepackage{ytableau}
\usepackage{tikz}
\usepackage{mathtools}
\usepackage{color}

\usepackage{charter}

\usepackage{comment}
\usepackage{titletoc} 

\theoremstyle{plain}
\newtheorem{thm}{Theorem}[section]

\newtheorem{prop}[thm]{Propsition}
\theoremstyle{definition}
\newtheorem{defn}[thm]{Definition}
\theoremstyle{remark}
\newtheorem{rmk}[thm]{Remark}

\theoremstyle{notation}
  
\theoremstyle{lemma}
\newtheorem{lem}[thm]{Lemma}
\theoremstyle{Corollary}
\newtheorem{coro}[thm]{Corollary}
\newtheorem{scho}[thm]{Scholium}
\theoremstyle{example}
\newtheorem{exam}[thm]{Example}

\newcommand{\thmref}[1]{Theorem~\ref{#1}}
\newcommand{\secref}[1]{Section~\ref{#1}}
\newcommand{\lemref}[1]{Lemma~\ref{#1}}
\newcommand{\schref}[1]{Scholium~\ref{#1}}
\newcommand{\propref}[1]{Proposition~\ref{#1}}
\newcommand{\corref}[1]{Corollary~\ref{#1}}
\newcommand{\defref}[1]{Definition~\ref{#1}}

\newcommand{\eqnref}[1]{(\ref{#1})}
\newcommand{\exref}[1]{Example~\ref{#1}}
\newcommand{\figref}[1]{Figure~\ref{#1}}
\newcommand{\rmkref}[1]{Remark~\ref{#1}}

\newcommand{\be}{\begin{equation}}
\newcommand{\ee}{\end{equation}}

\newcommand{\ra}{\rightarrow}

\newcommand{\mc}{\mathcal}

\newcommand{\C}{{\mathbb C}}

\newcommand{\Z}{{\mathbb Z}}
\newcommand{\N}{{\mathbb N}}

\newcommand{\Q}{{\mathbb Q}}
\newcommand{\bU}{{\mathbb U}}

\newcommand{\CA}{{\mathcal A}}

\newcommand{\CP}{{\mathcal P}}

\newcommand{\CI}{{\mathcal I}}

\newcommand{\Ind}{{\rm{Ind}}}
\newcommand{\Ker}{{\rm{Ker}}}
\newcommand{\id}{{\rm{id}}}

\newcommand{\ty}{{\rm{ty}}}

\newcommand{\mf}{\mathfrak}

\newcommand{\fs}{{\mathfrak s}}  
\newcommand{\ft}{{\mathfrak t}}

\newcommand{\fP}{{\mf P}}
\newcommand{\fc}{{\mf c}}

\newcommand{\cB}{\mc B}

\newcommand{\cT}{{\mc T}}

\newcommand{\bi}{\mathbf i}
\newcommand{\bj}{\mathbf j}

\newcommand{\La}{\Lambda}
\newcommand{\la}{\lambda}

\newcommand{\de}{\delta}

\newcommand{\tcp}{2\CP_r}

\newcommand{\End}{{\rm{End}}}

\newcommand{\Hom}{{\rm{Hom}}}

\newcommand{\Sym}{{\rm{Sym}}}

\newcommand{\inv}{{^{-1}}}
\advance\headheight by 2pt

\newcommand{\ot}{\otimes}

\newcommand{\sdim}{{\rm{sdim}}}

\newcommand{\fso}{{\mathfrak {so}}}

\newcommand{\osp}{{\mathfrak {osp}}}

\newcommand{\Sp}{{\rm Sp}}
\newcommand{\Or}{{\rm O}}

\newcommand{\OSp}{{\rm OSp}}
\newcommand{\Std}{{\rm Std}}
\newcommand{\Tab}{{\rm Tab}}

\numberwithin{equation}{section}
\makeatletter
\def\moverlay{\mathpalette\mov@rlay}
\def\mov@rlay#1#2{\leavevmode\vtop{%
   \baselineskip\z@skip \lineskiplimit-\maxdimen
   \ialign{\hfil$\m@th#1##$\hfil\cr#2\crcr}}}
\newcommand{\charfusion}[3][\mathord]{
    #1{\ifx#1\mathop\vphantom{#2}\fi
        \mathpalette\mov@rlay{#2\cr#3}
      }
    \ifx#1\mathop\expandafter\displaylimits\fi}
\makeatother

\begin{document}
\title[SFT of invariant theory for $\OSp(m|2n)$]{On the second fundamental theorem  of  invariant theory for the orthosymplectic supergroup}
\author{Yang Zhang}
\address{School of Mathematical Sciences,
	University of Science and Technology of China, Hefei, China}
\address{School of Mathematics and Statistics, University of Sydney, Sydney, Australia}
\email{yang91@mail.ustc.edu.cn}
\begin{abstract}
Let $\OSp(V)$ be the orthosymplectic supergroup on an orthosymplectic vector superspace $V$ of superdimension $(m|2n)$.	Lehrer and Zhang showed that  there is a surjective algebra homomorphism $F_r^r: B_r(m-2n)\rightarrow {\rm{End}}_{{\rm OSp}(V)}(V^{\otimes r})$, where $B_r(m-2n)$ is the Brauer algebra of degree $r$ with parameter $m-2n$.  The second fundamental theorem of invariant theory in this setting seeks to describe the kernel ${\rm Ker} F_r^r$ of $F_r^r$ as a 2-sided ideal of $B_r(m-2n)$. In this paper, we show that ${\rm Ker} F_r^r\neq 0$ if and only if $r\geq r_c:=(m+1)(n+1)$, and give a basis and a dimension formula  for ${\rm Ker} F_r^r$.  We show that  ${\rm Ker} F_r^r$   as a 2-sided ideal of $ B_r(m-2n)$ is generated by ${\rm Ker} F_{r_c}^{r_c}$ for any $r\ge r_c$, and we provide an explicit set of generators for ${\rm Ker} F_{r_c}^{r_c}$.  These generators coincide in the  classical case with those obtained in recent papers of  Lehrer and  Zhang on the second fundamental theorem of invariant theory for the orthogonal and symplectic groups. As an application we obtain the necessary and sufficient conditions for the endomorphism algebra $\rm{End}_{{\mathfrak {osp}}(V)}(V^{\otimes r})$ over the orthosymplectic Lie superalgebra ${\mathfrak {osp}}(V)$ to be isomorphic to $B_r(m-2n)$	
\end{abstract}



\maketitle


\section{Introduction}

Let $\OSp(V)$ be the orthosymplectic supergroup on an orthosymplectic vector superspace $V$ of superdimension $(m|2n)$.
G. Lehrer and R. Zhang \cite{LZ2} (also see \cite{DLZ}) constructed a full tensor functor $F$ from the Brauer category $\cB(m-2n)$ with parameter $m-2n$ to the category of tensor modules for $\OSp(V)$. When restricted to homomorphism spaces $B_k^\ell(m-2n)$ of $\cB(m-2n)$, the functor gave rise to surjective linear maps $F_k^\ell: B_k^\ell(m-2n)\to \Hom_{\OSp(V)}(V^{\otimes k}, V^{\otimes\ell})$. Lehrer and Zhang \cite{LZ3} gave a description of the kernel of $F_k^\ell$ for all $k$ and $\ell$.  These results amount to the first and second fundamental theorems (FFT and SFT) of invariant theory for the orthosymplectic supergroup in a category theoretical setting.

Particularly interesting is the case when $k=\ell=r$. In this case, the morphism space $B_r^r(m-2n)$   acquires the structure of a unital associative algebra with the composition of morphisms as the multiplication. This is the same algebra introduced by Brauer \cite{B} in the 30s when studying tensor decompositions for the orthogonal and symplectic groups, which is now known as the Brauer algebra.  Now $F_r^r: B_r^r(m-2n)\to \End_{\OSp(V)}(V^{\otimes r})$ is a surjective algebra homomorphism, which is a generalisation of the celebrated Schur-Weyl-Brauer duality to the orthosymplectic supergroup.  The SFT in this setting  is equivalent to characterising $\Ker F_r^r$ as a 2-sided ideal of the Brauer algebra $B_r^r(m-2n)$, which gives the relations among $\OSp(V)$-invariants in the endomorphism space $\End_{\C}(V^{\ot r})$. 

This paper is devoted to developing the SFT for $\OSp(V)$ in this endomorphism algebra setting,  which particularly includes the orthogonal group $\Or(V)$  and symplectic group $\Sp(V)$ \cite{ LZ1, LZ4,H,HX}  as special cases.   This requires a deep understanding of  $\Ker F_r^r$; see, e.g., \cite{LZ1} for the orthogonal group case. However, Lehrer and Zhang's description for $\Ker F_r^r$  \cite[Corollary 5.8]{LZ3} is actually extrinsic. Even the much simpler problem, that the minimal $r$ such that $\Ker F_r^r\neq 0$, turned out to be  nontrivial. It was conjectured in \cite{LZ3} that this happens if and only if $r\ge r_c:=(m+1)(n+1)$, which was only proved in the case of $\OSp(1|2)$. This conjecture will be completely confirmed in this paper.

As $\Ker F_r^r$ is a $2$-sided ideal of $B_r^r(m-2n)$,  one would like to have a convenient set of generators for it.  This is required, e.g., when constructing a presentation for the endomorphism algebra $\End_{\OSp(V)}(V^{\otimes r})$. In particular,  this allows one to study decompositions of tensor representations of $\OSp(V)$ through the representation theory of the Brauer algebra \cite{CdVP, ES1}. In the classical case of the orthogonal and symplectic groups (i.e., $n=0$ or $m=0$),  it was shown in \cite{LZ1, LZ4} that the kernel is remarkably generated by a single idempotent.

The main result of this paper is a significantly improved SFT for $\OSp(V)$; see \thmref{thm:SFT-refine} (cf. \cite[Corollary 5.8]{LZ3}). Our strategy is as follows:
\begin{enumerate}
	\item Find the minimal $r：=r_c$ such that $\Ker F_r^r\neq 0$;
	\item Find a generating set for the 2-sided ideal $\Ker F_{r_c}^{r_c}$ of the Brauer algebra $B_{r_c}^{r_c}(m-2n)$;
	\item In the general case $r\geq r_c$, we prove that the embedding of $\Ker F_{r_c}^{r_c}$ into $B_r^r(m-2n)$ exactly generates $\Ker F_r^r$, and hence the generators of $\Ker F_{r_c}^{r_c}$ make up a generating set for the 2-sided ideal $\Ker F_r^r$.
\end{enumerate}

Now we describe the content of this paper in more details. 

We consider a natural symmetric group $\Sym_{2r}$ action on the Brauer algebra $B_r^r(m-2n)$.  The composition of Brauer diagrams naturally defines a right action of the Brauer algebra $B_{2r}^{2r}(m-2n)$ on $B_{2r}^0(m-2n)$. This  restricts to a right action of $\C\Sym_{2r}$, which is contained in  $B_{2r}^{2r}(m-2n)$ as a subalgebra. Using a particular  isomorphism $B_{2r}^0(m-2n)\cong B_r^r(m-2n)$, this natural action of $\C\Sym_{2r}$ can be translated to a $\C\Sym_{2r}$ action on $B_r^r(m-2n)$ and its 2-sided ideal $\Ker F_r^r$.

We  investigate thoroughly  the $\C\Sym_{2r}$-module structure on $\Ker F_r^r$. We show that $\Ker F_r^r$ admits a multiplicity-free direct sum of Specht modules over $\C\Sym_{2r}$, which are identified  in $\Ker F_r^r$ explicitly.  This  enables us to obtain a closed dimension formula for $\Ker F_r^r$ and  deduce that $\Ker F_r^r$ is nonzero if and only if $r\ge r_c:=(m+1)(n+1)$ in \thmref{thmKeriso}. Using the  well known basis for Specht modules, we obtain a basis for $\Ker F_r^r$ indexed by standard tableaux on even partitions containing $\la_c:=((2n+2)^{m+1})$ for all $r\geq r_c$.

We introduce combinatorial gadgets which we call standard sequences of increasing types, and use them to construct diagrammatically a set of elements of $\Ker F_{r_c}^{r_c}$.  We prove in \thmref{thmKermini} that this set generates  $\Ker F_{r_c}^{r_c}$ as a $2$-sided ideal of $B_{r_c}^{r_c}(m-2n)$.
In the course of proof, we have used Garnir relations  arising from the representation theory of symmetric groups (see e.g., \cite[\S 7.2]{J}) to reduce the number of generators. Traditionally Garnir relations were used to express polytabloids  in terms of standard ploytabloids.

In the general case $r\geq r_c$, recall that there is a canonical embedding of Brauer algebras $B_{r_c}^{r_c}(m-2n)\hookrightarrow B_{r}^r(m-2n)$. We prove in \thmref{thm:SFT-refine} that the canonical embedding of $\Ker F_{r_c}^{r_c}$  generates $\Ker F_r^r$ as a 2-sided ideal of $B_{r}^r(m-2n)$.  This proof is based on the fact that $\Ker F_r^r$  as a $\C\Sym_{2r}$-module (not as a 2-sided ideal of $B_r^r(m-2n)$) is generated by a single element which we construct explicitly in terms of Brauer diagrams. Therefore, the  generators for $\Ker F_{r_c}^{r_c}$  form a generating set for $\Ker F_r^r$. 

The above  results can be significantly sharpened in the case of  $\OSp(1|2n)$. We show in \thmref{thmgen} that $\Ker F_r^r$  as a 2-sided ideal of $B_r^r(1-2n)$  is singly generated,  similar to the situation of the orthogonal and symplectic groups  \cite{LZ1,LZ4}.  For general $\OSp(V)$, we do not know whether $\Ker F_r^r$ is singly generated,  but we feel that the answer probably is negative.

We consider two applications of the general results on the SFT of $\OSp(V)$ obtained here.

We obtain in \thmref{thm:osp-end} the necessary and sufficient conditions for  the Brauer algebra $B^r_r(m-2n)$  to be isomorphic to the endomorphism algebra $\End_{\osp(V)}(V^{\ot r})$. This answers completely the isomorphism question of  Ehrig and Stroppel \cite{ES}, thus considerably strengthens their main result \cite[Theorem  A]{ES}.

In \secref{Secclassical}, we recover the main results in the papers \cite{ LZ1, LZ4} of Lehrer and Zhang on SFTs for the orthogonal and symplectic groups as special cases of results proved here. This contextualises the treatment in those papers and provides uniform and simpler proofs for all main theorems.

\vskip 0.3cm

\noindent {\bf Acknowledgements:}
I would like to thank Professors Gus Lehrer and Ruibin Zhang for many enlightening discussions and helpful suggestions. This work was supported at different stages by student stipends from the China Scholarship Council and the Australian Research Council.

\section{Invariant theory of the orthosymplectic supergroup}
We recall  the FFT and SFT of invariant theory for the orthosymplectic supergroup given in  \cite{LZ2, LZ3}.
We  work over the complex number field $\C$ throughout this paper.

\subsection{Tensor representations of the orthosymplectic supergroup}
Let  $V=V_{\bar{0}} \oplus V_{\bar{1}}$  be a complex vector superspace  with superdimension $\sdim(V)=(m|\ell)$, which means that $\dim (V_{\bar{0}})=m$ and $\dim (V_{\bar{1}})=\ell$. The \textit{parity} $[v]$ of any homogeneous element $v\in V_{\bar{i}}$ is defined by $[v]:=\bar{i}$ ($i=0,1$). 
Assume that $V$ admits a non-degenerate even bilinear form
\begin{eqnarray}\label{eq:form} (-,-):\; V\times V\longrightarrow \C,\end{eqnarray}
which is supersymmetric in the sense that $(u,v)=(-1)^{[u][v]}(v, u)$ for $u,v\in V$. This implies that the form is symmetric on $V_{\bar0}\times V_{\bar0}$ and skew-symmetric on $V_{\bar1}\times V_{\bar1}$, and satisfies $(V_{\bar0},  V_{\bar1})=0=(V_{\bar1}, V_{\bar0})$. Therefore $\ell$ must be even. 

We refer to  Harish-Chandra super pair \cite{CCF,DM} as the  supergroup. Let $\osp(V)$ be the orthosymplectic Lie superalgebra \cite{K} preserving the bilinear form \eqref{eq:form}.  Let ${\rm O}(V_{\bar 0})$ and ${\rm Sp}(V_{\bar 1})$ be the orthogonal and  symplectic groups, which are the isometry algebraic groups preserving the restrictions of the form  \eqref{eq:form} to $V_{\bar 0}$ and to $V_{\bar 1}$ respectively. Then $\OSp(V)_0:={\rm O}(V_{\bar 0})\times{\rm Sp}(V_{\bar 1})$ naturally acts on $\osp(V)$ as automorphisms, and we have the Harish-Chandra super pair $\left(\OSp(V)_0, \osp(V)\right)$.  One may regard this as the orthosymplectic supergroup \cite{DM},  and  hereafter $\OSp(V)$ denotes the Harish-Chandra super pair $(\OSp(V)_0, \osp(V))$.

An $\OSp(V)$-module $M$ is defined to be a module for the Harish-Chandra super pair, namely, $M$ is a vector superspace that is a module for both $\osp(V)$
and $\OSp(V)_0$ (as an algebraic group) such that the two actions are compatible
with the action of  $\OSp(V)_0$ on $\osp(V)$. The subspace of invariants is given by
\[
M^{\OSp(V)}=\left\{v\in M
\mid
g v = v, \ \  Xv=0, \ \
\forall \  g\in \OSp(V)_0, \ \  X\in \osp(V)
\right\}.
\]
If $N$ is another $\OSp(V)$-module, then $\Hom_\C(M, N)$ is naturally an $\OSp(V)$-module with the action defined for any $g\in \OSp(V)_0$ and $X\in \osp(V)$ on $\phi\in \Hom_\C(M, N)$ by
\[
(g.\phi)(v) = g \phi( g^{-1}v), \quad (X.\phi)(v)= X\phi(v) - (-1)^{[X][\phi]} \phi(Xv), \quad \forall v\in M.
\]
The second equation can be  extended to  inhomogeneous elements by linearity as usual.

We are largely interested in the category  $\cT_{\OSp(V)}(V)$ of tensor representations of $\OSp(V)$, which has as objects the  tensor powers $V^{\ot r}, r\in \N$ and  is the full subcategory of $\OSp(V)$-modules \cite{LZ2}. The usual tensor product of $\OSp(V)$-modules and of $\OSp(V)$-homomorphisms endows $\cT_{\OSp(V)}(V)$ with the structure of a strict symmetric tensor category \cite{JS}.

There are some special morphisms in $\cT_{\OSp(V)}(V)$.
Let $v_a$ ($a=1, 2, \dots, m+2n$) be a homogeneous basis of $V$, and let
$\bar{v}_a$ ($a=1, 2, \dots, m+2n$) be the dual basis, that is $(\bar{v}_a, v_b)=\delta_{a b}$ for all $a, b$.  Let $c_0=\sum_{a=1}^{m+2n} v_a\otimes\bar{v}_a$, which is canonical in that it is independent of the choice of the basis.
The following linear maps are clearly morphisms of $\cT_{\OSp(V)}(V)$:
\begin{eqnarray}\label{P-C-C}
\begin{aligned}
&\tau: V\otimes V\longrightarrow V\otimes V, \quad v\otimes w \mapsto (-1)^{[v][w]}w\otimes v, \\
&\check{C}: \C \longrightarrow V\otimes V,  \quad 1\mapsto c_0, \\
&\hat{C}: V\otimes V \longrightarrow \C, \quad v\otimes w\mapsto (v, w).
\end{aligned}
\end{eqnarray}
A set of relations among $\tau, \check{C}$ and $\hat{C}$ were given in \cite[Lemma 5.3]{LZ2}.

\subsection{The Brauer category}\label{SecBracat}
The Brauer category was  introduced in \cite{LZ4} to study the invariant theory for symplectic and orthogonal groups.  Let $k$ and $l$ be nonnegative integers and assume that $k+l$ is even. A Brauer $(k,l)$-diagram is a graph with $(k+l)/2$ edges and $k+l$ vertices, where the vertices  are arranged in two rows such that there are $l$ vertices in the top row and $k$ vertices in the bottom row, and each vertex is incident to exactly one edge.  Fixing $\de\in \C$, we denote by $B^l_k(\de)$ the vector space over $\C$ with basis consisting of Brauer  $(k,l)$-diagrams. There are two bilinear operations  \cite[Definition 2.3]{LZ4} ,
\begin{eqnarray}\label{eq:products}
\begin{aligned}
&&\text{composition} & \quad \circ:  & & B_l^p(\delta)
\times B_k^l(\delta)\longrightarrow B_k^p(\delta),  \\
&&\text{tensor product} & \quad  \otimes: & & B_p^q(\delta)
\times B_k^l(\delta)\longrightarrow B_{k+p}^{q+l}(\delta),
\end{aligned}
\end{eqnarray}
which are defined as follows:
\begin{itemize}
	\item Let $D_1\in B_l^p(\delta)$ and $D_2\in  B_k^l(\delta)$. We  place $D_1$ above $D_2$ and  then identify vertices in the bottom row of $D_1$ with the corresponding vertices in the top row of $D_2$.  After deleting all closed loops, say $f(D_1,D_2)$ loops, in the concatenation, we obtain a new diagram $D\in  B_k^p(\delta)$, and hence define $D_1\circ D_2:=\delta^{f(D_1,D_2)}D$.
	\item The tensor product $D_1\ot D_2$ with $D_1\in B_p^q(\delta)$ and $D_2\in B_k^l(\delta)$ is obtained by simply placing $D_2$ on the right of $D_1$ without overlapping.
\end{itemize}

\begin{defn}  \cite[Definition 2.4]{LZ4}
    The {\em Brauer category} $\cB(\delta)$
	is the $\C$-linear category equipped with the tensor product bi-functor $\otimes$
	such that
	\begin{enumerate}
		\item the set of objects is $\N=\{0, 1, 2, \dots\}$,  and $\Hom_{\cB(\delta)}(k, l):=B_k^l(\delta)$ for any pair of objects $k, l\in \N$; the composition of morphisms is given by the composition of Brauer diagrams;
		\item the tensor product $k\otimes l$ of objects  $k, l$ is  $k+l$ in $\N$, and the tensor product of morphisms is given by the tensor product of Brauer diagrams.
	\end{enumerate}
\end{defn}
All Brauer diagrams in the Brauer category can be generated by
the four elementary Brauer diagrams
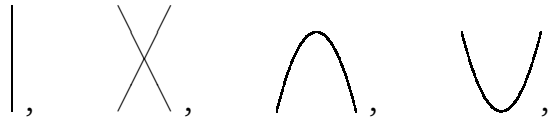
\begin{center}
\begin{picture}(205, 40)(-5,0)
\put(0, 0){\line(0, 1){40}}
\put(5, 0){,}

\put(40, 0){\line(1, 2){20}}
\put(60, 0){\line(-1, 2){20}}
\put(65, 0){,}

\qbezier(100, 0)(115, 60)(130, 0)
\put(135, 0){,}

\qbezier(170, 30)(185, -30)(200, 30)
\put(200, 0){,}
\end{picture}
\end{center}
\smallskip

\noindent
which will be denoted by $I$, $X$, $A_0$ and $U_0$ respectively,
by composition and tensor product. The complete set of relations among these generators is described in \cite[Theorem 2.6(2)]{LZ4}.

The Brauer category has the structure of a strict symmetric tensor category 
\cite{JS} with
$0$ being the identity object. It admits identical left and right dualities, where all objects are self dual, and the evaluation and co-evaluation maps arise from $A_0$ and $U_0$.

The Brauer algebra  $B_r(\delta)$ \cite{B} now arises as the endomorphism space $B_r^r(\delta)$ with the multiplication
given by composition of morphisms  in the Brauer category. It is generated by elements $s_i, e_i, 1\leq i\leq r-1$ as shown below
	\begin{center}
		\begin{tikzpicture}
		\node at (-0.5,0.75){$s_i=$};
		\draw (0,0) -- (0,1.5);
		\draw (1.5,0)--(1.5,1.5);
		\node at (0.75,0.75){$...$};
		\node at (0.75,0.2){$i-1$};
		\draw (1.9,0) -- (2.6,1.5);
		\draw (2.6,0) -- (1.9,1.5);
		\draw (3.0,0) -- (3.0,1.5);
		\draw (4.0,0)--(4.0,1.5);
		\node at (3.5,0.75){$...$};
		\node at (4.2,0.1){$,$};
		
		\node at (5,0.75){$e_i=$};
		\draw (5.5,0) -- (5.5,1.5);
		\draw (7,0)--(7,1.5);
		\node at (6.25,0.75){$...$};
		\node at (6.25,0.2){$i-1$};
		\draw  plot  [smooth,tension=1] coordinates{(7.4,1.5) (7.75,0.8) (8.1,1.5)};
		\draw  plot  [smooth,tension=1] coordinates{(7.4,0) (7.75,0.7) (8.1,0)};
		
		\draw (8.5,0) -- (8.5,1.5);
		\draw (9.5,0)--(9.5,1.5);
		\node at (9,0.75){$...$};	
		\end{tikzpicture}
	\end{center}
 with  relations:
$$\begin{matrix}s_i^2=1,\,\,e_i^2=\delta e_i,\,\,e_is_i=e_i=s_ie_i,
\quad  \,1\leq i\leq r-1,\\
s_is_j=s_js_i,\,\,s_ie_j=e_js_i,\,\,e_ie_j=e_je_i,\quad \,1\leq
i<j-1\leq r-2,\\ 
s_is_{i+1}s_i=s_{i+1}s_is_{i+1},\,\,
e_ie_{i+1}e_i=e_i,\,\,
e_{i+1}e_ie_{i+1}=e_{i+1},\,\, \,1\leq
i\leq r-2,\\
s_ie_{i+1}e_i=s_{i+1}e_i,\,\,e_{i+1}e_is_{i+1}=e_{i+1}s_i,\quad \,1\leq
i\leq r-2.\end{matrix}
$$
In particular, the elements $s_i$ generate the group algebra  $\CA_r:=\C\Sym_r$  of the symmetric group $\Sym_r$ of degree $r$.  Hence $B_r(\delta)$ contains   $\CA_r$ as a subalgebra.

\begin{rmk}\label{rmk:embed}
There is a canonical embedding of $B_r(\delta)$ in $B_{r+1}(\delta)$ for each $r$, which identifies the standard generators of $B_r(\delta)$ with the first $r-1$ pairs of generators $s_i, e_i$ ($i\le r-1$) of $B_s(\delta)$.  This leads to a canonical embedding of
$B_r(\delta) \hookrightarrow B_s(\delta)$ for any $s>r$.
\end{rmk}

\subsection{Categorical FFT and SFT}
We assume throughout the paper that the superdimension $\sdim(V)=(\dim V_{\bar 0}|\dim V_{\bar 1})$ of $V$ is equal to $(m|2n)$.

\begin{thm}{\rm{(\cite[Theorem 5.4]{LZ2})}}\label{thm:functor} There exists a unique tensor functor $F: \cB(m-2n) \longrightarrow \cT_{\OSp(V)}(V)$, which sends the object $r$ to $V^{\otimes r}$ and the morphism $D: k \to \ell$ to $F(D): V^{\otimes k}\longrightarrow V^{\otimes l}$,  where $F(D)$
is defined on the generators of Brauer diagrams by
\[F(I)=\id_V, F(X)=\tau, F(U_0)=\check{C}, F(A_0)=\hat{C}.   \]
\end{thm}

We will denote by $F_k^l: B_k^l(m-2n)\to \Hom_{\OSp(V)}(V^{\otimes k}, V^{\otimes l})$ the restriction of $F$ to the morphism space $B_k^l(m-2n)$.
Let $\hat{A}^r$ be the Brauer diagram in $B_{2r}^0(m-2n)$ corresponding to the partition $\{1,2,\dots,2r\}=\{1,2\}\bigcup\{3,4\}\dots\bigcup \{2r-1, 2r\}$ as shown in Figure \ref{Fig3}.
\setlength{\unitlength}{0.35mm}
\begin{figure}[h]
\begin{center}
\begin{picture}(200, 30)(0,0)
   \qbezier(30, 0)(45, 60)(60, 0)
   \qbezier(80, 0)(95, 60)(110, 0)
   \put(120, 8){$......$}
   \qbezier(150, 0)(165, 60)(180, 0)
   \put(28,-12){$1$}
   \put(58,-12){$2$}
   \put(78,-12){$3$}
   \put(108,-12){$4$}
   \put(135,-12){$2r-1$}
   \put(177,-12){$2r$}
\end{picture}
\end{center}
\caption{The diagram $\hat{A}^r: 2r\to 0$}
\label{Fig3}
\end{figure}

The following results were proved in \cite{LZ2} and \cite{LZ3}, which are the
first and second fundamental theorems of invariant theory for $\OSp(V)$
in the categorical language.

\begin{thm}{ \rm{(FFT, \cite[Theorem 5.6, Corollary 5.8]{LZ2})}}\label{thm:fft-cat}  The functor $F: \cB(m-2n)\rightarrow \cT_{\OSp(V)}(V)$ is full.
	That is, $F_k^l: B_k^l(m-2n)\to \Hom_{\OSp(V)}(V^{\otimes k}, V^{\otimes l})$ is surjective for all $k, l$.
\end{thm}
Therefore we obtain the following result.
\begin{coro}\label{corofft}
	The  algebra homomorphism $F_r^r:  B_r^r(m-2n)\rightarrow \End_{\OSp(V)}(V^{\ot r})$ is surjective.
\end{coro}
Recall that $\Ker F_k^l\cong\Ker F_{k+l}^0$ for all $k, l$ as vector spaces \cite{LZ4} (see also Remark \ref{rmk:iso-kernels}). The method used in the proofs of \cite[Theorem 3.2, Corollary 5.8]{LZ2} allows one to change base rings between $\C$ and the infinite dimensional Grassmann algebra.  Using the same method we obtain the following result from \cite[Theorem 5.4]{LZ3}.

\begin{thm}{ \rm{(SFT, \cite[Theorem 5.4]{LZ3})}}\label{thm:sft}
	If $k+l$ is odd, $\Ker F_{k+l}^0=0$ trivially; if $k+l=2r$ is even,
	$\Ker F_{2r}^0=\hat{A}^r\circ I(m,n)$, where
	$\hat{A}^r\in B^0_{2r}(m-2n)$ is the $(2r, 0)$-diagram shown
	in Figure \ref{Fig3},  and  $I(m, n)$ is the sum of the 2-sided ideals  of $\C\Sym_{2r}$ corresponding
	to partitions which contain an $(m+1)\times (2n+1)$ rectangle.
\end{thm}

Note that $F_{2r}^0: B_{2r}^0(m-2n)\rightarrow \Hom_{\OSp(V)}(V^{\ot 2r},\C)$. In other words, \thmref{thm:sft} gives relations among invariants of $\Hom_{\OSp(V)}(V^{\ot 2r},\C)$, which are linear functions on $V^{\ot 2r}$ that are constant on $\OSp(V)$-orbits. This can be translated to $\Ker F_r^r$ via the isomorphism $\Ker F_{2r}^0\cong \Ker F_r^r$; see \cite[Corollary 5.8]{LZ3}. However,  we would like to obtain a more intrinsic characterisation for $\Ker F_r^r$ in the following section.

\section{$\Sym_{2r}$-module structure on $\Ker F_r^r$}
We shall determine the minimal $r$ such that $\Ker F_r^r\neq 0$ and construct a basis for it.

\subsection{Actions of $\Sym_{2r}$ on $B_{2r}^0(m-2n)$ and  $B_r^r(m-2n)$}
\subsubsection{Actions of $\Sym_{2r}$ on $B_{2r}^0(m-2n)$ and  $B_r^r(m-2n)$}
We begin by  describing a particular isomorphism between $B_{2r}^0(m-2n)$ and $B_r^r(m-2n)$, which will be used extensively in the remainder of the paper.

As a convention, we shall regard any permutation $\sigma\in \Sym_k$ as a Brauer $(k,k)$-diagram with vertices $\{1,2\dots,k\}$ in the bottom row and $\{\sigma(1),\sigma(2),\dots,\sigma(k)\}$ in the top row. Let $\omega_{2r}\in \Sym_{2r}$ be the following diagram 
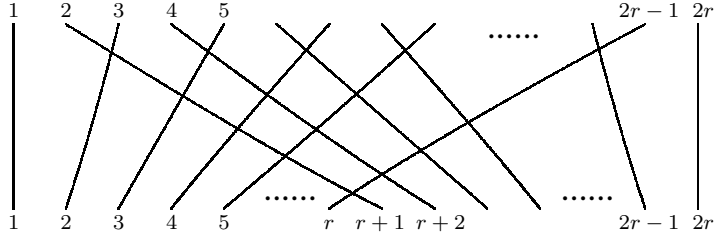
\begin{center}
	\begin{picture}(250, 90)(40,0)
	\qbezier(20, 80)(20, 40)(20, 10)
	\qbezier(60, 80)(50, 40)(40, 10)
	\qbezier(100, 80)(80, 45)(60, 10)
	\qbezier(140, 80)(100, 35)(80, 10)
	\qbezier(180, 80)(130, 35)(100, 10)
	\qbezier(260, 80)(185, 40)(140, 10)
	
	\qbezier(40, 80)(120, 30)(160, 10)
	\qbezier(80, 80)(140, 35)(180, 10)
	\qbezier(120, 80)(160, 45)(200, 10)
	\qbezier(160, 80)(200, 35)(220, 10)
	\qbezier(240, 80)(250, 40)(260, 10)
	\qbezier(280, 80)(280, 40)(280, 10)
	
	\put(18,2){\tiny$1$}
	\put(38,2){\tiny$2$}
	\put(58,2){\tiny$3$}
	\put(78,2){\tiny$4$}
	\put(98,2){\tiny$5$}
	\put(138,2){\tiny$r$}
	\put(150,2){\tiny$r+1$}
	\put(173,2){\tiny$r+2$}
	\put(250,2){\tiny$2r-1$}
	\put(278,2){\tiny$2r$}
	
	\put(18,83){\tiny$1$}
	\put(38,83){\tiny$2$}
	\put(58,83){\tiny$3$}
	\put(78,83){\tiny$4$}
	\put(98,83){\tiny$5$}
	\put(250,83){\tiny$2r-1$}
	\put(278,83){\tiny$2r$}
	
	\put(115, 13) {......}
	\put(200, 75){......}
	\put(228, 13){......}
	\end{picture}
\end{center}
such that $\omega_{2r}(i)=2i-1$ and $\omega_{2r}(r+i)=2i$ for all $1\leq i\leq r$. Let $U_r\in B_0^{2r}(m-2n)$ be the diagram shown in \figref{TranU}.  We define the following map
\[
\mathbb{U}_r: B_{2r}^0(m-2n) \ra B_r^r(m-2n),
\]
which sends $A\in B_{2r}^0(m-2n)$ to the composition
$(I_{r}\ot A)\circ (I_r\ot \omega_{2r})\circ (U_r\ot I_r)$  of Brauer diagrams. Here $I_r$ is the  $r$-th tensor power $I\otimes\dots\otimes I$ of the elementary Brauer diagram $I$, which is the unit in the Brauer algebra $B_r^r(m-2n)$.

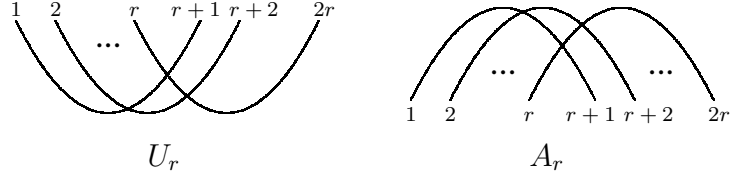
\begin{figure}[ht]
	\begin{center}
		\begin{picture}(250,49) (0,20) 
		\qbezier(0,70)(35,0)(70,70)
		\qbezier(15,70)(50,0)(85,70)
		\qbezier(45,70)(80,0)(115,70)
		\put(-2,72){\tiny$1$}
		\put(13,72){\tiny$2$}
		\put(43,72){\tiny$r$}
		\put(59,72){\tiny$r+1$}
		\put(81,72){\tiny$r+2$}
		\put(113,72){\tiny$2r$}
		\put(30,60){...}
		\put(50,15){$U_r$}
		
		\qbezier(150,40)(185,110)(220,40)
		\qbezier(165,40)(200,110)(235,40)
		\qbezier(195,40)(230,110)(265,40)
		\put(148,32){\tiny$1$}
		\put(163,32){\tiny$2$}
		\put(193,32){\tiny$r$}
		\put(209,32){\tiny$r+1$}
		\put(231,32){\tiny$r+2$}
		\put(263,32){\tiny$2r$}	
		\put(180,50){...}
		\put(240,50){...}
		\put(195,15){$A_r$}
		\end{picture}
	\end{center}
	\caption{$U_r$ and $A_r$}
	\label{TranU}
\end{figure}

\begin{exam}
	{\rm
		Let $r=3$ and
		\begin{center}
			\begin{picture}(130,40)(0,0)
			\put(-30,10){$A=$}
			
			\qbezier(0,10)(30,70)(60,10)
			\qbezier(30,10)(75,70)(120,10)
			\qbezier(90,10)(120,70)(150,10)
			\put(155,4){,}
			\end{picture}
		\end{center}
		then  $\mathbb{U}_3(A)$ can be represented diagrammatically as
		
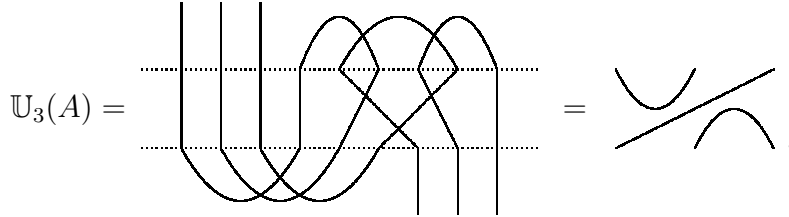
\begin{center}
			\begin{picture}(240,100)(-80,-35)
			\put(-110,12){$\bU_3(A)=$}
			
			\qbezier(0,30)(15,70)(30,30)
			\qbezier(15,30)(37.5,70)(60,30)
			\qbezier(45,30)(60,70)(75,30)
			
			\qbezier(-45,55)(-45,30)(-45,0)
			\qbezier(-30,55)(-30,30)(-30,0)
			\qbezier(-15,55)(-15,30)(-15,0)
			
			\qbezier[80](-60,30)(15,30)(90,30)
			
			\qbezier(0,30)(0,15)(0,0)
			\qbezier(15,30)(30,15)(45,0)
			\qbezier(30,30)(22.5,15)(15,0)
			\qbezier(45,30)(52.5,15)(60,0)
			\qbezier(60,30)(45,15)(30,0)
			\qbezier(75,30)(75,15)(75,0)
			
			\qbezier[80](-60,0)(15,0)(90,0)
			
			\qbezier(-45,0)(-22.5,-40)(0,0)
			\qbezier(-30,0)(-7.5,-40)(15,0)
			\qbezier(-15,0)(7.5,-40)(30,0)
			
			\qbezier(45,0)(45,-10)(45,-25)
			\qbezier(60,0)(60,-10)(60,-25)
			\qbezier(75,0)(75,-10)(75,-25)
			
			\put(100,12){$=$}

			\qbezier(120,30)(135,0)(150,30)
			\qbezier(180,30)(150,15)(120,0)
			\qbezier(150,0)(165,30)(180,0)        		
			\put(185,0){.}
			
			\end{picture}
		\end{center}		
		In particular, $\bU_r(\hat{A}^{ r})=I_r$ for any $r\in \Z_{>0}$.
	}
\end{exam}

Similarly, we introduce the  diagram $A_r$ as shown in
Figure \ref{TranU}, and define the map
\[
\mathbb{A}_r: B_r^r(m-2n)\longrightarrow B_{2r}^0(m-2n),\quad D\mapsto A_r\circ (I_r\ot D)\circ \omega_{2r}^{-1},
\]
Clearly,  the linear maps
$\bU_r$ and $\mathbb{A}_r$ are mutual inverses.

Now $B_k^k(m-2n)=B_k(m-2n)$, the Brauer algebra of degree $k$. For any even number $k=2r$, the composition of morphisms in $\cB(m-2n)$ naturally defines
a right action of the Brauer algebra of degree $2r$ on $B_{2r}^0(m-2n)$:
\[
B_{2r}^0(m-2n)\times B_{2r}^{2r}(m-2n)\longrightarrow  B_{2r}^0(m-2n), \quad
(A, D)\mapsto A\circ D.
\]
Since the group algebra $\CA_{2r}:=\C\Sym_{2r}$ is a subalgebra of the Brauer algebra, this restricts to a right action of $\CA_{2r}$.  By using the  anti-automorphism $\sharp$ of $\CA_{2r}$ induced by the map $\sigma\rightarrow \sigma\inv$ for $\sigma\in \Sym_{2r}$ (i.e., the antipode
of $\CA_{2r}$ with the standard Hopf algebra structure), we can turn the right  $\CA_{2r}$-action into a left action
\begin{eqnarray}\label{eq:action-1}
\ast:  \CA_{2r}\times B_{2r}^0(m-2n)\longrightarrow  B_{2r}^0(m-2n),
\quad \alpha\ast A:=A\circ(\alpha)^{\sharp}.
\end{eqnarray}

Using the vector space isomorphism $\bU_r$ between $B_r^r(m-2n)$ and
$B_{2r}^0(m-2n)$ and its inverse map $\mathbb{A}_r$, we can translate the above
$\CA_{2r}$-action on $B_{2r}^0(m-2n)$ to $B_r^r(m-2n)$:
\begin{eqnarray}\label{eq:action-2}
\ast:  \CA_{2r}\times B_r^r(m-2n)\longrightarrow  B_r^r(m-2n), \quad
\alpha\ast D :=\bU_r(\mathbb{A}_{r}(D)\circ \alpha^\sharp).
\end{eqnarray}

\begin{lem} \label{AUmaps}
	The maps
	\[
	\begin{aligned}
	\bU_r: B_{2r}^0(m-2n)\ra B_r^r(m-2n),\quad
	\mathbb{A}_r:  B_r^r(m-2n) \ra B_{2r}^0(m-2n)
	\end{aligned}
	\]
	are  mutual inverse $\CA_{2r}$-isomorphisms.
\end{lem}

\subsubsection{Functoriality}
We now consider  properties of $B_{2r}^0(m-2n)$ and $B_r^r(m-2n)$ as $\CA_{2r}$-modules under the functor $F:\cB(m-2n)\longrightarrow \cT_{\OSp(V)}(V)$.
\begin{lem}\label{lem:hom-iso-1}
The following  diagrams are commutative:
	\begin{equation} \label{CommDiagram}
	\xymatrix{
		B_{2r}^0(m-2n) \ar[r]^{{\mathbb U}_r} \ar[d]_{F_{2r}^0} &  B^r_r(m-2n)  \ar[d]^{F_r^{r}} \\
		\Hom_{G}(V^{\ot 2r},\C)\ar[r]_{F{\mathbb U}_r}  & \End_{G}(V^{\ot r})}	
	\quad\quad
     \xymatrix{
		B_r^r(m-2n) \ar[r]^{{\mathbb A}_r} \ar[d]_{F_{r}^r} &  B_{2r}^{0}(m-2n)  \ar[d]^{F_{2r}^{0}} \\
		\End_{G}(V^{\ot r}) \ar[r]_{F{\mathbb A}_r}  & \Hom_{G}(V^{\ot 2r},\C). }
		\end{equation}
\end{lem}
\begin{proof}
Since $F$ preserves both composition and tensor product of Brauer diagrams, given any $A\in B_{2r}^0(m-2n)$,  we have
$$
F_r^r(\bU_r(A))=(\id_{V}^{\ot r}\ot F_{2r}^0(A))\circ (\id_{V}^{\ot r}\ot F(\omega_{2r}))\circ (F(U_r)\ot \id_{V}^{\ot r})=F\bU_r(F_{2r}^0(A)).
$$
The commutativity of the second diagram can be verified similarly.
\end{proof}

Since both
$F_{2r}^0: B_{2r}^0(m-2n)\longrightarrow \Hom_G(V^{\otimes 2r},  \C)$ and $F_r^r: B_r^r(m-2n)\longrightarrow \End_G(V^{\otimes r})$  are surjective, by using Lemma \ref{lem:hom-iso-1}, we obtain natural $\CA_{2r}$-actions
\begin{eqnarray}
\begin{aligned}
&\CA_{2r}\otimes \Hom_G(V^{\otimes 2r},  \C) \longrightarrow  \Hom_G(V^{\otimes 2r},  \C), \\
&\CA_{2r}\otimes \End_G(V^{\otimes r}) \longrightarrow  \End_G(V^{\otimes r}),
\end{aligned}
\end{eqnarray}
which are defined as follows. For any $\phi\in \Hom_G(V^{\otimes 2r},  \C)$, there exists $A\in B_{2r}^0(m-2n)$ such that $\phi=F_{2r}^0(A)$. Then we can define
\[
\alpha.\phi:=F_{2r}^0(\alpha\ast A) = F_{2r}^0( A\circ \alpha^\sharp) = F_{2r}^0( A) F_{2r}^{2r}(\alpha^\sharp) .
\]
The action on $\End_G(V^{\otimes r})$ is similarly defined.
It is clear that $F_{2r}^0$ and $F_r^r$ are $\CA_{2r}$-maps. Therefore,  the kernels  of $F_{2r}^0$ and $F_r^r$ are $\CA_{2r}$-submodules of $B_{2r}^0(m-2n)$ and $B_r^r(m-2n)$ respectively.
We also have the following result.
\begin{lem} \label{lem:hom-iso-2} The linear maps
  $$
   \begin{aligned}
    F\bU_r&=(\id_{V}^{\ot r}\ot -)\circ (\id_{V}^{\ot r}\ot F(\omega_{2r}))\circ (F(U_r)\ot \id_{V}^{\ot r}): \Hom_{G}(V^{\ot 2r},\C)\ra \End_{G}(V^{\ot r}),\\
    F\mathbb{A}_r&=F(A_r)\circ (\id_{V}^{\ot r}\ot -)\circ F(\omega_{2r}^{-1}): \End_{G}(V^{\ot r}) \ra \Hom_{G}(V^{\ot 2r},\C)
    \end{aligned}
$$
are mutual inverse $\CA_{2r}$-isomorphisms.
\end{lem}
\begin{proof}
Applying the functor $F$ to \lemref{AUmaps}, we immediately arrive at the lemma.
\end{proof}

It follows from Lemma \ref{lem:hom-iso-1} and the fact that $\bU_r$  and $F\bU_r$ are isomorphisms that
\begin{coro}\label{CorIsoKer}
$\Ker F_{2r}^0$ and $\Ker F_r^r$  are isomorphic as $\CA_{2r}$-modules.
\end{coro}

\begin{rmk}\label{rmk:iso-kernels}
Using slight variations of $\bU_r$ and $\mathbb{A}_r$, one can similarly define $\Sym_{k+l}$-actions on $B_k^l(m-2n)$ and show that $B_k^l(m-2n)\cong B_{k+l}^0(m-2n)$ and $\Ker F_k^l\cong\Ker F_{k+l}^0$ as $\CA_{k+l}$-modules
for all $k$ and $l$.
\end{rmk}

\subsection{$\Ker F_{2r}^0$ and $\Ker F_r^r$ as $\Sym_{2r}$-modules}\label{secmindeg}
The $\Sym_{2r}$-actions on $B_{2r}^0(m-2n)$ and $B_r^r(m-2n)$ will enable us to undertsand $\Ker F_{2r}^0$ and $\Ker F_r^r$.
\subsubsection{Basic facts on $\Sym_k$}  \label{sect:sym}
Fix a nonnegative integer $k$.
Denote a partition $\lambda$ of $k$ ($\la\vdash k$) by $\lambda=(1^{m_1}2^{m_2}\dots)$.  Let $\Std(\la)$ be the set of standard $\la$-tableaux, and let
$\ft^{\la}$ (resp. $\ft_{\la}$) be the standard $\la$-tableau with $1,2,\dots,r$ appearing in  order from left to right (resp. top to bottom) along successive rows (resp. columns).

We will always consider the right action of $\Sym_k$ on the set of $\lambda$-tableaux.
Then for any $\ft\in \Std(\la)$ there exists $d(\ft)\in \Sym_k$ such that $\ft=\ft^{\la}d(\ft)$, where $d(\ft)$ acts on $\ft^{\la}$ by permuting its entries. In particular, there exists $w_{\la}:=d(\ft_{\la})\in \Sym_k$ such that $\ft^{\la}w_{\la}=\ft_{\la}$.

Given a $\la$-tableau $\ft$,  we denote by $R(\ft)$ and $C(\ft)$ respectively the row  and column stabilisers in $\Sym_k$. Define
\begin{equation}\label{xyla}
\begin{aligned}
x_\la(\ft)=\sum_{\sigma\in R(\ft)}\sigma,&\quad y_\la(\ft)=\sum_{\sigma\in C(\ft)}\epsilon(\sigma)\sigma, \quad \fc_\la(\ft):=x_\la(\ft) y_\la(\ft),
\end{aligned}
\end{equation}
where  $\epsilon$ is the sign character of $\Sym_k$, and $\fc_\la(\ft)$ is the Young symmetriser on $\ft$ of shape $\la$. For $\ft=\ft^\lambda$, we will denote these elements by
$x_{\la}$, $y_{\la}$ and $\fc_{\la}=x_{\la}y_{\la}$ respectively.
In particular,
\[
x_{(k)}=\sum_{\sigma\in \Sym_k}\sigma, \quad
y_{(1^k)}=\sum_{\sigma\in \Sym_k}\epsilon(\sigma)\sigma.
\]

Let $\CA_k:=\C\Sym_k$, then the left Specht module $S^{\la}=\CA_k\fc_{\la}$ is simple and has the standard basis  $\{d(\ft)\fc_{\la}\mid\ft\in\Std(\la)\}$. Similarly, the set $\{\fc_{\la}d(\ft) \mid\ft\in\Std(\la)\}$ forms a basis for the simple  right Specht module $\widetilde{S^{\la}}=\fc_{\la}\CA_k$ (see, e.g., \cite[6.3c, 6.3e]{EGS}).
We have the following dimension formula (see \cite[\S 4.1]{FH})
\begin{equation}\label{hooklen}
\dim \widetilde{S^{\la}}=\dim  S^{\la}=\dfrac{k!}{h_\lambda}, \quad \text{with \ } h_\lambda={\prod_{i,j} h_{ij}^{\la}},
\end{equation}
where $h_{ij}^{\la}$ is the hook length of the $(i,j)$-box in the Young diagram of shape $\la$.

\begin{rmk}\label{rmkYoungSym}
	For any $\la\vdash k$ and $\ft\in \Std(\la)$, one has $\fc_{\la}(\ft)^2=h_{\la}\fc_{\la}(\ft)$.  The elements $h_{\la}^{-1}\fc_{\la}(\ft)$ are pairwise orthogonal primitive idempotents in $\CA_k$, and the canonical resolution of the identity into sum of the minimal central idempotents in $\CA_k$ is 
	\begin{equation}\label{eqidenres}
	1_{\CA_k}=\sum_{\la\vdash k}\fP_{\la} \quad\text{with}\quad \fP_{\la}=\sum_{\ft\in \Tab (\la)} h_{\la}^{-2} \fc_{\la}(\ft).
	\end{equation}
	where $\Tab (\la)$ denotes the set of all $\la$-tableaux,  see, e.g.,  \cite[Theorem 3.1.24]{J}.
\end{rmk}

\subsubsection{$\Ker F_{2r}^0$ as $\Sym_{2r}$-module}

Now let $k=2r$.
For any partition $\mu=(\mu_1,\mu_2,\dots,\mu_s)$ of $r$,  we denote by $2\mu$ the partition $(2\mu_1,2\mu_2,\dots,2\mu_s)\vdash 2r$, and
call it an \textit{even partition} of $2r$. We use $\mathcal{P}_r$ to denote the set of partitions of $r$ and define $2\mathcal{P}_r:=\{2\mu\;|\;\mu \in \mathcal{P}_r\}$.

Using the notation introduced, we can now give a precise description of the two-sided ideal $I(m, n)$ of $\CA_{2r}$ in \thmref{thm:sft}, which is formulated as 
\begin{eqnarray}\label{eq:I-def}\label{lem:ideal}
I(m,n)=\bigoplus_{2r \dashv \la\supseteq ((2n+1)^{m+1})} \CA_{2r}\fc_{\la}\CA_{2r}.
\end{eqnarray}

Let $K_r\subset\Sym_{2r}$ be the stabiliser of $\hat{A}^{ r}$ as depicted in \figref{Fig3}, i.e., $\hat{A}^{ r}\circ \xi=\hat{A}^{ r}$ for any $\xi\in K_r$.  Then $K_r=\mathbb{Z}_2^r \rtimes \Sym_r$.

\begin{lem}\label{LemIsoSym} As left $\CA_{2r}$-modules,
	$B_{2r}^0(m-2n)\cong \Ind_{\C K_r}^{\CA_{2r}}1_{\C K_r}\cong \bigoplus_{\la \in \tcp} S^{\la}$, where $1_{\C K_r}$ denotes the trivial representation of $\C K_r$.
\end{lem}

\begin{proof}
	Note that there is a $\C$-linear isomorphism
	$$\psi: \Ind_{\C K_r}^{\CA_{2r}}1_{\C K_r} \ra B_{2r}^{0}(m-2n)$$
	such that $\psi(\sigma\ot 1_{\C K_r})=\sigma \ast \hat{A}^{ r}=\hat{A}^{ r}\circ \sigma^{-1}$ for all $\sigma\in \Sym_{2r}$.  While the surjectivity of the map is clear, injectivity can be seen by noting that $\Ind_{\C K_r}^{\CA_{2r}}1_{\C K_r}$ has a basis  of the form
	$\{\sigma\ot 1_{\C K_r}\}$ with $\sigma$'s being representatives of distinct
	left cosets of $K_r$ in $\Sym_{2r}$. Thus $\psi$ is indeed an isomorphism. This proves the first isomorphism.
	
	The second isomorphism is a known fact from the representation theory of the symmetric group, see, e.g., \cite[Chapter \uppercase\expandafter{\romannumeral7}, (2.4)]{M}.
\end{proof}

Similarly, we can prove that 
\begin{lem}\label{LemIsoKer}
	As left $\CA_{2r}$-modules, $\Ker F_{2r}^0\cong I(m,n)\ot_{\C K_r}1_{\C K_r}$.
\end{lem}

The following notation will be used throughout the paper: 
\[
r_c:=(m+1)(n+1), \quad \lambda_c:=\left((2n+2)^{m+1}\right)\vdash 2r_c.
\]

\begin{thm} \label{LemKer} 
	As a left $\CA_{2r}$-submodule of $B_{2r}^0(m-2n)$, $\Ker F_{2r}^0$ admits a multiplicity-free decomposition
	$$\Ker F_{2r}^0 \cong \bigoplus_{\tcp\ni\la \supseteq \la_c} S^{\la}\quad \text{and} \quad \dim \Ker F_{2r}^0=\sum_{\tcp\ni\la\supseteq \la_c} \frac{(2r)!}{h_{\la}}.$$ Furthermore,  $\Ker F_{2r}^0\ne 0$ if and only if $r\ge r_c$.
\end{thm}
\begin{proof}
	The dimension formula is an immediate consequence of the first assertion, which we now prove. By \lemref{LemIsoKer}, we have
	\[
	\begin{aligned}
	\Ker F_{2r}^0&\cong I(m,n)\ot_{\C K_r}1_{\C K_r}=\bigoplus_{2r \dashv \la\supseteq ((2n+1)^{m+1})} \CA_{2r}c_{\la}\CA_{2r} \ot_{\C K_r}1_{\C K_r}\\
	&= \bigoplus_{2r \dashv \la\supseteq ((2n+1)^{m+1})} \CA_{2r}c_{\la}\Ind_{\C K_r}^{\CA_{2r}}1_{\C K_r}.
	\end{aligned}
	\]
	Using \lemref{LemIsoSym} and  orthogonality of elements $\fc_{\la}(\ft)$, we can express the far right hand side of the above equation as
	\[
	\bigoplus_{2r \dashv \la\supseteq ((2n+1)^{m+1})}  \bigoplus_{\mu \in \tcp} \CA_{2r}c_{\la} S^{\mu}\quad \cong
	\bigoplus_{\tcp\ni\la \supseteq \la_c} S^{\la}.
	\]
	This proves the claim on the decomposition of $\Ker F_{2r}^0$. The dimension formula follows from this decomposition and \eqref{hooklen}.
	
	Now  $\la\in \tcp$  and $\la\supseteq \la_c$ implies that $r\geq r_c$. If  $r\ge r_c$, there always exists
	$\la\in \tcp$ such that $\la\supseteq \la_c$. In particular, when
	$r=r_c$, the only such even partition is $\la_c$ itself, and we have $\Ker F_{r_c}^{r_c}\cong\Ker F_{2r_c}^0\cong S^{\la_c}$.
\end{proof}

Since $\Ker F_{r}^r\cong\Ker F_{2r}^0$ as $\CA_{2r}$-modules by \corref{CorIsoKer}, we have the following result.
\begin{thm}\label{thmKeriso}
	The kernel of $F_{r}^r$ is nonzero if and only if $r\ge r_c$, and in that case
	\[
	\begin{aligned}
	\Ker F_r^r&\cong \bigoplus_{\tcp\ni \la \supseteq \la_c} S^{\la}
	\text{ \ and \ }
	\dim \Ker F_{r}^{r}&=\sum_{\tcp\ni\la\supseteq \la_c} \frac{(2r)!}{h_{\la}}.
	\end{aligned}
	\]
\end{thm}

\begin{coro}\label{CorInj}
	The map $F_r^r: B_r^r(m-2n) \longrightarrow \End_G(V^{\ot r})$
	is an algebra isomorphism if and only if $r<r_c$.
\end{coro}
\begin{proof}
	This follows from \thmref{thm:fft-cat} and \thmref{thmKeriso}.
\end{proof}

Elements of $\Ker F_{r_c}^{r_c}$ have the following annihilation property. 

\begin{prop}\label{AnnThm}
	Any $\Psi\in \Ker F_{r_c}^{r_c}$ satisfies $e_i\Psi=\Psi e_i=0$ for any generator $e_i\in B_{r_c}(m-2n)$.
\end{prop}
\begin{proof}
	In the composite Brauer diagram $e_i\Psi$, we can find a subdiagram which is exactly obtained from $\Psi$ by connecting top vertex $i$ with the top vertex $i+1$. Denoting this subdiagram by $\Psi_{i,i+1}$, we have $\Psi_{i,i+1}\in \Ker F_{r_c}^{r_c-2}\subset B_{r_c}^{r_c-2}(m-2n)$. By Remark \ref{rmk:iso-kernels}, we have $B_{r_c}^{r_c-2}(m-2n)\cong B_{2r_c-2}^0(m-2n)$ and  $\Ker F_{r_c}^{r_c-2}\cong \Ker F_{2r_c-2}^0$. However, $ \Ker F_{2r_c-2}^0=0$ by  \thmref{LemKer}. Therefore,  $\Psi_{i,i+1}=0$ and hence $e_i\Psi=0$. Another case $\Psi e_i=0$ can be proved similarly.
\end{proof}

\begin{coro}
	Let $D$ be any Brauer diagram in $B_{r_c}(m-2n)$ which has horizontal edges, i.e., $D$ is in the 2-sided ideal of $B_{r_c}(m-2n)$ generated by $e_1$. Then $D\Psi=\Psi D=0$ for any $\Psi\in \Ker F_{r_c}^{r_c}$. 
\end{coro}

\subsection{A basis for $\Ker F_r^r$ in minimal degree}\label{BasisKer}
We shall construct a basis for  $\Ker F_{r_c}^{r_c}$.

\subsubsection{The Lehrer-Zhang element}
Given any Brauer diagram $A\in B_{2r}^0(m-2n)$, there exists
$\sigma\in \Sym_{2r}$ such that $A=\hat{A}^{ r}\circ \sigma=\sigma^{-1}\ast \hat{A}^{ r} $, and hence we have the following relation in $B_r^r(m-2n)$,
\[
\bU_r(A)=\sigma^{-1}\ast \bU_r(\hat{A}^{ r})=\sigma^{-1}\ast I_r,
\]
Therefore $I_r$ generates $B_r^r(m-2n)$ as an $\CA_{2r}$-module.

The above simple fact particularly enables us to gain a conceptual understanding  of the element of the Brauer algebra introduced in \cite[Section 5.2]{LZ4}. For each $i$ such that $0\leq i\leq [\frac{r}{2}]$, let $E(i)\in B_r^r(m-2n)$ be defined by $E(i)=\prod_{j=1}^ie_{2j-1}$, where $E(0)=I_r$ by convention. We define
$\Xi_r(i)=x_{(r)}E(i)x_{(r)},$
which is represented pictorially as follows:
	
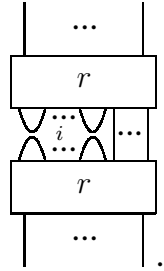
\begin{center}	
		\begin{picture}(100, 100)(-20,-40)
		\put(5, 40){\line(0, 1){20}}

		\put(23, 50){...}
		\put(50, 40){\line(0, 1){20}}

		\put(0, 20){\line(1, 0){55}}
		\put(0, 20){\line(0, 1){20}}
		\put(55, 20){\line(0, 1){20}}
		\put(0, 40){\line(1, 0){55}}
		\put(25, 28){${r}$}
		
		\put(39, 20){\line(0, -1){20}}
		\put(40, 10){...}
		\put(52, 20){\line(0, -1){20}}
		
		\qbezier(3, 20)(8, 2)(13, 20)
		\put(15, 16){...}
		\qbezier(26, 20)(31, 2)(36, 20)
		\put(17, 8){\tiny$i$}
		\qbezier(3, 0)(8, 18)(13, 0)
		\put(15, 4){...}
		\qbezier(26, 0)(31, 18)(36, 0)

		\put(0, 0){\line(1, 0){55}}
		\put(0, 0){\line(0, -1){20}}
		\put(55, 0){\line(0, -1){20}}
		\put(0, -20){\line(1, 0){55}}
		\put(25, -12){${r}$}
		\put(5, -20){\line(0, -1){20}}
		\put(23, -30){...}
		\put(50, -20){\line(0, -1){20}}
		\put(55, -40){.}
		\end{picture}
	\end{center} 
We define the Lehrer-Zhang element by
\begin{equation}\label{eqLZelmt}
\Phi_{LZ}(r):= \sum_{i=0}^{[\frac{r}{2}]}c_i\Xi_r(i),\quad
c_i =((2^i i! )^2 (r-2i)!)^{-1}.
\end{equation}

\begin{rmk}
 The original Lehrer-Zhang element defined in \cite[Section 5.2]{LZ4} is the element $\Phi_{LZ}(n+1)$ in the Brauer algebra $B_{n+1}^{n+1}(-2n)$, i.e., taking $r=n+1$ and $m=0$ in our case. This element turns out to be an idempotent and remarkably generates the kernel of the surjective algebra homomorphism $F_r^r: B_{r}^r(-2n)\rightarrow \End_{\Sp(V)}(V^{\ot r})$ when $r\geq n+1$.
\end{rmk}

We obtain a new description of Lehrer-Zhang element using group actions 

\begin{lem}\label{LemSymAct}
Let $x_{(2r)}=\sum_{\sigma\in \Sym_{2r}}\sigma\in \CA_{2r}$ be the Young symmetriser associated with the
one row Young diagram of $2r$ boxes. Then
\[
\Phi_{LZ}(r)=\sum_D D=\frac{x_{(2r)} \ast I_r}{2^rr!},
\]
where the sum is over all the  Brauer diagrams in $B_r^r(m-2n)$.
\end{lem}
\begin{proof}
The first equality follows from \cite[Lemma 5.2]{LZ4}, which is proved for the case $r=n+1$ and $m=0$. However, it can be easily seen that the proof is actually independent of $r$ and $m$. So it remains to prove the second equality.

Recall that  the subgroup $K_r=\mathbb{Z}_2^r \rtimes \Sym_r$ of $\Sym_{2r}$ is the stabiliser of $\hat{A}^r$. Denote by $R$  the set of representatives of the right  coset $K_r$ in $\Sym_{2r}$, we have
\[
\begin{aligned}
x_{(2r)}\ast I_r &=\bU_r(\hat{A}^{ r}\circ x_{(2r)})= \bU_r(\hat{A}^{ r}\circ \sum_{\sigma\in \Sym_{2r}}\sigma)\\
&=\sum_{\tau \in R}\bU_r(\hat{A}^{ r} \circ \sum_{\xi\in K_r}\xi\tau )
=|K_r|\sum_{\tau\in R}\bU_r(\hat{A}^{ r}\circ \tau).
\end{aligned}
\]
Since $\bU_r$ is an $\CA_{2r}$-isomorphism and $|R|=(2r-1)\times(2r-3)\cdots 3\times1=\dim B_r^r(m-2n)$, we obtain $x_{(2r)}\ast  I_r =2^rr!\sum_D D$ as required.
\end{proof}

\begin{rmk}
Both \cite[Lemma 5.2]{LZ4}  and \lemref{LemSymAct} are valid in $B_r^r(\delta)$ for arbitrary $\delta$.
\end{rmk}

\subsubsection{Labelling of Brauer diagrams} \label{sect:lablelling}
To keep track of the changes of Brauer diagrams under the symmetric group actions, it is useful to label their vertices and regard the symmetric group as the permutation group on the set of labels.   This will be useful in Section \ref{sect:basis-min}.

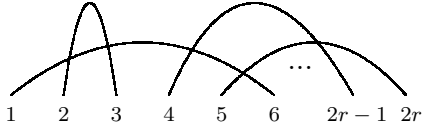
\begin{figure}[h]
	\begin{center}
		\begin{picture}(190, 40)(0,0)
		\qbezier(30, 10)(80, 50)(130, 10)
		\qbezier(50, 10)(60, 80)(70, 10)
		
		\put(135, 20){$...$}
		
		\qbezier(90, 10)(120, 80)(160, 10)
		
		\qbezier(110, 10)(145, 50)(180, 10)
		\put(28,0){\tiny$1$}
		\put(48,0){\tiny $2$}
		\put(68,0){\tiny $3$}
		\put(88,0){\tiny $4$}
		\put(108,0){\tiny $5$}
		\put(128,0){\tiny $6$}
		\put(150,0){\tiny $2r-1$}
		\put(178,0){\tiny $2r$}
		\end{picture}
	\end{center}
	\caption{A diagram $A \in B_{2r}^0(m-2n)$}
	\label{Arcs}
\end{figure}
For any Brauer diagram $A$ in $B_{2r}^0(m-2n)$, we label its vertices by $1, 2, \dots, 2r$ from left to right as indicated in Figure \ref{Arcs}. Since $\bU_r$ is bijective, any Brauer diagram in $B_r^r(m-2n)$ is of the form $\bU_r(A)$ for a $(2r, 0)$-diagram $A$. Thus the labelling of the vertices of $A$ induces a labelling for
the vertices of $\bU_r(A)$ such that the vertices in the top row are numbered by the odd integres $1, 3, \dots 2r-1$, and those in the bottom row by the even integers
$2, 4, \dots, 2r$.

Given the elementary transposition $s_i=(i,i+1)\in \Sym_{2r}$ and a Brauer diagram $D\in B_{r}^r(m-2n)$, it follows from definition \eqref{eq:action-2} that the left action of $s_i$ on $D$ is to switch the positions of endpoints in $D$ labelled by $i$ and $i+1$, viz.:
	
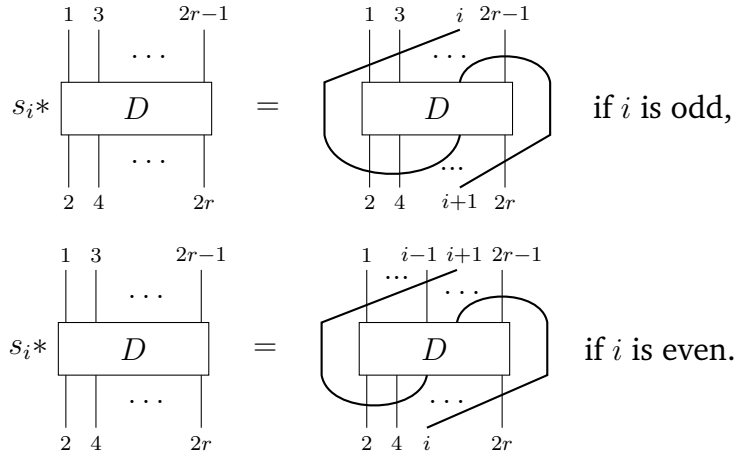
\begin{center}
		\begin{tikzpicture}
		\node at (-0.4,0.35){$s_i\ast$};		
		\draw (0,0) rectangle (2,0.7);
		\node at (1,0.35){$D$};
		
		\draw (0.1,0.7) -- (0.1,1.4);
		\draw (0.5,0.7) -- (0.5,1.4);	
		\draw (1.9,0.7) -- (1.9,1.4);
		\node at (1.2,1.05){$\dots$};
		
		\node at (0.1,1.6){$\scriptstyle{1}$};  	
		\node at (0.5,1.6){$\scriptstyle{3}$};
		\node at (1.9,1.6){$\scriptstyle{2r-1}$};
		
		\draw (0.1,0) -- (0.1,-0.7);
		\draw (0.5,0) -- (0.5,-0.7);	
		\draw (1.9,0) -- (1.9,-0.7);
		\node at (1.2,-0.35){$\dots$};
		
		\node at (0.1,-0.9){$\scriptstyle{2}$};  	
		\node at (0.5,-0.9){$\scriptstyle{4}$};
		\node at (1.9,-0.9){$\scriptstyle{2r}$};
		
		\node at (2.75,0.35) {$=$};
		
		\draw (4,0) rectangle (6,0.7);
		\node at (5,0.35){$D$};
		\draw (4.1,0.7) -- (4.1,1.4);
		\draw (4.5,0.7) -- (4.5,1.4);	
		\draw (5.9,0.7) -- (5.9,1.4);
		\node at (5.2,1.05){$\dots$};
		
		\node at (4.1,1.6){$\scriptstyle{1}$};  	
		\node at (4.5,1.6){$\scriptstyle{3}$};
		\node at (5.3,1.6){$\scriptstyle{i}$};
		\node at (5.9,1.6){$\scriptstyle{2r-1}$};
		
		\draw (4.1,0) -- (4.1,-0.7);
		\draw (4.5,0) -- (4.5,-0.7);	
		\draw (5.9,0) -- (5.9,-0.7);
		\node at (5.2,-0.45){$...$};
		
		\node at (4.1,-0.9){$\scriptstyle{2}$};  	
		\node at (4.5,-0.9){$\scriptstyle{4}$};
		\node at (5.3,-0.9){$\scriptstyle{i+1}$};
		\node at (5.9,-0.9){$\scriptstyle{2r}$};		
		
		\draw[-,thick] (5.3,0.7) to [out=80,in=100] (6.5,0.7)
		to  (6.5,0) to (5.3,-0.7);
		\draw[-,thick] (5.3,0) to [out=260,in=-80] (3.5,0)
		to (3.5,0.7) to (5.3,1.4);
		
		\node at (8.0,0.35){\text{if $i$ is odd,}};	   	  	 	  	  	
		\end{tikzpicture}
	\end{center}
	
	\begin{center}
		\begin{tikzpicture}
		\node at (-0.4,0.35){$s_i\ast$};		
		\draw (0,0) rectangle (2,0.7);
		\node at (1,0.35){$D$};
		
		\draw (0.1,0.7) -- (0.1,1.4);
		\draw (0.5,0.7) -- (0.5,1.4);	
		\draw (1.9,0.7) -- (1.9,1.4);
		\node at (1.2,1.05){$\dots$};
		
		\node at (0.1,1.6){$\scriptstyle{1}$};  	
		\node at (0.5,1.6){$\scriptstyle{3}$};
		\node at (1.9,1.6){$\scriptstyle{2r-1}$};
		
		\draw (0.1,0) -- (0.1,-0.7);
		\draw (0.5,0) -- (0.5,-0.7);	
		\draw (1.9,0) -- (1.9,-0.7);
		\node at (1.2,-0.35){$\dots$};
		
		\node at (0.1,-0.9){$\scriptstyle{2}$};  	
		\node at (0.5,-0.9){$\scriptstyle{4}$};
		\node at (1.9,-0.9){$\scriptstyle{2r}$};
		
		\node at (2.75,0.35) {$=$};
		
		\draw (4,0) rectangle (6,0.7);
		\node at (5,0.35){$D$};
		\draw (4.1,0.7) -- (4.1,1.4);
		\draw (4.9,0.7) -- (4.9,1.4);	
		\draw (5.9,0.7) -- (5.9,1.4);
		\node at (4.5,1.3){$...$};
		
		\node at (4.1,1.6){$\scriptstyle{1}$};  	
		\node at (4.8,1.6){$\scriptstyle{i-1}$};
		\node at (5.4,1.6){$\scriptstyle{i+1}$};
		\node at (6.1,1.6){$\scriptstyle{2r-1}$};
		\node at (5.4,1.1){$\dots$};
		
		\draw (4.1,0) -- (4.1,-0.7);
		\draw (4.5,0) -- (4.5,-0.7);	
		\draw (5.9,0) -- (5.9,-0.7);
		\node at (5.2,-0.35){$\dots$};
		
		\node at (4.1,-0.9){$\scriptstyle{2}$};  	
		\node at (4.5,-0.9){$\scriptstyle{4}$};
		\node at (4.9,-0.9){$\scriptstyle{i}$};
		\node at (5.9,-0.9){$\scriptstyle{2r}$};
		\node at (8.0,0.35){\text{if $i$ is even.}};
		
		\draw[-,thick] (5.3,0.7) to [out=80,in=100] (6.5,0.7)
		to  (6.5,0) to (4.9,-0.7);
		\draw[-,thick] (4.9,0) to [out=260,in=-80] (3.5,0)
		to (3.5,0.7) to (5.3,1.4);    		 		
		\end{tikzpicture}
	\end{center}

\begin{exam}\label{examDact}
	{\rm 
	Let $r=3$ and $D$ be the diagram as shown below. Then we have $(12)*D$ and $(45)*D$ depicted as follows:
	\begin{center}
			\begin{picture}(360,60)(-10,-10)
			\put(-30,12){$D=$}
			\qbezier(0,30)(30,10)(60,30)
			\qbezier(30,30)(15,15)(0,0)
			\qbezier(30,0)(45,30)(60,0)	
			
			\put(-2,32){\tiny$1$}
			\put(-2,-8){\tiny$2$}
			\put(28,32){\tiny$3$}
			\put(28,-8){\tiny$4$}
			\put(58,32){\tiny$5$}
			\put(58,-8){\tiny$6$}
			\put(66,0)	{,}	
			
			\put(80,12){$(12)\ast D$=}
			
			\qbezier(143,30)(158,10)(173,30)
			\qbezier(203,30)(173,15)(143,0)
			\qbezier(173,0)(188,30)(203,0)	
			
			\put(141,32){\tiny$1$}
			\put(141,-8){\tiny$2$}
			\put(171,32){\tiny$3$}
			\put(171,-8){\tiny$4$}
			\put(201,32){\tiny$5$}
			\put(201,-8){\tiny$6$}			
			\put(210,0)	{,}	
			
        	\put(220,12){$(45)\ast D=$}
	
       	\qbezier(283,30)(298,15)(313,0)
	    \qbezier(313,30)(298,15)(283,0)
	    \qbezier(343,0)(343,15)(343,30)	
	
	    \put(281,32){\tiny$1$}
	    \put(281,-8){\tiny$2$}
	    \put(311,32){\tiny$3$}
	    \put(311,-8){\tiny$4$}
	    \put(341,32){\tiny$5$}
	    \put(341,-8){\tiny$6$}			
	    \put(350,0)	{.}					
			
			\end{picture}		
	\end{center}
   }
\end{exam}

\begin{rmk}\label{rmkoddevenact}
  We deduce from  the above diagrammatic definition of group action that 
   $$(\sigma_1,\sigma_2)\ast D=\sigma_1\circ D\circ \sigma_2^{-1}$$
   for any  $(\sigma_1,\sigma_2)\in\Sym\{1,3,\dots2r-1\}\times \Sym\{2,4,\dots,2r\}$ and $D\in B_r^r(m-2n)$.
\end{rmk}

\subsubsection{A basis for the kernel}\label{sect:basis-min}
If $r=r_c$,  then $\la_c:=((2n+2)^{m+1})$ is the only partition of $2r$ which contains $((2n+2)^{m+1})$.  We have the following key lemma.

 \begin{lem}\label{LemyX}
 The element $\hat{A}^{ r_c}\circ \fc_{\la_c}$ in $\Ker F_{2r_c}^0$ is nonzero.
\end{lem}
\begin{proof}
The lemma holds if and only if the following element of $\Ker F_{r_c}^{r_c}$ is nonzero:
\begin{eqnarray}\label{eq:Phi-def}
\Phi:=\bU_{r_c}(\hat{A}^{ r_c}\circ \fc_{\la_c}).
\end{eqnarray}
By  \lemref{AUmaps},  we obtain
$\Phi=(\fc_{\la_c} )^{\sharp}\ast \bU_{r_c}(\hat{A}^{ r_c})=y_{\la_c}x_{\la_c}\ast I_{r_c}.$

Let $X_{\la_c}:=x_{\la_c}\ast I_{r_c}$. We shall first analyse the action of $x_{\la_c}$ on $I_r$. Recall that both $x_{\la_c}$ and $y_{\la_c}$ are defined on $\ft^{\la_c}$, which is filled with $1,2,\dots,2r_c$ from left to right  along successive rows. Then the $i$-th row $R_i$ ($ 1\leq i\leq m+1$) of $\ft^{\la_c}$ is 
\[ R_i=\{ (2n+2)(i-1)+1, (2n+2)(i-1)+2,\dots, (2n+2)(i-1)+2n+2\}. \]
Corresponding to each row, we define the row symmetriser $x(R_i):=\sum_{\sigma\in \Sym\{R_i\}} \sigma$ and hence $x_{\la_c}=\prod_{i=1}^{m+1}x(R_i)$. We now invoke the labelling of vertices of Brauer diagrams discussed in Section \ref{sect:lablelling}. Note that each $x(R_i)$ only acts on the vertices of $I_r$ labelled by the elements of $R_i$ and leaves other vertices unchanged. Using \lemref{LemSymAct}, we have 
$$x(R_i)\ast I_r= 2^{n+1} (n+1)!I^{\ot (i-1)(n+1)}\ot \Phi_{LZ}(n+1) \ot I^{\ot r_c-i(n+1)}.$$ 
Recalling \eqref{eqLZelmt}, we obtain
\begin{equation}\label{EqnXla}
\begin{aligned}
  X_{\la_c}&=\prod_{i=1}^{m+1}x(R_i)\ast I_{r_c}\\
  &=(2^{n+1} (n+1)!)^{m+1} \Phi_{LZ}(n+1)\ot \Phi_{LZ}(n+1) \ot \dots\ot \Phi_{LZ}(n+1)\\ 
  &=\sum\limits_{0\leq i_1,i_2,\dots,i_{m+1}\leq [\frac{n+1}{2}]}c_{i_1,i_2,\dots,i_{m+1}} \Xi_{n+1}(i_1)\ot \Xi_{n+1}(i_2)\ot\cdots\ot \Xi_{n+1}(i_{m+1}),
\end{aligned}
\end{equation}
where $ c_{i_1,i_2,\dots,i_{m+1}}=(2^{n+1} (n+1)!)^{m+1}\prod_{k=1}^{m+1}c_{i_k}$ with $c_{i_k}=((2^{i_k} (i_k)! )^2 (n+1-2i_k)!)^{-1}$. Pictorially the element $\Xi_{n+1}(i_1)\ot \Xi_{n+1}(i_2)\ot\cdots\ot \Xi_{n+1}(i_{m+1})$ can be represented by
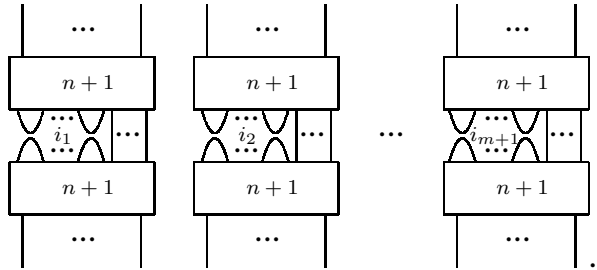
\begin{center}
\begin{picture}(200, 100)(10,-40)
	\put(5, 40){\line(0, 1){20}}
	\put(23, 50){...}
	\put(50, 40){\line(0, 1){20}}	
	
	\put(0, 20){\line(1, 0){55}}
	\put(0, 20){\line(0, 1){20}}
	\put(55, 20){\line(0, 1){20}}
	\put(0, 40){\line(1, 0){55}}
	\put(20, 28){\tiny ${n+1}$}
	
	\qbezier(3, 20)(8, 2)(13, 20)
	\put(15, 16){...}
	\qbezier(26, 20)(31, 2)(36, 20)
	\put(17, 8){\tiny$i_1$}
	\qbezier(3, 0)(8, 18)(13, 0)
	\put(15, 4){...}
	\qbezier(26, 0)(31, 18)(36, 0)
	
	\put(39, 20){\line(0, -1){20}}
	\put(40, 10){...}
	\put(52, 20){\line(0, -1){20}}
	
	\put(0, 0){\line(1, 0){55}}
	\put(0, 0){\line(0, -1){20}}
	\put(55, 0){\line(0, -1){20}}
	\put(0, -20){\line(1, 0){55}}
	\put(20, -12){\tiny ${n+1}$}
	\put(5, -20){\line(0, -1){20}}
	\put(23, -30){...}
	\put(50, -20){\line(0, -1){20}}
	
	\put(75, 40){\line(0, 1){20}}
	\put(93, 50){...}
	\put(120, 40){\line(0, 1){20}}
	
	\put(70, 20){\line(1, 0){55}}
	\put(70, 20){\line(0, 1){20}}
	\put(125, 20){\line(0, 1){20}}
	\put(70, 40){\line(1, 0){55}}
	\put(90, 28){\tiny ${n+1}$}
	
	\qbezier(73, 20)(78, 2)(83, 20)
	\put(85, 16){...}
	\qbezier(96, 20)(101, 2)(106, 20)
	\put(87, 8){\tiny$i_2$}
	\qbezier(73, 0)(78, 18)(83, 0)
	\put(85, 4){...}
	\qbezier(96, 0)(101, 18)(106, 0)
	
	\put(109, 20){\line(0, -1){20}}
	\put(110, 10){...}
	\put(122, 20){\line(0, -1){20}}
		
	\put(70, 0){\line(1, 0){55}}
	\put(70, 0){\line(0, -1){20}}
	\put(125, 0){\line(0, -1){20}}
	\put(70, -20){\line(1, 0){55}}
	\put(90, -12){\tiny ${n+1}$}
	\put(75, -20){\line(0, -1){20}}
	\put(93, -30){...}
	\put(120, -20){\line(0, -1){20}}   	
	
	\put(140,10){...}
	
	\put(170, 40){\line(0, 1){20}}
	\put(188, 50){...}
	\put(215, 40){\line(0, 1){20}}
	
	\put(165, 20){\line(1, 0){55}}
	\put(165, 20){\line(0, 1){20}}
	\put(220, 20){\line(0, 1){20}}
	\put(165, 40){\line(1, 0){55}}
	\put(185, 28){\tiny ${n+1}$}	
	
	\qbezier(167, 20)(172, 2)(177, 20)
	\put(180, 16){...}
	\qbezier(191, 20)(196, 2)(201, 20)
	\put(175, 8){\tiny$i_{m+1}$}
	\qbezier(167, 0)(172, 18)(177, 0)
	\put(180, 4){...}
	\qbezier(191, 0)(196, 18)(201, 0)
	
	\put(204, 20){\line(0, -1){20}}
	\put(205, 10){...}
	\put(217, 20){\line(0, -1){20}}
		
	\put(165, 0){\line(1, 0){55}}
	\put(165, 0){\line(0, -1){20}}
	\put(220, 0){\line(0, -1){20}}
	\put(165, -20){\line(1, 0){55}}
	\put(185, -12){\tiny ${n+1}$}
	\put(170, -20){\line(0, -1){20}}
	\put(188, -30){...}
	\put(215, -20){\line(0, -1){20}}	
	
	\put(220,-40){.}
	\end{picture}
\end{center}

We now turn to the action of $y_{\la_c}$ on $X_{\la_c}$.
We note in particular that the odd columns of $\ft^{\la_c}$ are filled with odd integers, while the even columns are filled with even integers. Denote by $C_{i}$ the $i$-th column of $\mathfrak{t}^{\la_c}$. Then $\Sym\{C_i\}$ is the group of permutations of odd (resp. even) numbers if $i$ is odd (resp. even). Let $y(C_i):=\sum_{\sigma\in \Sym\{C_i\}} \epsilon(\sigma)\sigma$, then
\begin{equation*}
    y_{\la_c}=\prod_{i=1}^{2n+2}y(C_i)=\prod_{k=1}^{n+1}y(C_{2k-1}) \prod_{k=1}^{n+1}y (C_{2k}).
\end{equation*}
Note that  the vertices in the top row and bottom row of $X_{\la_c}$ are labelled by odd and even integers respectively. Thus,  by \rmkref{rmkoddevenact} we have
\begin{equation*} \label{EqnXy}
   \Phi= y_{\la_c}\ast X_{\la_c}=\prod_{k=1}^{n+1}y(C_{2k-1})\circ X_{\la_c} \circ \prod_{k=1}^{n+1}y (C_{2k}).
\end{equation*}
Note that the right hand side is now a product of elements of $\CA_{r_c}$.

 Let $\mu_c=((n+1)^{m+1})$, and let $x_{\mu_c}$ and $y_{\mu_c}$ be the elements in $\CA_{r_c}$ defined with respect to $\mathfrak{t}^{\mu_c}$.  Observe that  $\prod_{k=1}^{n+1}y(C_{2k-1})$ and $\prod_{k=1}^{n+1}y(C_{2k})$ are both equal to  $y_{\mu_c}$ when viewed as alternating sums of Brauer  $(r_c,r_c)$-diagrams. Therefore, we  rewrite the above equation as
 \begin{equation}\label{EqnXy3}
  \Phi = y_{\mu_c}\circ  X_{\la_c}\circ y_{\mu_c}.
 \end{equation}
Using \eqnref{EqnXla}, we obtain
\begin{equation}\label{EqnXy2}
   \Phi = \sum_{0\leq i_1,i_2,\dots,i_{m+1}\leq [\frac{n+1}{2}]} c_{i_1,i_2,\dots,i_{m+1}} y_{\mu_c} \bigg(\Xi_{n+1}(i_1)\ot \Xi_{n+1}(i_2)\ot\cdots\ot \Xi_{n+1}(i_{m+1})\bigg)y_{\mu_c}.
 \end{equation}
It can be easily seen from the diagram that  $\Xi_{n+1}(0)\ot \Xi_{n+1}(0)\ot\cdots\ot \Xi_{n+1}(0)=x_{\mu_c}x_{\mu_c}$.
Observe that the leading term $c_{0,0,\dots,0}y_{\mu_c}x_{\mu_c}x_{\mu_c}y_{\mu_c}$ on  the right hand side  of \eqref{EqnXy2} has no horizontal edges, while the other terms all have at least one horizontal edge. Hence there is no cancellation between the leading term and the rest. Moreover, the leading term is nonzero, since $x_{\mu_c }c_{0,0,\dots,0}y_{\mu_c}x_{\mu_c}x_{\mu_c}y_{\mu_c}=((n+1)!)^{m+1}c_{0,0,\dots,0}h_{\mu_c}\fc_{\mu_c}\neq 0.$  Therefore $\Phi$ is nonzero.
\end{proof}

 Let $a_{m,n}=(2^{n+1}(n+1)!)^{m+1}$. It follows from \eqref{EqnXla}  that $a_{m,n}^{-1} X_{\la_c}$  is a linear combination of Brauer diagrams with integral coefficients, thus so is  $a_{m,n}^{-1}\Phi$. For ease of notation, we shall adopt the following normalised forms:
 \begin{equation}\label{eqnorm}
 \hat{\fc}_{\la_c}:=a_{m,n}^{-1} \fc_{\la_c}, \quad \hat{\Phi}:=\bU_{r_c}(\hat{A}^{ r_c}\circ \hat{\fc}_{\la_c})= y_{\mu_c}\circ \big(\Phi_{LZ}(n+1)^{\ot (m+1)}\big)\circ y_{\mu_c}.
 \end{equation}
 We shall give a diagrammatic description of $\hat{\Phi}$ in \lemref{propidemp} below.


We now construct a basis for the kernel $\Ker F_{r_c}^{r_c}$. Recall that $\{\fc_{\la_c}d(\ft)\mid \ft\in \Std(\la_c)\}$ is a basis for the right Specht module $\fc_{\la_c}\CA_{2r_c}$. For any standard tableau $\ft\in \Std(\la_c)$, we define
\begin{equation}\label{eqn:defBasis}
  \Phi_{\ft}:= d(\ft)^{-1}\ast \Phi.
\end{equation}

\begin{thm} \label{thmBasisMin}
 The set
 $
 \Theta_{r_c}:=\{\Phi_{\ft}\mid \ft\in \Std(\la_c)\}
 $
forms a basis for $\Ker F_{r_c}^{r_c}$.
\end{thm}
\begin{proof}
 Consider  $\hat{A}^{ r_c}\circ\fc_{\la_c}$, it is nonzero by  \lemref{LemyX}, thus $\hat{A}^{ r_c}\circ\fc_{\la_c} \CA_{2r_c}$ is a nonzero submodule of $\Ker F_{2r_c}^{0}$ spanned by the set  $\Upsilon:=\{\hat{A}^{ r_c}\circ\fc_{\la_c}d(\ft)\;|\; \ft\in \Std(\la_c)\}$.  By \thmref{LemKer},
$\Ker F_{2r_c}^0$ is isomorphic to the simple module $S^{\la_c}$, hence $\hat{A}^{ r_c}\circ\fc_{\la_c}\CA_{2r_c}=\Ker F_{2r_c}^0$. As the cardinality of $\Upsilon$ is precisely $\dim S^{\la_c}$,  it is a basis of $\Ker F_{2r_c}^{0}$,  and hence $\bU_{r_c}(\Upsilon)$
 is a basis for $\Ker F_{r_c}^{r_c}$.
Using \lemref{AUmaps}, we obtain
   $$
  \bU_{r_c}(\hat{A}^{ r_c}\circ\fc_{\la_c}d(\ft))=d(\ft)^{-1}\ast \bU_{r_c}(\hat{A}^{ r_c}\circ\fc_{\la_c})=\Phi_{\ft}.
  $$
This completes the proof.
\end{proof}

\begin{rmk}
The $\Sym_{2r}$-structure of $\Ker F_r^r$ in the case of classical groups was
investigated by Doty and Hu in \cite{DH,H}, where they constructed integral bases for the Brauer algebra and for the annihilators of tensor spaces.
\end{rmk}

\subsection{Basis for $\Ker F_r^r$ in general case}
For any even partition $\la=(\la_1,\dots,\la_k)\in \tcp$, we define $\la/2:=(\frac{\la_1}{2},\dots,\frac{\la_k}{2})$ and 
$\Phi_{\la}:=\bU_r(\hat{A}^{ r}\circ \fc_{\la})$.

\begin{lem}\label{lemPhila}
	Let $\tcp\ni\la=(\la_1,\dots,\la_k)\supseteq \la_c$. Then 
	\[  \Phi_{\la}=a_{\la} y_{\la/2}\circ \bigg(\Phi_{LZ}(\frac{\la_1}{2})\ot \Phi_{LZ}(\frac{\la_2}{2})\ot\dots\ot \Phi_{LZ}(\frac{\la_k}{2}) \bigg)\circ y_{\la/2},\]
	is nonzero, where $a_{\la}=\prod_{i=1}^k2^{\frac{\la_i}{2}}(\frac{\la_i}{2})!$ and $y_{\la/2}$ is associated with $\ft^{\la/2}$. 
\end{lem}
\begin{proof}
	This can be proved similarly as in \lemref{LemyX}, so we only give a sketch of proof. Recalling  $\fc_{\la}=x_{\la}y_{\la}$ which is associated to $\ft^{\la}$, we have $\Phi_{\la}=(\fc_{\la} )^{\sharp}\ast \bU_{r}(\hat{A}^{ r})=y_{\la}x_{\la}\ast I_{r}$. Now let $X_{\la}=x_{\la}\ast I_{r}$. One can show that $X_{\la}=a_{\la}\Phi_{LZ}(\frac{\la_1}{2})\ot \Phi_{LZ}(\frac{\la_2}{2})\ot\dots\ot \Phi_{LZ}(\frac{\la_k}{2})$ with $a_{\la}$ as given in this lemma. Further, we will have $\Phi_{\la}=y_{\la}\ast X_{\la}=y_{\la/2}\circ X_{\la} \circ y_{\la/2}$ as desired. Here we regard  $y_{\la/2}$ as alternating sum of Brauer diagrams and we have used the fact that the even (resp. odd) columns of $\ft^{\la}$ are filled with even (resp. odd) integers.
\end{proof}

The following proposition identifies explicitly the Specht modules in the kernel.
\begin{prop}\label{lemKer}
	Let $r\geq r_c$. Then we have
	\begin{equation}\label{eqKeriden}
	\Ker F_{2r}^0=\bigoplus_{\tcp\ni\la\supseteq \la_c} \hat{A}^{ r}\circ \fc_{\la}\CA_{2r}, \quad 	\Ker F_{r}^{r}=\bigoplus_{\tcp\ni\la\supseteq \la_c} \bU_r(\hat{A}^{ r}\circ \fc_{\la}\CA_{2r}),
	\end{equation}
	where $\fc_{\la}$ is the Young symmetriser associated with $\ft^{\la}$. 
\end{prop}
\begin{proof}
	For any $\la$ as given,  by \lemref{lemPhila} the element
	$\hat{A}^{ r}\circ \fc_{\la}\in \Ker F_{2r}^0$ is nonzero.  Hence $\hat{A}^{ r}\circ \fc_{\la}\CA_{2r}\cong S^{\la}$ is a nonzero simple submodule in $\Ker F_{2r}^0$. The equation \eqref{eqKeriden} follows from \thmref{LemKer}.
\end{proof}

\begin{thm}\label{thmbasisgeneral}
	The set $\Theta_{r}:=\{d(\ft)^{-1}\ast \Phi_{\la} \mid  \tcp\ni\la\supseteq \la_c, \ft\in \Std(\la) \}$ forms a basis for $\Ker F_r^r$ for any $r\geq r_c$.
\end{thm}
\begin{proof}
	This follows from \lemref{lemPhila} and \propref{lemKer}.
\end{proof}

\section{Generators of $\Ker F_{r}^{r}$ in minimal degree}
  The basis elements of $\Ker F_{r_c}^{r_c}$ given in Theorem \ref{thmBasisMin} obviously form a generating set of $\Ker F_{r_c}^{r_c}$ as a $2$-sided ideal of $B_{r_c}(m-2n)$. However this set is too large, thus is unwieldy to use. In this section, we shall construct a more convenient set of generators for $\Ker F_{r_c}^{r_c}$.

\subsection{Diagrammatics for $\Ker F_{r_c}^{r_c}$}\label{SecOrbHr}
We shall give a diagrammatic description for $\Ker F_{r_c}^{r_c}$.

\subsubsection{The element $\widetilde{\Phi}$}\label{SecPhi}
For convenience, we temporarily use another labelling for Brauer diagrams. Given a Brauer $(r, r)$-diagram $D$, we label the vertices in the top and bottom rows respectively by  $1,2,\dots,r$ and $\overline{1},\overline{2},\dots,\overline{r}$ from left to right.  The bijection
\begin{equation}\label{eqpi}
\pi:\{1, 2, \dots, r, \overline{1}, \overline{2}, \dots, \overline{r}\} \to \{1, 2, \dots, 2r\}, \quad
\pi(i)=2i-1, \  \pi(\overline{i})=2i, \ \forall1\leq i\leq r,
\end{equation}
enables us to convert the present labelling to that defined in Section \ref{sect:lablelling}.

In the present labelling,   $\Sym_{2r}$ may be identified with the permutation group on $\{1, 2, \dots, r,\\ \overline{1}, \overline{2}, \dots, \overline{r}\}$. We define the parabolic subgroup $H_r$ of $\Sym_{2r}$ by 
\begin{equation}\label{eq:H_r}
H_{r}:=\Sym\{1,2,\dots,r\}\times \Sym\{\overline{1},\overline{2},\dots,\overline{r}\}.
\end{equation}
From \rmkref{rmkoddevenact}  we obtain  the $H_{r}$-action as the composition of Brauer diagrams, i.e.,  $\sigma\ast D=\sigma_1 \circ D\circ \sigma_2\inv$ for any  $\sigma=(\sigma_1,\sigma_2)\in H_{r}$.
We denote the $H_r$-orbit by $[D]_{H_r}$, i.e.,  
$[D]_{H_r}:=H_r\ast D$.

Note that $[\Psi]_{H_r} \subseteq \Ker F_r^r$ if $\Psi\in \Ker F_r^r$, and hence particularly we have
$[\Phi]_{H_{r_c}} \subseteq \Ker F_{r_c}^{r_c}$. We now construct an element in $[\Phi]_{H_{r_c}}$.  Write $\mu_c=((n+1)^{m+1})$. Let $w_{\mu_c}:=d(\ft_{\mu_c})$ and $w:=(w_{\mu_c},w_{\mu_c})\in H_{r_c}$. Using \rmkref{rmkoddevenact}, we define
\begin{equation}\label{EqnwPhidef}
\widetilde{\Phi}:=w \ast\hat{\Phi}=w_{\mu_c}\circ\hat{\Phi}\circ w_{\mu_c}^{-1} \quad
\text{with $\hat{\Phi}$ as in \eqref{eqnorm}. }
\end{equation}
For notational convenience, we let
\begin{equation*}
A(m+1)=\sum_{\sigma\in \Sym_{m+1}}\epsilon(\sigma)\sigma,\quad S(n+1)=\sum_{\sigma\in \Sym_{n+1}}\sigma,\quad
B(n+1)=\Phi_{LZ}(n+1),
\end{equation*}
and regard them as morphisms in the Brauer category. Then for all $\sigma\in \Sym_{m+1}$ and $\tau\in \Sym_{n+1}$, we clearly have  the following symmetry properties
\begin{equation}\label{eqsympro}
\begin{aligned}
&A(m+1)\sigma=\sigma A(m+1)=\epsilon(\sigma)A(m+1),\\
&S(n+1)\tau=\tau S(n+1)=S(n+1),\\
&B(n+1)\tau=\tau B(n+1)=B(n+1).
\end{aligned}
\end{equation}
Also regarding $x_{\mu_c}$ and $y_{\mu_c}$   as morphisms in the Brauer category,  one has
\begin{equation}\label{eqxymuc}
x_{\mu_c}=S(n+1)^{\ot (m+1)}, \quad  y_{\mu_c}=w_{\mu_c}^{-1} A(m+1)^{\ot (n+1)}w_{\mu_c},
\end{equation}
where $A(m+1)^{\ot (n+1)}$ is  the alternating sum of column stablisers of $\ft_{\mu_c}$.                                                                   

\begin{lem}\label{propidemp}
	$\widetilde{\Phi}$ is an integral sum of Brauer diagrams and is a quasi-idempotent, that is,
	\begin{eqnarray}\label{EqnwPhi}
	&\widetilde{\Phi}&=A(m+1)^{\ot (n+1)}w_{\mu_c} B(n+1)^{\ot (m+1)} w_{\mu_c}^{-1} A(m+1)^{\ot (n+1)},\\
	&\widetilde{\Phi}^2&=c_{m,n}\widetilde{\Phi}, \quad \text{where $c_{m,n}=((m+1)!)^{n+1}h_{\mu_c}$}. \label{EqnwPhi2}
	\end{eqnarray}
\end{lem}
\begin{proof}
	Using \eqref{eqnorm}, we have 
	$\hat{\Phi}=y_{\mu_c} B(n+1)^{\ot (m+1)}  y_{\mu_c}$. Hence equation  \eqref{EqnwPhi} follows from  \eqref{EqnwPhidef} and \eqref{eqxymuc}. To see \eqref{EqnwPhi2}, we only need to show $\hat{\Phi}$ satisfies the same equation.  Let $B_{r_c}(m-2n)^{(1)}$ the subalgebra of $B_{r_c}(m-2n)$ generated by $e_1$. Then every nonzero element in $B_{r_c}(m-2n)^{(1)}$ contains horizontal edges. Using $B(n+1)\equiv S(n+1) \pmod{B_{r_c}(m-2n)^{(1)}}$, we have
	$\hat{\Phi} \equiv y_{\mu_c} S(n+1)^{\ot (m+1)}  y_{\mu_c} \pmod{B_{r_c}(m-2n)^{(1)}}.$
	By the annihilation property in \propref{AnnThm}, we obtain
	\begin{equation}\label{eqhatPhi}
	\hat{\Phi}^2=((m+1)!)^{n+1}y_{\mu_c}B(n+1)^{\ot (m+1)}y_{\mu_c} S(n+1)^{\ot (m+1)}  y_{\mu_c}.
	\end{equation}
	By the symmetric property \eqref{eqsympro}, we have $B(n+1)=((n+1)!)^{-1}B(n+1)S(n+1)$. Since $x_{\mu_c}=S(n+1)^{\ot (m+1)}$ and $\fc_{\mu_c}=x_{\mu_c}y_{\mu_c}$ is an idempotent, we obtain  $\hat{\Phi}^2=((m+1)!)^{n+1}h_{\mu_c}\hat{\Phi} $ from \eqref{eqhatPhi} as desired.	
\end{proof}

We now give a visualisation of $\widetilde{\Phi}$ in \eqref{EqnwPhi}.   Recall that $d(\ft_{\mu_c})\in \Sym_{r_c}$ satisfies $\ft^{\mu_c}d(\ft_{\mu_c})=\ft_{\mu_c}$. Thus as a Brauer diagram $w_{\mu_c}=d(\ft_{\mu_c})$ has the vertical edges $\{(m+1)(j-1)+i,\overline{(n+1)(i-1)+j}\}$ with $1\leq i\leq m+1,1\leq j\leq n+1$, which join $A(m+1)^{\ot (n+1)}$ and $B(n+1)^{\ot (m+1)}$ pictorially in the following way:
\vskip 0.3cm
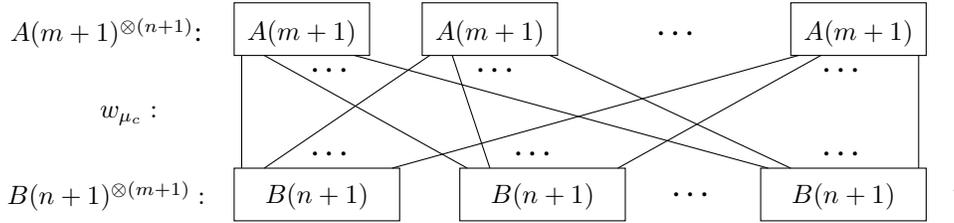
\begin{center}
	\begin{tikzpicture}
	\node at (-1.8,0.13){{\footnotesize $B(n+1)^{\ot (m+1)}:$}};
	\node at (-1.8,2.3){{\footnotesize $A(m+1)^{\ot (n+1)}$:}};
	\node at (-1.5,1.2){{\footnotesize $w_{\mu_c}:$}};
	
	\draw (-0.1,-0.2) rectangle (2.1,0.5);
	\draw (2.9,-0.2) rectangle (5.1,0.5);
	\node at (6.0,0.15){\dots};   	
	\draw (6.9,-0.2) rectangle (9.1,0.5);
	
	\node at (1,0.13){{\footnotesize $B(n+1)$}};
	\node at (4,0.13){{\footnotesize $B(n+1)$}};
	\node at (8,0.13){{\footnotesize $B(n+1)$}};
	
	\draw (0,2) -- (0,0.5);
	\draw (0.3,2) -- (3.0,0.5);
	\draw (1.5,2) -- (7,0.5);
	\node at (1.2,1.8){\dots};  
	
	\draw (2.5,2) -- (0.3,0.5);
	\draw (2.8,2) -- (3.3,0.5);
	\draw (4.1,2) -- (7.3,0.5);
	\node at (3.4,1.8){\dots};
	
	\draw (7.4,2) -- (2,0.5);
	\draw (7.7,2) -- (5,0.5);
	\draw (9,2) -- (9,0.5);
	\node at (8.0,1.8){\dots}; 
	
	\node at (1.2,0.7){\dots};
	\node at (3.9,0.7){\dots};  
	\node at (8.0,0.7){\dots};
	
	\draw (-0.1,2) rectangle (1.7,2.7);  	
	\draw (2.4,2) rectangle (4.2,2.7);
	\node at (5.8,2.3){\dots};  
	\draw (7.3,2) rectangle (9.1,2.7);
	\node at (0.8,2.3){{\footnotesize $A(m+1)$}};
	\node at (3.3,2.3){{\footnotesize $A(m+1)$}};
	\node at (8.2,2.3){{\footnotesize $A(m+1)$}};
	
	\node at (9.5,0.2){.};	 	  	 	  	  	
	\end{tikzpicture}
\end{center}

\begin{exam} \label{ExamOSp12}
	{\rm
		In the case of $\OSp(1|2))$, i.e.,
		$m=n=1$, we have $\mu_c=(2,2)$,
		\begin{equation*}
		\begin{aligned}
		\ft^{\mu_c}=\;&
		\ytableausetup{boxsize=0.5cm}
		\begin{ytableau}
		1&2\\
		3&4	
		\end{ytableau}\;,
		\quad & \ft_{\mu_c}=\;&
		\begin{ytableau}
		1&3\\
		2&4	
		\end{ytableau}\;,
		&\quad& w_{\mu_c}=(23).
		\end{aligned}
		\end{equation*}
		Then $\widetilde{\Phi}$ can be depicted as follows:
		\begin{center}
			\begin{picture}(100, 100)(0,0)
			\put(0,0){\line(0,1){18}}
			\put(0,0){\line(1,0){40}}	
			\put(40,18){\line(0,-1){18}}
			\put(40,18){\line(-1,0){40}}
			\put(8,5){$A(2)$}
			
			\put(60,0){\line(0,1){18}}
			\put(60,0){\line(1,0){40}}	
			\put(100,18){\line(0,-1){18}}
			\put(100,18){\line(-1,0){40}}
			\put(68,5){$A(2)$}
			
			\qbezier(5,18)(5,29)(5,40)
			\qbezier(35,18)(50,29)(65,40)
			\qbezier(35,40)(50,29)(65,18)
			\qbezier(95,18)(95,29)(95,40)
			
			\put(0,40){\line(0,1){18}}
			\put(0,40){\line(1,0){40}}	
			\put(40,58){\line(0,-1){18}}
			\put(40,58){\line(-1,0){40}}
			\put(8,45){$B(2)$}
			
			\put(60,40){\line(0,1){18}}
			\put(60,40){\line(1,0){40}}	
			\put(100,58){\line(0,-1){18}}
			\put(100,58){\line(-1,0){40}}
			\put(68,45){$B(2)$}
			
			\qbezier(5,58)(5,69)(5,80)
			\qbezier(35,58)(50,69)(65,80)
			\qbezier(35,80)(50,69)(65,58)
			\qbezier(95,58)(95,69)(95,80)      		
			
			\put(0,80){\line(0,1){18}}
			\put(0,80){\line(1,0){40}}	
			\put(40,98){\line(0,-1){18}}
			\put(40,98){\line(-1,0){40}}
			\put(8,85){$A(2)$}
			
			\put(60,80){\line(0,1){18}}
			\put(60,80){\line(1,0){40}}	
			\put(100,98){\line(0,-1){18}}
			\put(100,98){\line(-1,0){40}}
			\put(68,85){$A(2)$}	\
			\put(110,20){.}
			\end{picture}	
		\end{center}
		Also see \exref{Examty} for another example.
	}
\end{exam}

\subsubsection{The elements $\widetilde{\Phi}_{\bj}^{\bi}$}
Let us first introduce some notation. Write
\begin{equation}\label{Eqnij}
\begin{aligned}
&\mathbf{i}:=(i_1,i_2,\cdots,i_k),\quad 1\leq i_1<i_2<\cdots<i_k\leq r,\\
&\mathbf{j}:=(j_1,j_2,\cdots,j_l),\quad \overline{1}\leq j_1<j_2<\cdots<j_l\leq \overline{r}
\end{aligned}
\end{equation}
for the sequences and denote their lengths, which are $k$ and $l$,  by $\ell(\bi)$ and $\ell(\bj)$ respectively. 

\begin{defn}\label{defnDij}
	Let $\bi,\bj$ be as above. We define the bijective map $\psi_{\bj}^{\bi}: B_{r}^r(m-2n) \rightarrow B_{r-l+k}^{r-k+l}(m-2n)$ by $ \psi_{\bj}^{\bi}(D)=D_{\bj}^{\bi}$ with $D_{\bj}^{\bi}$ as shown below:
	
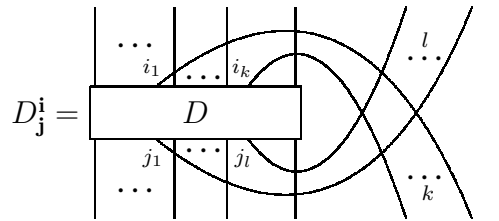
\begin{center}
		\begin{picture}(145, 88)(-30,-35)
		\put(-30,6){$D_{\bj}^{\bi}=$}	
		\put(2,20){\line(0,1){30}}
		\put(78,20){\line(0,1){30}}
		
		\qbezier(25,20)(100,80)(145,-30)
		\qbezier(60,20)(90,60)(120,-30)
		\put(20,25){\tiny $i_1$}
		\put(54,25){\tiny $i_k$}
		\put(36,23){\dots}
		\put(120,-15){\dots}
		\put(126,-23){{\tiny $k$}}
		
		\put(32,20){\line(0,1){30}}
		\put(52,20){\line(0,1){30}}
		\put(32,0){\line(0,-1){30}}
		\put(52,0){\line(0,-1){30}}
		\put(10,35){\dots}
		\put(10,-20){\dots}
		
		\put(0, 0){\line(1, 0){80}}
		\put(0, 0){\line(0, 1){20}}
		\put(80, 0){\line(0, 1){20}}
		\put(0, 20){\line(1, 0){80}}
		\put(35, 6){$D$}
		
		\put(2,0){\line(0,-1){30}}
		\put(78,0){\line(0,-1){30}}
		\qbezier(25,0)(100,-60)(145,50)
		\qbezier(60,0)(90,-40)(120,50)
		\put(20,-9){\tiny $j_1$}
		\put(55,-9){\tiny $j_l$}
		\put(36,-5){\dots}
		\put(120,30){\dots}
		\put(126,35){{\tiny $l$}}
		\put(155,-20){.}	
		\end{picture}
	\end{center}
	This can be extended linearly, i.e.,  $\psi^{\bi}_{\bj}(\sum_{D} a_D D)=\sum_{D}a_D D^{\bi}_{\bj}$.
\end{defn}

\begin{exam}
	{\rm Let $D$ be the diagram as in \exref{examDact}. Then we have
		\begin{center}
			\begin{picture}(100,60)(-10,-10)
			\put(-50,12){$D^{(2,3)}_{(\bar{3})}=$}
			\qbezier(0,30)(30,10)(60,30)
			\qbezier(30,30)(15,15)(0,0)
			\qbezier(30,0)(45,30)(60,0)	
			
			\put(-2,32){\tiny$1$}
			\put(-2,-9){\tiny$\bar{1}$}
			\put(28,32){\tiny$2$}
			\put(28,-9){\tiny$\bar{2}$}
			\put(56,32){\tiny$3$}
			\put(56,-9){\tiny$\bar{3}$}
			
			\qbezier(60,30)(75,60)(80,0)
			\qbezier(30,30)(90,80)(100,0)
			\qbezier(60,0)(90,-40)(100,30)
			\put(110,12){$=$}
			
			\qbezier(130,30)(150,15)(170,0)
			\qbezier(150,30)(150,15)(150,0)
			\qbezier(130,0)(160,25)(190,0)
			\put(200,4){,}
			\end{picture}		
		\end{center}
		which is a Brauer $(4,2)$-diagram.	Note that $D_{\bj}^{\bi}\in B_{r}(m-2n)$ if and only if $\ell(\bi)=\ell(\bj)$. 		
	}
\end{exam}

We now relate the diagram $D^{\bi}_{\bj}$ to the symmetric group action on $D$.  Algebraically,  the  representatives for right cosets $ H_r\sigma$ in $\Sym_{2r}$ arise from such pair $\bi,\bj$ of equal length. Denote by $(i_s, j_s)$  the transposition in $\Sym_{2r}$, and let 
$(\bi,\bj):=\prod_{s=1}^{k}(i_s,j_s)$.
Then $(\bi,\bj)$ is one of representatives for the cosets $ H_r\sigma$. In what follows, we shall assume that the right coset representatives of $H_r$ are of this form. We have 
$D_{\bj}^{\bi}\in [(\bi,\bj)\ast D]_{H_r},$
since it is clear from diagram that there exists $\sigma=(\sigma_1,\sigma_2)\in H_r$ such that $(\bi,\bj)\ast D=\sigma\ast D=\sigma_1\circ D_{\bj}^{\bi}\circ \sigma_2^{-1}$. Therefore, $[D_{\bj}^{\bi}]_{H_r}=[(\bi,\bj)\ast D]_{H_r}$.

Given any elements $D_1, D_2, \dots, D_N$ in the Brauer algebra $B_r(m-2n)$, we use $\langle D_1, \dots, D_N\rangle_{r}$ to denote the $2$-sided ideal in $B_r(m-2n)$ which they generate.

\begin{prop} \label{PropPsiij}
	$	\Ker F_{r_c}^{r_c}=\langle\widetilde{\Phi}_{\bj}^{\bi} \mid         \ell(\bi)=\ell(\bj)\le r_c \rangle_{r_c}.$
\end{prop}

\begin{proof}
	Since $\widetilde{\Phi}_{\bj}^{\bi}\in [(\bi,\bj)\ast \widetilde{\Phi}]_{H_{r_c}}$ and $\Ker F_{r_c}^{r_c}$ is $\CA_{2r}$-module, we have $\widetilde{\Phi}_{\bj}^{\bi}\in\Ker F_r^r$ for all ${\bi}, {\bj}$ of equal length. So by \thmref{thmBasisMin} it remains to prove that the basis elements $\Phi_{\ft}$ of $\Ker F_{r_c}^{r_c}$  are generated by elements $\widetilde{\Phi}_{\bj}^{\bi}$.
	Using \eqref{eqnorm} and \eqnref{EqnwPhidef}, we have $\Phi_{\ft}=a_{m,n} (wd(\ft))^{-1}\ast\widetilde{\Phi}$. Observe that there exist $\bi$ and $\bj$ of equal length such that $(wd(\ft))^{-1} \in H_{r_c}(\bi,\bj)$. Thus
	$a_{m,n}^{-1}\Phi_{\ft}\in [(\bi,\bj)\ast\widetilde{\Phi}]_{H_{r_c}}=[\widetilde{\Phi}_{\bj}^{\bi}]_{H_{r_c}}$. For any $h=(\sigma_1, \sigma_2)\in H_{r_c}$,
	we have $
	h\ast\widetilde{\Phi}_{\bj}^{\bi}=\sigma_1\circ\widetilde{\Phi}_{\bj}^{\bi}\circ\sigma_2^{-1}
	\in \langle\widetilde{\Phi}_{\bj}^{\bi} \mid   \ell(\bi)=\ell(\bj)\le r_c \rangle_{r_c}$, and hence we have $\Phi_{\ft}\in \langle\widetilde{\Phi}_{\bj}^{\bi} \mid   \ell(\bi)=\ell(\bj)\le r_c \rangle_{r_c}$.
\end{proof}

\begin{rmk}
	By \propref{PropPsiij}, to reduce the number of generators for the kernel $\Ker F_{r_c}^{r_c}$,  we should focus on the right coset representatives of $H_{r_c}$ in $\Sym_{2r_c}$. 
\end{rmk}

\subsection{Garnir relations}\label{secGarnir}
We want to select out a set of the elements of the form  $\widetilde{\Phi}^{\bi}_{\bj}$, which  will generate $\Ker F_{r_c}^{r_c}$ and can be
characterised more explicitly. To this end, we shall introduce Garnir relations \cite[\S 7.2]{J} among these elements.

\subsubsection{Standard sequences of increasing types}\label{SecStdSeq}
We need some combinatorial notions first. 
\begin{defn} \label{defstdseq}
	Given  sequences $\bi$ and $\bj$ with $1\leq i_1<\dots<i_k\leq r_c$ and  $1\leq j_1<\dots<j_l\leq r_c$.
	We refer to
	$\ty(\bi)=(a_1,a_2,\dots,a_{n+1})$ with
	\[
	a_p=\#\{i_s\;|\; (p-1)(m+1)+1\leq i_s\leq p(m+1), \; 1\leq s\leq k \}, \quad
	1\leq p\leq n+1,
	\]
	as the \textit{type} of $\bi$.  Similarly, we define $\ty(\bj):=(b_1,b_2,\dots,b_{n+1})$ such that 
	\[
	b_q=\#\{j_s\;|\; \overline{(q-1)(m+1)+1}\leq j_s\leq \overline{q(m+1)}, \; 1\leq s\leq l \}, \quad 1\leq q\leq n+1.
	\]
\end{defn}
Clearly, $0\leq a_p \leq m+1$ and $0\leq b_q \leq m+1$ for all $p,q$.
Furthermore, $\sum_{p=1}^{n+1}a_p=\ell(\bi)$ and  $\sum_{q=1}^{n+1}b_q=\ell(\bj)$. The relevance of the types of $\bi$ and $\bj$ in the description of $\widetilde{\Phi}_{\bj}^{\bi}$ is best illustrated by an example.

\begin{exam} 
	{\rm
		\label{Examty}
		Suppose that $m=1,n=2$. Then we have $\mu_c=(3,3)$ and $w_{\mu_c}=(2354)$. By \eqnref{EqnwPhi} we represent $\widetilde{\Phi}$ pictorially as follows:
		
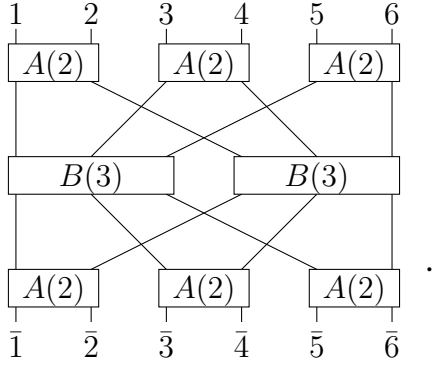
\begin{center}
			\begin{tikzpicture}
			\draw (-0.1,0) rectangle (2.1,0.5);
			\draw (2.9,0) rectangle (5.1,0.5);
			\node at (1,0.21){$B(3)$};
			\node at (4,0.21){$B(3)$};
			
			\draw (0,0.5) -- (0,1.5);	
			\draw (1,0.5) -- (2,1.5);	  		
			\draw (2,0.5) -- (4,1.5);	  	
			\draw (3,0.5) -- (1,1.5);	
			\draw (4,0.5) -- (3,1.5);	
			\draw (5,0.5) -- (5,1.5);	
			
			\draw (-0.1,1.5) rectangle (1.1,2);  	
			\draw (1.9,1.5) rectangle (3.1,2);
			\draw (3.9,1.5) rectangle (5.1,2);
			\node at (0.5,1.71){$A(2)$};
			\node at (2.5,1.71){$A(2)$};
			\node at (4.5,1.71){$A(2)$};
			
			\draw (0,2) -- (0,2.2);	
			\draw (1,2) -- (1,2.2);	  		
			\draw (2,2) -- (2,2.2);	  	
			\draw (3,2) -- (3,2.2);	
			\draw (4,2) -- (4,2.2);	
			\draw (5,2) -- (5,2.2);	  	
			
			\node at (0,2.4){$\tiny 1$};  	
			\node at (1,2.4){$\tiny 2$};
			\node at (2,2.4){$\tiny 3$};    	
			\node at (3,2.4){$\tiny 4$};    	
			\node at (4,2.4){$\tiny 5$};
			\node at (5,2.4){$\tiny 6$};
			\draw (0,0) -- (0,-1);	
			\draw (1,0) -- (2,-1);	  		
			\draw (2,0) -- (4,-1);	  	
			\draw (3,0) -- (1,-1);	
			\draw (4,0) -- (3,-1);	
			\draw (5,0) -- (5,-1);			
			
			\draw (-0.1,-1) rectangle (1.1,-1.5);  	
			\draw (1.9,-1) rectangle (3.1,-1.5);
			\draw (3.9,-1) rectangle (5.1,-1.5); 		
			\node at (0.5,-1.28){$A(2)$};
			\node at (2.5,-1.28){$A(2)$};
			\node at (4.5,-1.28){$A(2)$};
			
			\draw (0,-1.5) -- (0,-1.7);	
			\draw (1,-1.5) -- (1,-1.7);	  		
			\draw (2,-1.5) -- (2,-1.7);	  	
			\draw (3,-1.5) -- (3,-1.7);	
			\draw (4,-1.5) -- (4,-1.7);	
			\draw (5,-1.5) -- (5,-1.7);	  	
			
			\node at (0,-2){$\tiny \bar{1}$};  	
			\node at (1,-2){$\tiny \bar{2}$};
			\node at (2,-2){$\tiny \bar{3}$};    	
			\node at (3,-2){$\tiny \bar{4}$};    	
			\node at (4,-2){$\tiny \bar{5}$};
			\node at (5,-2){$\tiny \bar{6}$};
			\node at (5.5,-1){.};	 	  	 	  	  	
			\end{tikzpicture}
		\end{center}
		If we take $\bi=(1,3,4)$, then $\ty(\bi)=(1,2,0)$. From the picture we can see that  $a_p$,   the $p$-th entry of $\ty(\bi)$, is the number of vertices in $\bi$ which belong to  $p$-th  block $A(2)$ at the top. These vertices are moved down to the bottom row in $\widetilde{\Phi}_{\bj}^{\bi}$. We have similar arguments  for  $\bj$.
	}
	
\end{exam}

The interpretation of the entries of $\ty(\bi)$ given in the example clearly generalises to arbitrary sequences for any $m$ and $n$.
It then follows from the total skew symmetry of $A(m+1)$ (see \eqref{eqsympro}) that we have the following result.
\begin{lem} Assume that $\ell(\bi)=\ell(\bj)\le r_c$. Then for any $\bi'$ and $\bj'$ which are of the same types as $\bi$ and $\bj$ respectively, $\widetilde{\Phi}_{\bj}^{\bi}=\pm\widetilde{\Phi}_{\bj'}^{\bi'}$.
\end{lem}

\begin{defn}
	We choose a special sequence $\bi$ of the given type $(a_1, a_2, \dots, a_{n+1})$ as follows: for each $a_p\neq 0$, we pick the first $a_p$ vertices in the $p$-th $A(m+1)$ from left to right and denote the set of these $a_p$ vertices by $\bi_p$. We call  $\bi$ the \textit{standard sequence} of type $(a_1,a_2,\dots,a_{n+1})$ and $\bi_p$ the \textit{$p$-th piece} of $\bi$. Furthermore, if $0\leq a_1\leq a_2\leq\cdots\leq a_{n+1}\leq m+1$, we call $\bi$ the \textit{standard sequence of increasing type}.
\end{defn}

\begin{exam}
	{\rm If we take
		$\bi=(1,3,4), \bi^{\prime}=(2,3,4)$ and $\bj=(\overline{3},\overline{4},\overline{5})$
		in \exref{Examty},
		then $\bi$ and $\bj$ are standard sequences with type $(1,2,0)$ and $(0,2,1)$ respectively, while $\bi^{\prime}$ is a sequence of type $(1,2,0)$ that is not standard. By the symmetry property of $A(2)$, we have $\widetilde{\Phi}^{\bi}_{\bj}=-\widetilde{\Phi}^{\bi^{\prime}}_{\bj}$. The pieces of $\bi$ corresponding to each entries of $\ty(\bi)$ are: $\bi_1=(1)$, $\bi_2=(3,4)$ and $\bi_3=\emptyset$.
	}
\end{exam}

\begin{lem}\label{lemincty}	
	Let $\ty(\bi)$ and $\ty(\bj)$ be the types of standard sequences $\bi$ and $\bj$ respectively.  Then $\widetilde{\Phi}^{\bi}_{\bj}\in [\widetilde{\Phi}^{\bi^{\prime}}_{\bj^{\prime}}]_{H_{r_c}}$, where $\bi^{\prime}$ and $\bj^{\prime}$ are standard sequences of increasing types with $\ty(\bi^{\prime})$ and $\ty(\bj^{\prime})$  respectively obtained from $\ty(\bi)$ and $\ty(\bj)$ by re-arranging the entries in increasing order.
\end{lem}
\begin{proof}
	It is equivalent to showing that $(\bi,\bj)\ast \widetilde{\Phi}\in  [(\bi^{\prime},\bj^{\prime})\ast \widetilde{\Phi}]_{H_{r_c}}$, since $(\bi,\bj)\ast \widetilde{\Phi}$ and $\widetilde{\Phi}^{\bi}_{\bj}$ are in the same $H_{r_c}$-orbit for any $\bi$ and $\bj$.
	
	We need a symmetry property of $\widetilde{\Phi}$. For that purpose, we define  the injective map $X_{m,n}: \Sym_{n+1}\rightarrow \Sym_{(m+1)(n+1)}$ by sending  $\sigma$ to $X_{m,n}(\sigma)$. Here $X_{m,n}(\sigma)$ is a permutation on the set $\{1,2,\dots,(m+1)(n+1)\}$  defined by 
	\[ ( X_{m,n}(\sigma))(k)=
	\begin{cases}
	\sigma(k_1+1)+k_2-1, \quad & 0<k_2\leq m,\\
	\sigma(k_1)+m,\quad  &k_2=0,
	\end{cases}
	\]
	where $k_1$ and $k_2$ are respectively  the unique  quotient and residue in the expression $k=k_1(m+1)+k_2$ for all $1\leq k\leq r_c=(m+1)(n+1)$. Write $X_{m,n}(\sigma)$ as $\widetilde{\sigma}$ for simplicity. A key observation is that $\widetilde{\Phi}$ is fixed under the action of $(\widetilde{\sigma}_1, \widetilde{\sigma}_2)\in H_{r_c}$ for any $\sigma_1,\sigma_2\in \Sym_{n+1}$, i.e.,
	\[(\widetilde{\sigma}_1, \widetilde{\sigma}_2)\ast \widetilde{\Phi}=\widetilde{\sigma}_1\circ \widetilde{\Phi}\circ \widetilde{\sigma}_2^{-1}= \widetilde{\Phi}. \]
	
	Now let $\ty(\bi)=(a_1,a_2,\dots,a_{n+1})$, then there exists a permutation $\sigma\in \Sym_{n+1}$ such that $\sigma.\ty(\bi)=(a_{\sigma^{-1}(1)},a_{\sigma^{-1}(2)},\dots,a_{\sigma^{-1}(n+1)})$ is of increasing type. Similarly, there exists $\tau\in \Sym_{n+1}$ making the sequence $\tau.\ty(\bj)$ increasing. Assuming $\ell(\bi)=k$ and $\ell(\bj)=l$,  we take $\bi^{\prime}=\widetilde{\sigma}(\bi)=(\widetilde{\sigma}(i_1),\widetilde{\sigma}(i_2),\dots,\widetilde{\sigma}(i_k))$, and $\bj^{\prime}=\widetilde{\tau}^{-1}(\bj)=(\widetilde{\tau}^{-1}(j_1),\widetilde{\tau}^{-1}(j_2),\dots,\widetilde{\tau}^{-1}(j_l))$. It can be verified that  $\bi^{\prime}$ and $\bj^{\prime}$ are  the standard sequences of increasing types $\sigma.\ty(\bi)$ and $\tau.\ty(\bj)$, respectively. Particularly, we have
	$ (\widetilde{\sigma}^{-1}, \widetilde{\tau})\ast \widetilde{\Phi}=\widetilde{\Phi}$ by the above symmetry property. Noting that $(\widetilde{\sigma}, \widetilde{\tau}^{-1}) (\bi,\bj)(\widetilde{\sigma}^{-1}, \widetilde{\tau}) =(\widetilde{\sigma}(\bi),\widetilde{\tau}^{-1}(\bj ) =(\bi^{\prime},\bj^{\prime})$, we have
	\[(\bi,\bj)\ast \widetilde{\Phi} = (\bi,\bj) (\widetilde{\sigma}^{-1}, \widetilde{\tau}) \ast \widetilde{\Phi} =(\widetilde{\sigma}^{-1}, \widetilde{\tau})(\bi^{\prime},\bj^{\prime})\ast \widetilde{\Phi}.  \]
	Since $(\widetilde{\sigma}^{-1}, \widetilde{\tau})\in H_{r_c}$, we obtain $(\bi,\bj)\ast \widetilde{\Phi}\in  [(\bi^{\prime},\bj^{\prime})\ast \widetilde{\Phi}]_{H_{r_c}}$ as desired.
\end{proof}

\subsubsection{Garnir relations}
We start by recalling some basic facts on Garnir relations.
Let $\ft$ be a tableau of shape $\la$. For $i < j$, let $X$ and $Y$ be  the $i$-th and $j$-th columns of $\ft$ respectively.
The \textit{Garnir element} \cite[\S 7.2]{J} associated with $\ft$ and $X,Y$ is
$
G_{X,Y}=\sum_{\sigma}\epsilon(\sigma)\sigma,
$
where the sum is over a set of the right coset representatives of  $\Sym\{X\}\times\Sym\{Y\}$ in $\Sym\{X\cup Y\}$. 

We will choose right coset representatives as follows.
For any $k$ such that $1\leq k\leq min(|X|, |Y|)$, we select elements $a_1,a_2,\dots,a_k$ from $X$ and $b_1,b_2,\dots,b_k$ from $Y$. Let $(a_i, b_i)$ denote the transposition in $\Sym\{X\cup Y\}$. Then the product $\prod_{i=1}^{k}(a_i, b_i)$ is the right coset representative of $\Sym\{X\}\times\Sym\{Y\}$ in $\Sym\{X\cup Y\}$.

{\color{red}	
}

\begin{thm} [\cite{J}]\label{thmGanirrel}
	Let $\ft$, $X$ and $Y$ be as above. If $|X\cup Y|$ is greater than the length of the $i$-th column of $\ft$, then $\fc_{\la}(\ft)G_{X,Y}=0$.
\end{thm}

\begin{exam}
	{\rm
		Assume that
		$
		\ft=
		\ytableausetup
		{boxsize=0.5cm}
		\begin{ytableau} 
		1&3&4\\
		2&5	
		\end{ytableau}
		$
		and $X=\{1,2\}$, $Y=\{3,5\}$. Then $G_{X,Y}=(1)-(23)-(25)-(13)-(15)+(13)(25)$, and it can be easily verified that  $\fc_{\la}(\ft)G_{X,Y}=0$.
	}
\end{exam}

We now use Garnir relations to derive relations among elements of $\Ker F_{r_c}^{r_c}$. For notational convenience, we shall adopt the labelling of Brauer diagram in \secref{sect:lablelling}. All the results in the previous sections can be translated to this setting by using the bijection $\pi$ defined in \eqref{eqpi}. For instance, $H_{r_c}=\Sym\{1,3,\dots,2r_c-1\}\times \Sym\{2,4,\dots,2r_c\}$.

In what follows,  we shall  assume that  $\bi$ and $\bj$ are  standard sequences of increasing types of equal length. Then $\bi$ (resp. $\bj$) is a sequence of odd (resp. even) integers from the set $\{1,3,\dots,2r_c-1 \}$ (resp. $\{2,4,\dots, 2r_c\}$). Recall that $\la_c=((2n+2)^{m+1})$, the odd (resp. even) columns of the standard tableau $\ft^{\la_c}$ are exactly filled with integers $\{1,3,\dots,2r_c-1 \}$ (resp. $\{2,4,\dots, 2r_c\}$). Therefore, we may identify $\bi$ (resp. $\bj$) as a sequence of odd (resp. even) column entries in $\ft^{\la_c}$.

Recall in \eqref{EqnwPhidef} that $w_{\mu_c}\in \Sym_{r_c}$, which can now be identified as the permutation on $\{1,3,\dots,2r_c-1\}$ (resp. $\{2,4,\dots,2r_c\}$), i.e., $w_{\mu_c}(2k-1)=2w_{\mu_c}(k)-1$ (resp. $w_{\mu_c}(2k)=2w_{\mu_c}(k)$) with $1\leq k\leq r_c$.  In particular, this restricts to a permutation on the sequence $\bi$ (resp. $\bj$), which is denoted by $w_{\mu_c}(\bi)$ (resp. $w_{\mu_c}(\bj)$). 
\begin{lem} \label{LemOrbPhi}
	We have
	$$\bU_{r_c}(\hat{A}^{ r_c}\circ \hat{\fc}_{\la_c}(w_{\mu_c}^{-1}(\bi),w_{\mu_c}^{-1}(\bj)))=(w_{\mu_c}^{-1}(\bi),w_{\mu_c}^{-1}(\bj))\ast \hat{\Phi}\in [\widetilde{\Phi}^{\bi}_{\bj}]_{H_{r_c}}.$$
\end{lem}
\begin{proof}
	Recalling that $\widetilde{\Phi}=w\ast\hat{\Phi}$ with $w=(w_{\mu_c},w_{\mu_c})\in H_{r_c}$, we obtain
	\begin{equation}\label{EqnOrbPhi}
	w\inv(\bi,\bj)\ast\widetilde{\Phi}=w\inv(\bi,\bj)w\ast\hat{\Phi}=(w_{\mu_c}^{-1}(\bi),w_{\mu_c}^{-1}(\bj))\ast\bU_{r_c}(\hat{A}^{ r_c}\circ \hat{\fc}_{\la_c}).
	\end{equation}
	Using this,  we have $\bU_{r_c}(\hat{A}^{ r_c}\circ \hat{\fc}_{\la_c}(w_{\mu_c}^{-1}(\bi),w_{\mu_c}^{-1}(\bj)))\in [\widetilde{\Phi}^{\bi}_{\bj}]_{H_{r_c}}$	since $(\bi,\bj)\ast\widetilde{\Phi}\in [\widetilde{\Phi}^{\bi}_{\bj}]_{H_{r_c}}$.
\end{proof}

Recall that $\bi_p$ is the $p$-th piece of the sequence $\bi$. 	Let $\ty(\bi)=(a_1, a_2, \dots, a_{n+1})$ and $\ty(\bj)=(b_1, b_2, \dots, b_{n+1})$. Then it can be verified that $w_{\mu_c}^{-1}(\bi_p)$ (resp. $w_{\mu_c}^{-1}(\bj_q)$) is a sequence of  the first $a_p$-th (resp. $b_q$-th)  entries of the $2p-1$-th (resp. $2q$-th) column of $\ft^{\la_c}$; see \exref{examOSp12}. Then the following lemma is immediate by \thmref{thmGanirrel}.

\begin{lem}\label{lemPhiGarnir}
	Maintain the above notation. If $a_p+b_q>m+1$ for some $1\leq p, \ q\leq n+1$, then we have the following Garnir relation
	$$\bU_{r_c}(\hat{A}^{ r_c}\circ \hat{\fc}_{\la_c}G_{X_p,Y_q})=(G_{X_p,Y_q})^{\sharp}\ast \hat{\Phi}=0,$$
	where $X_p=w_{\mu_c}^{-1}(\bi_p) $  and $Y_q=w_{\mu_c}^{-1}(\bj_q)$.
\end{lem}

\begin{exam}\label{examOSp12}{\rm $(\OSp(1|2))$
		Let $m=n=1$. Then $\la_c=(4,4)$, $\mu_c=(2,2)$,  $r_c=4$. 
		Recall that
		$
		\ft^{\la_c}=
		\ytableausetup
		{boxsize=0.5cm}
		\begin{ytableau}
		1&2&3&4\\
		5&6&7&8	
		\end{ytableau}
		$.
		Then $w_{\mu_c}^{-1}=(23)$ can be viewed as a permutation on the top labelling set $\{1,3,5,7\}$ or bottom labelling set  $\{2,4,6,8\}$, i.e., 
		$$
		\begin{aligned}
		w_{\mu_c}^{-1}(1)=1,\quad  w_{\mu_c}^{-1}(3)=5, \quad w_{\mu_c}^{-1}(5)=3,\quad w_{\mu_c}^{-1}(7)=7,\\
		w_{\mu_c}^{-1}(2)=2, \quad w_{\mu_c}^{-1}(4)=6,\quad w_{\mu_c}^{-1}(6)=4,\quad w_{\mu_c}^{-1}(8)=8.
		\end{aligned}
		$$		
		If we take $\bi=(5,7)$ and $\bj=(2,6)$, then $\ty(\bi)=(0,2)$, $\ty(\bj)=(1,1)$ and $w_{\mu_c}^{-1}(\bi)=\{3,7\}$, $w_{\mu_c}^{-1}(\bj)=\{2,4\}$.
		Hence we  have  $X_2=w_{\mu_c}^{-1}(\bi_2)=\{3,7\}$ and $Y_2=w_{\mu_c}^{-1}(\bj_2)=\{4\}$. The corresponding Garnir element is $G_{X_2,Y_2}=(1)-(34)-(47)$. By  \lemref{lemPhiGarnir}, we have the Garnir relation
		\[ \hat{\Phi}-(34)\ast \hat{\Phi}-(47)\ast\hat{\Phi}=0.\]
		
		Note that $(w_{\mu_c}^{-1}(\bi), w_{\mu_c}^{-1}(\bj))=(23)(47)$. Using this relation  and \lemref{LemOrbPhi}, we have 
		\begin{equation}\label{eqngen}
		(23)\ast \hat{\Phi}-(234)\ast \hat{\Phi}=(23)(47)\ast \hat{\Phi}\in [\widetilde{\Phi}^{\bi}_{\bj}]_{H_{r_c}},
		\end{equation}
		where $\ty(\bi)=(0,2)$ and $\ty(\bj)=(1,1)$. We want to show that $\widetilde{\Phi}^{\bi}_{\bj}$ is generated by $\widetilde{\Phi}^{\bi^\prime}_{\bj^\prime}$ with $\ty(\bi^\prime)=\ty(\bj^\prime)=(0,1)$. By \eqref{eqngen} it suffices to show that $ (23)\ast \hat{\Phi}$ and $ (234)\ast \hat{\Phi}$ belong to $[\widetilde{\Phi}^{\bi^\prime}_{\bj^\prime}]_{H_{r_c}}$. The former is true by  \lemref{lemincty} and \lemref{LemOrbPhi}, so is the latter since $(234)\ast \hat{\Phi}=(24)(23)\ast \hat{\Phi}$ and $(24)\in H_{r_c}$.}	
\end{exam}

We are inspired by \exref{examOSp12} to introduce the following key lemma.

\begin{lem}\label{propPhigen}
	Let $\ty(\bi)=(a_1,a_2,\dots,a_{n+1})$ and $\ty(\bj)=(b_1,b_2,\dots,b_{n+1})$. We define
	\begin{equation}\label{eq:Genset2}
	\CI_{r_c}:=
	\Biggl\{ \{\bi,\bj\}\Biggm|
	\begin{matrix}
	&\text{$\bi,\bj$ are standard sequences of increasing types}\\
	&\text{ with equal length such that $a_{n+1}+b_{n+1}\leq m+1$}\\
	\end{matrix}
	\Biggr\}.
	\end{equation}
	Then for any standard sequences $\bi$ and $\bj$ of increasing types, the element 	
	$\widetilde{\Phi}_{\bj}^{\bi}$  belongs to the 2-sided ideal  in $B_{r_c}^{r_c}(m-2n)$ generated  by  the set $\{\widetilde{\Phi}_{\bj}^{\bi} \mid \{\bi,\bj\}\in \CI_{r_c}\}$.
\end{lem}
\begin{proof}
	Given any $\bi$ and $\bj$, we use induction on $a_{n+1}+b_{n+1}$. When $a_{n+1}+b_{n+1}\leq m+1$, there is nothing to prove.  Let $t=a_{n+1}$,
	$s=b_{n+1}$, and $\ell(\bi)=\ell(\bj)=k\leq r_c$.
	Without loss of generality, we may assume that  $t\geq s$ and $s+t>m+1$. Let
	$$
	\begin{aligned}
	&\bar{\bi}=(i_{k-s+1},i_{k-s+2},\dots,i_k), \quad \bar{\bi}^{c}=(i_1,i_2,\dots,i_{k-s}),\quad \bi= \bar{\bi}^{c}\cup \bar{\bi},\\
	&\bar{\bj}=(j_{k-s+1},j_{k-s+2},\dots,j_k), \quad \bar{\bj}^{c}=(j_1,j_2,\dots,j_{k-s}),\quad \bj=\bar{\bj}^{c}\cup \bar{\bj}.\\
	\end{aligned}
	$$
	Let $X_{n+1}=w^{-1}_{\mu_c}(\bi_{n+1})=\{w^{-1}_{\mu_c}(i)\;|\;i=i_{k-t+1},\dots,i_k\}$, $Y_{n+1}==w^{-1}_{\mu_c}(\bj_{n+1})=\{w^{-1}_{\mu_c}(j)\;|\;j=j_{k-s+1},\dots,j_k\}$.
	Then the corresponding Garnir element is
	$$G_{X_{n+1},Y_{n+1}}=\sum_{\sigma\in R} \epsilon(\sigma)\sigma=(-1)^s(w^{-1}_{\mu_c}(\bar{\bi}),w^{-1}_{\mu_c}(\bar{\bj}))+\sum_{\sigma\in R^{\prime}} \epsilon(\sigma)\sigma,$$
	where $R$ is the set of right coset representatives of  $\Sym\{X_{n+1}\}\times\Sym\{Y_{n+1}\}$ in $\Sym\{X_{n+1}\cup Y_{n+1}\}$, and $R^{\prime}=R\backslash \{(w^{-1}_{\mu_c}(\bar{\bi}),w^{-1}_{\mu_c}(\bar{\bj})) \}$. 	
	Using \lemref{lemPhiGarnir}, we obtain
	$$
	(w^{-1}_{\mu_c}(\bar{\bi}),w^{-1}_{\mu_c}(\bar{\bj}))\ast \hat{\Phi}=(-1)^{s+1}\sum_{\sigma\in R^{\prime}}\epsilon(\sigma)\sigma \ast \hat{\Phi}.
	$$
	Since $ (w^{-1}_{\mu_c}(\bi),w^{-1}_{\mu_c}(\bj))=(w^{-1}_{\mu_c}(\bar{\bi}^{c}),w^{-1}_{\mu_c}(\bar{\bj}^c))(w^{-1}_{\mu_c}(\bar{\bi}),w^{-1}_{\mu_c}(\bar{\bj}))$, we have
	\begin{equation}\label{eqnind}
	(w^{-1}_{\mu_c}(\bi),w^{-1}_{\mu_c}(\bj))\ast \hat{\Phi}=(-1)^{s+1}\sum_{\sigma\in R^{\prime}}\epsilon(\sigma)  (w^{-1}_{\mu_c}(\bar{\bi}^{c}),w^{-1}_{\mu_c}(\bar{\bj}^c))\sigma \ast \hat{\Phi}\in [\hat{\Phi}^{\bi}_{\bj}]_{H_{r_c}}.
	\end{equation}
	
	Now by \eqref{eqnind} it is enough to show that $(w^{-1}_{\mu_c}(\bar{\bi}^{c}),w^{-1}_{\mu_c}(\bar{\bj}^c))\sigma \ast \hat{\Phi}$ is generated by the elements $\widetilde{\Phi}_{\bj}^{\bi}$ with $ \{\bi,\bj\}\in \CI_{r_c}$ for any $\sigma\in R^{\prime}$.  Suppose that 
	$(w^{-1}_{\mu_c}(\bar{\bi}^{c}),w^{-1}_{\mu_c}(\bar{\bj}^c))\sigma \ast \hat{\Phi} \in [\widetilde{\Phi}^{\bi^{\prime}}_{\bj^{\prime}}]_{H_{r_c}},$
	where $\bi^{\prime},\bj^{\prime}$ are standard sequences of increasing types with  $\ty(\bi^{\prime})=(a_1^{\prime},\dots,a_{n+1}^{\prime})$ and 	$\ty(\bj^{\prime})=(b_1^{\prime},\dots,b_{n+1}^{\prime})$. Then we have $\ell(\bi^{\prime})=\ell(\bj^{\prime})<\ell(\bi)=\ell(\bj)$,  $a_{n+1}^{\prime}+b_{n+1}^{\prime}\leq a_{n+1}+b_{n+1}$, and hence
	\begin{equation}\label{eqab}
	a_{n+1}^{\prime}+b_{n+1}^{\prime}\leq \ell(\bi^{\prime})+\ell(\bj^{\prime})<\ell(\bi)+\ell(\bj)=\sum_{p=1}^{n+1}(a_p+b_p).
	\end{equation}
	We can always assume that $a_{n+1}^{\prime}+b_{n+1}^{\prime}< a_{n+1}+b_{n+1}$, since when the equality happens we use Garnir relations again as above for $\bi^{\prime},\bj^{\prime}$, and then by \eqnref{eqab} after finite repeated steps we stop at  $a_{n+1}^{\prime\prime}+b_{n+1}^{\prime\prime}< a_{n+1}+b_{n+1}$ for some $\bi^{\prime\prime}$ and $\bj^{\prime\prime}$. Thus, by induction each summand on the right hand side of $\eqnref{eqnind}$ is generated by $\{ \widetilde{\Phi}_{\bj}^{\bi} \mid \{\bi,\bj\}\in \CI_{r_c}\}$, so is $\hat{\Phi}^{\bi}_{\bj}$.
\end{proof}

\begin{thm}\label{thmKermini}
	Let $r_c=(m+1)(n+1)$. Then $\Ker F_{r_c}^{r_c}$ as a 2-sided ideal of $B_{r_c}(m-2n)$ is generated by the set $\{\widetilde{\Phi}_{\bj}^{\bi} \mid \{\bi,\bj\}\in \CI_{r_c}\}$.
\end{thm}
\begin{proof}
	Using \lemref{lemincty} and \propref{PropPsiij}, we conclude that $\Ker F_{r_c}^{r_c}$ is generated by the elements $\widetilde{\Phi}_{\bj}^{\bi}$ with
	standard sequences ${\bi}$ and ${\bj}$ of increasing types. Then our assertion immediately follows from   \lemref{propPhigen}.
\end{proof}

\section{SFT in the general case}\label{Secgeneral}
Let $r\geq r_c$.  Recall that  there is a canonical embedding $B_{r_c}(m-2n)\hookrightarrow B_{r}(m-2n)$ (see \rmkref{rmk:embed}). In  this section, we shall show that  the embedding of $\Ker F_{r_c}^{r_c}$ generates $\Ker F_r^r$ as a $2$-sided ideal of $B_r(m-2n)$.

\subsection{Module structure of $\Ker F_r^r$}\label{SecAlgStrKer}
We  shall prove  that $\Ker F_r^r$  is singly generated as left $\CA_{2r}$-module. We need a lemma on product of Young symmetrisers as follows.
\begin{lem}\label{lemProdYSym}
	\cite[Theorem 1.1]{R}
	Given positive integers $n\leq k$, let $k\dashv \la \supseteq\mu\vdash n$ be partitions, and let $\ft$ be a Young tableau of shape $\la$ containing a Young subtableau $\fs$ of shape $\mu$. There exists a subset of permutations $L(\ft;\fs)\subset \Sym_k$  and  $m_{\sigma}\in \Q$ such that
	$$\fc_{\la}(\ft)\fc_{\mu}(\fs)=\fc_{\la}(\ft)\sum_{\sigma\in L(\ft;\fs)}m_{\sigma}\sigma\neq 0.$$	
\end{lem}

\begin{lem}\label{lemfcexp}
	Let $r\geq r_c$, $\la\vdash 2r$ and $\la\supseteq\la_c$. Assume that $\fs$ is the standard tableau of shape $\la_c$ lying in the top left corner of $\ft^{\la}$. Then
	$\fc_{\la}=\fc_{\la_c}(\fs)\alpha(\ft^{\la};\fs),$
	where $\alpha(\ft^{\la};\fs)\in \CA_{2r}$, and $\fc_{\la}$ and $\fc_{\la_c}(\fs)$ are  Young symmetrisers associated with $\ft^{\la}$ and $\fs$ respectively.
\end{lem}
\begin{proof}	
	Using the resolution of identity \eqref{eqidenres}, we obtain
	\begin{equation*}
	\begin{aligned}
	\fc_{\la_{c}}(\fs)\CA_{2r}&=\biggl(\sum_{\la\vdash 2r}\sum_{\ft\in \Tab(\la)} h_{\la}^{-2}  \fc_{\la}(\ft)\biggl)\fc_{\la_{c}}(\fs)\CA_{2r}
	=\bigoplus_{(\la,\ft)\in \La}h_{\la}^{-2} \fc_{\la}(\ft)\fc_{\la_c}(\fs)\CA_{2r},
	\end{aligned}
	\end{equation*}
	where $\La=\{(\la,\ft)\;|\; \fc_{\la}(\ft)\fc_{\la_{c}}(\fs)\neq 0,\; \la\vdash 2r, \ft\in \Tab(\la)\}$. As $\fc_{\la}(\ft)$ is a projection operator mapping $\CA_{2r}$ onto a simple module isomorphic to the right Specht module $\fc_{\la}(\ft)\CA_{2r}$, we have $\fc_{\la}(\ft)\fc_{\la_c}(\fs)\CA_{2r}\cong \fc_{\la}(\ft)\CA_{2r}$ if $\fc_{\la}(\ft)\fc_{\la_{c}}(\fs)\neq 0$.  This reduces the above equation to
	\begin{equation}\label{equfcsum}	
	\fc_{\la_{c}}(\fs)\CA_{2r}=\bigoplus_{(\la,\ft)\in \La}h_{\la}^{-2} \fc_{\la}(\ft)\CA_{2r}.
	\end{equation}
	Using assumptions on $\la$ and $\fs$, we deduce from  \lemref{lemProdYSym} that $\fc_{\la}\fc_{\la_c}(\fs)\neq 0$, and hence by \eqref{equfcsum}
	$\fc_{\la}\CA_{2r}$ is a direct summand of the right hand side of \eqref{equfcsum}. In particular, $\fc_{\la}$ has an expression of the required form.
\end{proof}

Now we arrive at the following proposition.
\begin{prop}\label{propKerstr}
	Let $r\geq r_c$.  $\Ker F_r^r$ is singly  generated  by $\hat{\Phi}\ot I_{r-r_c}$ as left $\CA_{2r}$-module, i.e.,
	\begin{equation}\label{eqKerstr}
	\Ker F_r^r=\bU_r(\hat{A}^r\circ \hat{\fc}_{\la_c}\CA_{2r})=\CA_{2r}\ast(\hat{\Phi}\ot I_{r-r_c}).
	\end{equation}
\end{prop}
\begin{proof}
	Since $\bU_r(\hat{A}^r\circ \hat{\fc}_{\la_c})= \hat{\Phi}\ot I_{r-r_c}$ and $\bU_r$ is an $\CA_{2r}$-isomorphism, it is enough to prove the first equation of \eqref{eqKerstr}. Clearly,  $\bU_r(\hat{A}^r\circ \hat{\fc}_{\la_c}\CA_{2r})\subseteq \Ker F_r^r$ by \lemref{lemKer}, so we only need to prove the converse.  
	
	By \lemref{lemKer}, we need to show that $\bU_r(\hat{A}^{ r}\circ\hat{\fc}_{\la})$  as left $\CA_{2r}$-module is generated by $\bU_r(\hat{A}^r\circ \hat{\fc}_{\la_c})$ for any $\tcp\ni \la \supseteq \la_c$.
	Let $\fs$ be the subtableau of shape $\la_c$ in the top-left corner of $\ft^{\la}$. It follows from \lemref{lemfcexp} that
	\begin{equation}\label{eqnbuac}
	\bU_r(\hat{A}^{ r}\circ\hat{\fc}_{\la})=\bU_r(\hat{A}^{ r}\circ \hat{\fc}_{\la_c}(\fs)\alpha(\ft^{\la};\fs))=\alpha(\ft^{\la};\fs)^{\sharp}\ast \bU_r(\hat{A}^{ r}\circ \hat{\fc}_{\la_c}(\fs)).
	\end{equation}
	On the other hand, it is clear from the diagram that there exists $h\in H_r$ (see \exref{examexofh} below) such that
	$
	\bU_r(\hat{A}^{ r}\circ \hat{\fc}_{\la_c}(\fs))=h\ast \bU_r(\hat{A}^{ r}\circ \hat{\fc}_{\la_c}),$
	Using this in \eqnref{eqnbuac},  we complete the proof.
\end{proof}

\begin{exam}\label{examexofh}
	{\rm
		In the case of $\OSp(1|2)$, let $r=5$ and $\la=(6,4)\supset\la_c=(4,4)$.	Then
		$$
		\begin{aligned}
		\ft^{\la}=\;&
		\begin{ytableau}
		1&2&3&4&5&6\\
		7&8&9&10	
		\end{ytableau}\;,
		\end{aligned}
		\quad
		\begin{aligned}
		\ft^{\la_c}=\;&
		\begin{ytableau}
		1&2&3&4\\
		5&6&7&8	
		\end{ytableau}\;,
		\end{aligned}
		\quad
		\begin{aligned}
		\fs=\;&
		\begin{ytableau}
		1&2&3&4\\
		7&8&9&10	
		\end{ytableau}\;.
		\end{aligned}		
		$$   
		Now $\fc_{\la_c}$ and $\fc_{\la_c}(\fs)$ are Young symmetrisers associated with $\ft^{\la_c}$ and $\fs$, respectively. One can easily verify that 
		$$
		\bU_r(\hat{A}^{r}\circ \hat{\fc}_{\la_c}(\fs))=(5,7,9)\circ\bU_r(\hat{A}^{ r}\circ \hat{\fc}_{\la_c})\circ (6,10,8)=h\ast\bU_r(\hat{A}^{ r}\circ \hat{\fc}_{\la_c}),
		$$
		where $h=(5,7,9)(6,8,10)\in H_r$.
	}
\end{exam}

\subsection{Generators of $\Ker F_r^r$ for arbitrary $r$}\label{Secgen}
We shall show that  the canonical embedding of $\Ker F_{r_c}^{r_c}$ in $B_r(m-2n)$ generates $\Ker F_r^r$. For our purpose, we need the following new diagrams.

\begin{defn}\label{defnDij1}
	Assume that $r\geq r_c$. Let $D\in B_{r_c}^{r_c}(m-2n)$ be any Brauer $(r_c,r_c)$-diagram. Given  sequences $\bi,\bj$  with $\ell(\bi)=k$ and $\ell(\bj)=l$, we define a new Brauer $(r,r)$-diagram by
	\[ D(\bi,\bj):=
	\begin{cases}
	D^{\bi}_{\bj}\ot A_0^{\ot (l-k)}\ot I_{r-(r_c-k+l)},& \text{if $k\leq l$},\\
	D^{\bi}_{\bj}\ot U_0^{\ot (k-l)}\ot I_{r-(r_c+k-l)},& \text{if $k\geq l$},
	\end{cases} 
	\]
	which can be respectively represented as in \figref{FigPhi1} and \figref{FigPhi2}. Here $A_0$ and $U_0$ are elementary Brauer diagrams introduced in  \secref{SecBracat}. 
\end{defn}

\begin{figure}[h]
	\begin{center}
		\begin{picture}(250, 70)(-5,-30)
		
		\put(2,20){\line(0,1){25}}
		\put(78,20){\line(0,1){25}}
		\qbezier(25,20)(95,80)(120,-25)
		\qbezier(60,20)(90,60)(100,-25)
		\put(20,23){\tiny $i_1$}
		\put(55,23){\tiny $i_k$}
		\put(38,23){\dots}
		\put(102,-22){\dots}
		\put(108,-28){{\tiny $k$}}

		\put(0, 0){\line(1, 0){80}}
		\put(0, 0){\line(0, 1){20}}
		\put(80, 0){\line(0, 1){20}}
		\put(0, 20){\line(1, 0){80}}
		\put(35, 6){$D$}
		
		\put(2,0){\line(0,-1){25}}
		\put(78,0){\line(0,-1){25}}
		\qbezier(25,0)(95,-60)(180,45)
		\qbezier(60,0)(90,-50)(120,45)
		\put(20,-7){\tiny $j_1$}
		\put(55,-7){\tiny $j_l$}
		\put(38,-5){\dots}
		\put(137,30){\dots}
		\put(143,38){{\tiny $l$}}
		\qbezier(125,-25)(135,10)(145,-25)
		\qbezier(160,-25)(170,10)(180,-25)
		\put(146,-15){\dots}
		\put(143,-8){{\tiny $l-k$}}
		
		\put(182,6){$\ot$}
		\put(195,6){$I_{r-(r_c-k+l)}$}
		
		\end{picture}
	\end{center}
	\caption{$D(\bi,\bj),k\leq l$}
	\label{FigPhi1}
\end{figure}
\begin{figure}[h]
	\begin{center}
		\begin{picture}(250, 70)(-5,-30)
		
		\put(2,20){\line(0,1){25}}
		\put(78,20){\line(0,1){25}}
		\qbezier(25,20)(95,70)(180,-25)
		\qbezier(60,20)(90,60)(120,-25)
		\put(20,23){\tiny $i_1$}
		\put(55,23){\tiny $i_k$}
		\put(38,23){\dots}
		\put(102,36){\dots}
		\put(108,42){{\tiny $l$}}
		\qbezier(125,45)(135,10)(145,45)
		\qbezier(160,45)(170,10)(180,45)
		\put(146,35){\dots}
		\put(143,22){{\tiny $k-l$}}
		
		\put(0, 0){\line(1, 0){80}}
		\put(0, 0){\line(0, 1){20}}
		\put(80, 0){\line(0, 1){20}}
		\put(0, 20){\line(1, 0){80}}
		\put(35, 6){$D$}
		
		\put(2,0){\line(0,-1){25}}
		\put(78,0){\line(0,-1){25}}
		\qbezier(25,0)(95,-60)(120,45)
		\qbezier(60,0)(90,-40)(100,45)
		\put(20,-7){\tiny $j_1$}
		\put(55,-7){\tiny $j_l$}
		\put(38,-5){\dots}
		\put(137,-18){\dots}
		\put(143,-26){{\tiny $k$}}
		
		\put(182,6){$\ot$}
		\put(195,6){$I_{r-(r_c+k-l)}$}
		
		\end{picture}
	\end{center}
	\caption{$D(\bi,\bj),k\ge l$}
	\label{FigPhi2}
\end{figure}

In particular, if $\bi$ and $\bj$ are of the same length, we have by definition $D(\bi,\bj)=D^{\bi}_{\bj}\ot I_{r-r_c}$, where  $D^{\bi}_{\bj}$ is a Brauer  $(r_c,r_c)$-diagram. The following lemma tells us that for any two sequences $\bi$ and $\bj$ the diagram $D(\bi,\bj)$ is actually inside the 2-sided ideal of $B_r(m-2n)$ generated by $D(\bi^{\prime},\bj^{\prime})$ for some $\bi^{\prime}$ and $\bj^{\prime}$ of equal length.

\begin{lem} \label{LemPhi}
	Let $\bi$ and $\bj$ be any two sequence as above with $\ell(\bi)=k$ and $\ell(\bj)=l$.
	\begin{enumerate}[(1)]
		\item If $k\leq l$, then $D(\bi,\bj) \in  \langle D(\bi,\bj^{\prime}) \rangle_r $, where $\bj^{\prime}=(j_{l-k+1},j_{l-k+2},\dots, j_l)$ with $\ell(\bi)=\ell(\bj^{\prime})=k$.
		\item If $k \ge l$, then $D(\bi,\bj)\in \langle D(\bi^{\prime},\bj) \rangle_r$, where $\bi^{\prime}=(i_{k-l+1},i_{k-l+2},\dots, i_k)$ with $\ell(\bi^{\prime})=\ell(\bj)=l$.
	\end{enumerate}	
\end{lem}
\begin{proof}
	We only prove part (1), since part (2) can be proved similarly. It suffices to show that there exists a Brauer $(r,r)$-diagram $D^{\prime}$ such that $D(\bi,\bj)=D(\bi,\bj^{\prime})\circ D^{\prime}$, where $D(\bi,\bj^{\prime})=D^{\bi}_{\bj^{\prime}}\ot I_{r-r_c}$ with $D^{\bi}_{\bj^{\prime}}$ taken as in the lemma. Pictorially, we can separate $D^{\prime}$ from $D(\bi,\bj)$, which consists of $(l-k)$ top (bottom)  horizontal edges as shown below:
	
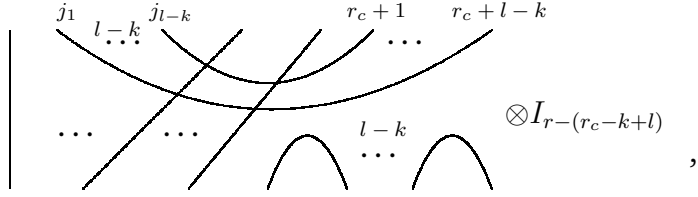
\begin{center}
		\begin{picture}(210, 70)(0,0)
		\put(2,0){\line(0,1){60}}
		\qbezier(20,60)(100,0)(185,60)
		\qbezier(60,60)(100,20)(140,60)
		\put(20,65){\tiny $j_1$}
		\put(55,65){\tiny $j_{l-k}$}
		\put(130,65){\tiny $r_c+1$}
		\put(170,65){\tiny $r_c+l-k$}
		\put(34,58){{\tiny $l-k$}}				
		\put(38,55){\dots}
		\put(20,20){\dots}
		\put(60,20){\dots}
		\put(145,55){\dots}
		\qbezier(90,60)(60,30)(30,0)
		\qbezier(120,60)(95,30)(70,0)
		
		\qbezier(100,0)(115,40)(130,0)
		\qbezier(155,0)(170,40)(185,0)
		\put(135,20){{\tiny $l-k$}}
		\put(135,12){\dots}		
		\put(190,25){$\ot I_{r-(r_c-k+l)}$}
		\put(260,10){,}
		\end{picture}
	\end{center}
	where  all vertical edges are strands with no crossings. This completes our proof.
\end{proof}

$D(\bi,\bj)$ can also be described by using group action. For simplicity, we identify $\Sym_{2r}$ as the symmetric group on the labelling set $\{1,2\dots,r,\bar{1},\dots,\bar{r}\}$ of $(r,r)$-Brauer diagram, and $H_r\subset \Sym_{2r}$ is the parabolic subgroup defined as in \eqref{eq:H_r}.

\begin{lem}\label{lemDijorbit}
	For any two sequences $\bi$ and $\bj$ with $\ell(\bi)=k\leq l=\ell(\bj)$, we have $$D(\bi,\bj)\in [\tau\ast (D\ot I_{r-r_c})]_{H_{r_c}},$$  where $\tau=\prod_{s=1}^{k}(i_s,j_s) \prod_{t=1}^{l-k}(r_c+t,j_{k+t})\in \Sym_{2r}$.
\end{lem}
\begin{proof}
	Let $\bi^{\prime}=(i_1,i_2,\dots,i_k,r_c+1,r_c+2,r_c+l-k)$ and $\bj^{\prime}=(j_1,j_2,\dots,j_l)$, then we have $\tau=(\bi^{\prime},\bj^{\prime})$. Pictorially, it is easily verified that
	$D(\bi^{\prime},\bj^{\prime})\in  [(D\ot I_{r-r_c})^{\bi^{\prime}}_{\bj^{\prime}}]_{H_r}$, and hence the lemma follows since $(D\ot I_{r-r_c})^{\bi^{\prime}}_{\bj^{\prime}}\in [\tau\ast (D\ot I_{r-r_c})]_{H_r}$.
\end{proof}
\begin{rmk}
	Analogous arguments go through for $D(\bi,\bj)$ with $\ell(\bi)=k\geq l=\ell(\bj)$. In this case, we have $D(\bi,\bj)\in [\tau^{\prime}\ast (D\ot I_{r-r_c})]_{H_r}$, where $\tau^{\prime}=\prod_{s=1}^{l}(i_s,j_s) \prod_{t=1}^{k-l}(\overline{r_c+t},i_{l+t})\in \Sym_{2r}$
	The element $\tau$ in \lemref{lemDijorbit} is not unique. Instead, we can replace $r_c+1, r_c+2,\dots, r_c+l-k$ by any $l-k$ increasing integers ranging from $r_c+1$ to $r$.
\end{rmk}

Notice that  \defref{defnDij1} can be extended linearly, i.e., if $\Psi=\sum a_D D$ is a finite sum  with $a_D\in \C$, then  $\Psi(\bi,\bj):=\sum a_D D(\bi,\bj)$. In particular, this applies to the element $\widetilde{\Phi} \in \Ker F_{r_c}^{r_c}$. In this case, both \lemref{LemPhi} and \lemref{lemDijorbit}  remain valid.

\begin{lem} \label{LemPhiIJ}
	Let $r\geq r_c$.  The element $\widetilde{\Phi} (\bi,\bj)$ belongs to  $\Ker F_r^r$ for any two sequences $\bi$ and $\bj$. 
\end{lem}
\begin{proof}
	Using \lemref{LemPhi} for $\widetilde{\Phi} (\bi,\bj)$, we have $\widetilde{\Phi} (\bi,\bj)\in \langle \widetilde{\Phi} (\bi^{\prime},\bj^{\prime})\rangle_r$ for some $\bi^{\prime}$ and $\bj^{\prime}$ with $\ell(\bi^{\prime})=\ell(\bj^{\prime})\leq r_c$, so it suffices to  show that $\widetilde{\Phi} (\bi^{\prime},\bj^{\prime})\in \Ker F_r^r$. In this case we have
	$\widetilde{\Phi} (\bi^{\prime},\bj^{\prime})=\widetilde{\Phi} _{\bj^{\prime}}^{\bi^{\prime}}\ot I_{r-r_c}$. Since $\widetilde{\Phi} _{\bj^{\prime}}^{\bi^{\prime}}\in \Ker F_{r_c}^{r_c}$ by \propref{PropPsiij}, it follows that $\widetilde{\Phi} (\bi^{\prime},\bj^{\prime})$ is simply the canonical embedding  of  $\widetilde{\Phi} _{\bj^{\prime}}^{\bi^{\prime}}$ into $\Ker F_r^r$.
\end{proof}

\begin{scho}\label{schoPhiij}
	Given any two sequences  $\bi$ and $\bj$, we have $\widetilde{\Phi} (\bi,\bj)\in \langle \widetilde{\Phi} _{\bj^{\prime}}^{\bi^{\prime}}\ot I_{r-r_c}\mid   \ell(\bi^{\prime})=\ell(\bj^{\prime})\leq r_c\rangle_r$, i.e.,  $\widetilde{\Phi} (\bi,\bj)$ lies in the 2-sided ideal of $B_r(m-2n)$ generated by $\Ker F_{r_c}^{r_c}\ot I_{r-r_c}$.
\end{scho}

\lemref{LemPhiIJ} can be significantly improved in the following key lemma.

\begin{lem}\label{lemKergen}
	Let $r\geq r_c$. Then $\Ker F_r^r$ is generated by all the elements $\widetilde{\Phi} (\bi,\bj)$ as a 2-sided ideal of $B_{r}(m-2n)$. 
\end{lem}
\begin{proof}
	By \propref{propKerstr}, $\Ker F_r^r$ is a singly generated  $\CA_{2r}$-module. Hence we only need to show that $\sigma\ast (\widetilde{\Phi} \ot I_{r-r_c})$ is generated by $\widetilde{\Phi} (\bi,\bj)$ for some $\bi$ and $\bj$, where $\sigma\in \Sym_{2r}$ is any permutation on the labelling set $\{1,2,\dots,r,\bar{1},\bar{2},\dots, \bar{r}\}$.
	
	This can be reduced to the following situation. Noting that $\sigma\in \Sym_{2r}$ belongs to some right coset $H_r\tau$, we take the coset representative  $\tau$ to be   $\tau=(\bi^{\prime},\bj^{\prime})$, where $\bi^{\prime}$ and $\bj^{\prime}$ are respectively sequences of equal length such that $1\leq  i_1<i_2<\dots<i_p\leq r $ and $\bar{1}\leq  j_1<j_2<\dots<j_p\leq \bar{r}$ with $p\in \N$. Therefore, we have $\sigma\ast (\widetilde{\Phi} \ot I_{r-r_c})\in [ \tau\ast (\widetilde{\Phi} \ot I_{r-r_c})]_{H_r}$ and hence it is enough to show that $\tau\ast (\widetilde{\Phi} \ot I_{r-r_c})$ is generated by $\widetilde{\Phi} (\bi,\bj)$ for some $\bi$ and $\bj$.
	
	Now let $k$ and $l$ be the largest numbers such that $i_k\leq r_c$ and $j_l\leq \overline{r_c}$. Without loss of generality, we may assume $k\leq l$. This assumption implies that there are at least $l-k$ bottom horizontal edges in the diagram of $\tau\ast (\widetilde{\Phi} \ot I_{r-r_c})$. Therefore, we can further assume $i_{k+1}=r_c+1, i_{k+2}=r_c+2,\dots, i_l=r_c+l-k$ and $j_{l+1}>\overline{r_c+l-k}$. This is depicted in the following diagram:
	
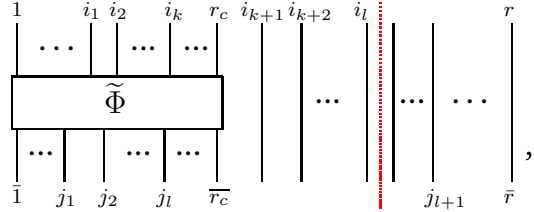
\begin{center}
		\begin{picture}(210, 80)(-5,-30)
		\put(2,20){\line(0,1){20}}	
		\put(30,20){\line(0,1){20}}
		\put(40,20){\line(0,1){20}}
		\put(60,20){\line(0,1){20}}
		\put(78,20){\line(0,1){20}}
		\put(0,43){\tiny $1$}   	   
		\put(27,43){\tiny $i_1$}
		\put(37,43){\tiny $i_2$}
		\put(58,43){\tiny $i_k$}
		\put(75,43){\tiny $r_c$}
		\put(10,30){\dots}
		\put(45,30){...}
		\put(65,30){...}
		
		\put(0, 0){\line(1, 0){80}}
		\put(0, 0){\line(0, 1){20}}
		\put(80, 0){\line(0, 1){20}}
		\put(0, 20){\line(1, 0){80}}
		\put(35, 6){$\widetilde{\Phi} $}

		\put(2,0){\line(0,-1){20}}	
		\put(20,0){\line(0,-1){20}}
		\put(35,0){\line(0,-1){20}}
		\put(58,0){\line(0,-1){20}}
		\put(78,0){\line(0,-1){20}}
		\put(0,-28){\tiny $\bar{1}$}   	   
		\put(17,-28){\tiny $j_1$}
		\put(33,-28){\tiny $j_2$}
		\put(55,-28){\tiny $j_l$}
		\put(75,-28){\tiny $\overline{r_c}$}
		\put(6,-10){...}
		\put(43,-10){...}
		\put(62,-10){...}   		
		
		\put(95,40){\line(0,-1){60}}
		\put(110,40){\line(0,-1){60}}
		\put(135,40){\line(0,-1){60}}
		\put(87,43){\tiny $i_{k+1}$}
		\put(105,43){\tiny $i_{k+2}$}
		\put(130,43){\tiny $i_l$}
		\put(115,10){...}
		
		{\color{red} \qbezier[60](140,48)(140,0) (140,-30)}
		
		\put(145,40){\line(0,-1){60}}
		\put(160,40){\line(0,-1){60}}
		\put(190,40){\line(0,-1){60}}
		\put(157,-28){\tiny $j_{l+1}$}
		\put(147,10){...}
		\put(167,10){\dots}
		\put(187,43){\tiny $r$}
		\put(187,-28){\tiny $\bar{r}$}
		\put (195,-10){,}
		\end{picture}
	\end{center}
	where the subdiagram on the left of the dotted line is a Brauer $(r_c+l-k,r_c+l-k)$-diagram.
	Under our assumption, we have $$
	\begin{aligned}
	\tau\ast (\widetilde{\Phi} \ot I_{r-r_c})&=\prod_{s=1}^{l}(i_s,j_s)\prod_{t=l+1}^{p}(i_t,j_t)\ast (\widetilde{\Phi} \ot I_{r-r_c})\\
	&=\left(\prod_{s=1}^{l}(i_s,j_s)\ast (\widetilde{\Phi} \ot I_{r-r_c}) \right)\circ \left(\prod_{t=l+1}^{p}(i_t,j_t)\ast I_{r}\right). 
	\end{aligned}
	$$
	It follows that $\tau\ast (\widetilde{\Phi} \ot I_{r-r_c})$ is a composition of two Brauer $(r,r)$-diagrams. The first Brauer diagram, $\prod_{s=1}^{l}(i_s,j_s)\ast (\widetilde{\Phi} \ot I_{r-r_c})$, by \lemref{lemDijorbit} is generated by $\widetilde{\Phi}(\bi,\bj)$ with $\bi=(i_1,\dots,i_k)$ and $\bj=(j_1,\dots,j_l)$. Therefore,  $\tau\ast (\widetilde{\Phi} \ot I_{r-r_c})$ is generated by $\widetilde{\Phi}(\bi,\bj)$. 	
\end{proof}

We now state the main result of of this paper.

\begin{thm}\label{thm:SFT-refine}
	Suppose that $\sdim(V)=(m|2n)$. Let 
	$
	F_r^r: B_r(m-2n)\ra \End_{\OSp(V)}(V^{\ot r})
	$ be the surjective algebra homomorphism. If $r<r_c$, then $F_r^r$ is an isomorphism. If $r\geq r_c$, then $\Ker F_r^r$ as a $2$-sided ideal of $B_r(m-2n)$ is generated by the  set $\{ \widetilde{\Phi}_{\bj}^{\bi}\ot I_{r-r_c} \mid \{\bi,\bj\}\in \CI_{r_c}\}.$
\end{thm}
\begin{proof}
	It is enough to  prove that $\Ker F_r^r$ as a $2$-sided ideal of $B_r(m-2n)$ is generated by $\Ker F_{r_c}^{r_c}\ot I_{r-r_c}$ when $r\geq r_c$. By \lemref{lemKergen}, $\Ker F_r^r$ is generated by elements of the form $\widetilde{\Phi} (\bi,\bj)$, which by \schref{schoPhiij} are generated by $\Ker F_{r_c}^{r_c}\ot I_{r-r_c}$. Now the theorem follows from \thmref{thmKermini}.
\end{proof}

\subsection{An application to the orthosymplectic Lie superalgebra}
We now turn to the category of finite dimensional representations of the
orthosymplectic Lie superalgebra $\osp(V)$ \cite{K}. 

We refer to \cite[Section 5]{LZ5} for some basic facts on $\osp(V)$-invariants in the tensor powers $V^{\ot r}$. Particularly, we want to mention that there exists a distinguished $\osp(V)$-invariant in  the tensor power $V^{m(2n+1)}$ with $m>0$ and $n\geq 0$, which is called super Pfaffian and is denoted by $\Omega$. Recall that the non-degenerate bilinear form $(-,-)$ on $V$ provides a canonical isomorphism of $\osp(V)$-modules:
$V^{\ot 2r}\rightarrow \Hom_{\C}(V^{\ot r}, V^{\ot r}).$
Therefore,  we have a distinguished element $E(\Omega)\in \End_{\osp(V)}(V^{\ot \frac{m(2n+1)}{2}})$ corresponding to the super Pfaffian $\Omega\in V^{\ot m(2n+1)}$ under the above isomorphism  if and only if $\frac{m(2n+1)}{2}\in \Z_{\geq 0}$.
This implies that there exists a distinguished element $E(\Omega)\ot \id_{V}^{\ot (r- \frac{m(2n+1)}{2})} \in \End_{\osp(V)}(V^{\ot r})$ if and only if $r-\frac{m(2n+1)}{2}\in\Z_{\ge 0}$. The distinguished element $E(\Omega)$ is not expressible in terms of Brauer diagrams, see, e.g., \cite[Theorem 5.2, Corollary 5.8]{LZ5}. 

\begin{prop}\cite[Corollary 5.9]{LZ5} \label{propEndosp}
	We have a canonical inclusion of associative algebras
	\[ \End_{\osp(V)}(V^{\ot r})\supseteq \End_{\OSp(V)}(V^{\ot r}).  \]	
	The equality holds if and only if $r-\frac{m(2n+1)}{2}\notin\Z_{\ge 0}$. If $r-\frac{m(2n+1)}{2}\in \Z_{\geq 0}$, then $\End_{\osp(V)}(V^{\ot r})$ is generated as associative algebra by $\End_{\OSp(V)}(V^{\ot r})$ together with $E(\Omega)\ot \id_{V}^{\ot (r- \frac{m(2n+1)}{2})}$.
\end{prop}

\begin{thm} \label{thm:osp-end} The Brauer algebra $B_r(m-2n)$ is isomorphic
	to the endomorphism algebra $\End_{\osp(V)}(V^{\ot r})$ if and only if
	\begin{enumerate}
		\item $m$ is an odd positive integer and $r<(m+1)(n+1)$; or
		\item $m$ is an even positive integer and $r<mn+\frac{m}{2}$.
	\end{enumerate}
\end{thm}
\begin{proof}
	We claim that $B_r(m-2n)\cong \End_{\osp}(V^{\ot r})$ if and only if the following conditions hold:
	\begin{enumerate}[(i)]
		\item $\End_{\osp(V)} (V^{\ot r})=\End_{\OSp(V)}(V^{\ot r})$; and
		\item $\End_{\OSp(V)}(V^{\ot r})\cong B_r(m-2n)$.
	\end{enumerate} 
	The ``if'' part is obvious. Assume that $B_r(m-2n)\cong \End_{\osp}(V^{\ot r})$. Then by \propref{propEndosp}  $\End_{\osp}(V^{\ot r})$ does not contain  the super Pfaffian $E(\Omega)\ot \id_{V}^{\ot (r- \frac{m(2n+1)}{2})}$, and hence we have $\End_{\osp}(V^{\ot r})=\End_{\OSp(V)}(V^{\ot r})\cong B_r(m-2n)$, completing the proof of ``only if'' part.
	
	Now the condition (i) holds if and only if $r-\frac{m(2n+1)}{2}\notin\Z_{\ge 0}$ by \propref{propEndosp}, which is certainly true when $m$ is odd and is equivalent to $r<mn+\frac{m}{2}$ when $m$ is even. Using  \thmref{thmKeriso}, we have condition (ii) if and only 	if $r<r_c=(m+1)(n+1)$. Therefore, both two conditions hold if and only if $r<(m+1)(n+1)$ when $m$ is  odd  and $r<\min(mn+\frac{m}{2}, (m+1)(n+1))=mn+\frac{m}{2}$ when $m>0$ is even.
\end{proof}

In particular, \thmref{thm:osp-end} recovers the well-known classical case $n=0$ (cf. \cite[Appendix F]{FH}). We have $B_r(2k)\cong \End_{\fso(V)}(V^{\ot r})$ if and only if $r<k$ and $B_r(2k+1)\cong \End_{\fso(V)}(V^{\ot r})$ if and only if $r<2k+1$. Here $V=\C^{2k}$ or $\C^{2k+1}$. 

\thmref{thm:osp-end} also considerably strengthens a result of  Ehrig and Stroppel in \cite[Theorem A]{ES}, in which it shows that if one of the following conditions holds 
\begin{itemize}
	\item $\sdim V\neq(2k|0)$ and $r\leq k+n$;
	\item $\sdim V=(2k|0)$ with $k>0$ and $r<k$,
\end{itemize}
then we have the isomorphism of algebras $B_r(m-2n)\cong \End_{\osp(V)}(V^{\ot r})$, where $\sdim V=(m|2n)$ with $m=2k$ or $2k+1$. The upper bound given in the first condition is smaller than that given in our theorem. The second condition agrees with the ``if'' part of the classical case ($n=0$) of our theorem.

\section{SFT for $\OSp(1|2n)$}\label{SecOSp12n}
We write $\OSp(V)$ as $\OSp(1|2n)$ when $\sdim V=(1|2n)$. In this case, we shall prove that $\Ker F_r^r$ is singly generated. An explicit formula for the generator will be constructed.

We start by a technical lemma which will be used later.
\begin{lem}\label{lemyongiden}
	Let $\la=(2^m)\,(m\in \Z_{>0})$ and $C_1$ be the first column of $\ft^{\la}$,  then we have the  identity
	$\fc_{\la}(12)\alpha^-(C_1)=(m-1)!\fc_{\la},$
	where $\alpha^-(C_1)=\sum_{\sigma\in \Sym\{C_1\}} \epsilon(\sigma)\sigma$.
\end{lem}
\begin{proof}
	Let $X=C_1=\{1,3,\dots,2m-1\}$ and $Y=\{2\}$, then $G_{X,Y}=(1)-\sum_{i=1}^{m}(2,2i-1)$ and  we have the Garnir relation
	$\fc_{\la}=\fc_{\la}\sum_{i=1}^{m}(2,2i-1)$.
	Using $\alpha^-(C_1)=\alpha^-(\{3,5,\dots,2m-1\})((1)- \sum_{i=2}^{m}(1,2i-1))$,
	we obtain
	$$
	\fc_{\la}(12)\alpha^-(C_1)=(m-1)!\fc_{\la}(12)((1)- \sum_{i=2}^{m}(1,2i-1))=(m-1)!\fc_{\la},
	$$
	where we have used $\fc_{\la}(12)(1,2i-1)=-\fc_{\la}(1,2i-1)(12)(1,2i-1)=-\fc_{\la}(2,2i-1)$ and the Garnir relation  in the last equation.
\end{proof}

Now we shall construct the generator of $\Ker F_{r_c}^{r_c}$ with $r_c=2(n+1)$, which by \thmref{thm:SFT-refine} will generates $\Ker F_r^r$ for $r\geq r_c$. Suppose that $\bi$ is a standard sequence of increasing type
$\ty(\bi)=(a_1,a_2,\dots,a_{n+1})$ with $\ell(\bi)=k\leq r_c$.  Then $0\leq a_p\leq 2$ for all $p$ in the present case.  We define the set of increasing types of standard sequences by
$$S(n+1):=\left\{(0^{n+1-k}1^{k})\mid 0\leq k\leq n+1\right\},$$
where for simplicity $(0^{n+1-k}1^{k})$ means that $0$ appears $n+1-k$ times and $1$ appears $k$ times. Then by \thmref{thmKermini}, $\Ker F_{r_c}^{r_c}$ is generated by elements $\widetilde{\Phi}^{\bi}_{\bj}$ with $\ty(\bi)=\ty(\bj)\in S(n+1)$.

We define the element $E:=\widetilde{\Phi}^{\bi}_{\bj}$, where $\bi$ and $\bj$ are standard sequences of increasing type $\ty(\bi)=\ty(\bj)=(1^{n+1})$. 
By \lemref{LemOrbPhi}, we have
\begin{equation}\label{OrbE}
\bU_{r_c}(\hat{A}^{ r_c}\circ \hat{\fc}_{\la_c}(w_{\mu_c}^{-1}(\bi),w_{\mu_c}^{-1}(\bj)))=\bU_{r_c}(\hat{A}^{ r_c}\circ \hat{\fc}_{\la_c} \prod_{s=1}^{n+1}(2s-1,2s)) \in [E]_{H_{r_c}},
\end{equation}
where $\la_c=(2n+2,2n+2)$. For instance, in  \exref{Examty}  $\bi=(1,5,9)$ and $\bj=(2,6,10)$ are the sequences we are taking (we need to change the labelling). In this case, we have
$w_{\mu_c}^{-1}(\bi)=(1,3,5)$ and $w_{\mu_c}^{-1}(\bj)=(2,4,6)$.
Our main theorem is as follows.

\begin{thm}\label{thmgen}
	Suppose that $G=\OSp(1|2n)$. If $r<2(n+1)$, the algebra homomorphism $F_r^r: B_r(1-2n)\rightarrow \End_{G}(V^{\ot r})$ is an isomorphism. If $r\geq 2(n+1)$, then $\Ker F_r^r$ is generated as a 2-sided ideal of $B_r(1-2n)$ by the element $E$.
\end{thm}

\begin{proof}
	We only need to prove the second assertion. For each $k$ such that $0\leq k\leq n+1$, let  $\bi$ and $\bj$ be standard sequences with increasing types $\ty(\bi)=\ty(\bj)=(0^{n+1-k}1^k)$. It is enough  to prove that $\widetilde{\Phi}^{\bi}_{\bj} \in \langle E\rangle_{r_c}$. 
	Let $C_{2t-1}$ be the $(2t-1)$-th column of $\ft^{\la_c}$ for all $1\leq t\leq n+1-k$. Then  we have
	\begin{equation}\label{eqnPhiorb2}
	\begin{aligned}
	&\prod_{t=1}^{n+1-k}\alpha^-(C_{2t-1})\ast\bU_{r_c}(\hat{A}^{ r_c}\circ \hat{\fc}_{\la_c}\prod_{s=1}^{n+1}(2s-1,2s))\\
	=&\bU_{r_c}(\hat{A}^{ r_c}\circ \hat{\fc}_{\la_c}\prod_{s=1}^{n+1}(2s-1,2s)\prod_{t=1}^{n+1-k}\alpha^-(C_{2t-1})).\\
	\end{aligned}	
	\end{equation}
	By \lemref{lemyongiden} ($m=2$ case), we have 
	$ \hat{\fc}_{\la_c} \prod_{s=1}^{n+1-k}(2s-1,2s) \prod_{t=1}^{n+1-k}\alpha^-(C_{2t-1})= \hat{\fc}_{\la_c}.$
	Using this in \eqref{eqnPhiorb2} and by \lemref{LemOrbPhi}, we obtain
	\[
	\begin{aligned}
	&\prod_{t=1}^{n+1-k}\alpha^-(C_{2t-1})\ast\bU_{r_c}(\hat{A}^{ r_c}\circ \hat{\fc}_{\la_c}\prod_{s=1}^{n+1}(2s-1,2s))\\
	=&   \bU_{r_c}(\hat{A}^{ r_c}\circ \hat{\fc}_{\la_c} \prod_{s=n+2-k}^{n+1}(2s-1,2s)) \in [\widetilde{\Phi}^{\bi}_\bj]_{H_{r_c}}.
	\end{aligned}
	\]
	Since $\prod_{t=1}^{n+1-k}\alpha^-(C_{2t-1})$ belongs to $\C H_{r_c}$, by \eqref{OrbE} we conclude that  $\widetilde{\Phi}^{\bi}_{\bj}\ \in \langle E\rangle_{r_c}$.
\end{proof}

\begin{rmk}
	The generator $E$ is not an idempotent or quasi idempotent, even although its is constructed from a quasi
	idempotent $\widetilde{\Phi}$ (see \lemref{propidemp}). One can verify this directly in the case of $\OSp(1|2)$. This is very different from the case of classical groups \cite{LZ1, LZ4}.
\end{rmk}

\section{Applications to the orthogonal and symplectic groups}\label{Secclassical}
As an application of results obtained in previous sections, we re-derive the main results of \cite{LZ1, LZ4} on SFTs of the orthogonal and symplectic groups over $\C$. Our treatment here will be uniform and more conceptual.

\subsection{The symplectic group}
We take $V=V_{\bar{1}}$ to be purely odd with $\dim  V_{\bar{1}}=2n$. Then $m=0$, and the  non-degenerate  bilinear form $(-,-)$ on $V$ is skew-symmetric. Hence the isometry group of this form is the symplectic group $\Sp(2n)$.
Applying \corref{corofft}, we deduce that
$$F_r^r: B_r(-2n)\longrightarrow \End_{\Sp(2n)}(V^{\ot r})$$
is a surjective algebra homomorphism. Since $r_c=n+1$, we immediately obtain the following result from \thmref{thmKeriso}.
\begin{lem}\label{lemKermini}
	$\Ker F_r^r\neq 0$ if and only if $r\geq n+1$. Furthermore, $\Ker F_{n+1}^{n+1}\cong S^{(2n+2)}$.
\end{lem}

Maintaining the notation in \secref{SecPhi}, we have $\la_c=(2n+2), \mu_c=(n+1)$ and  $w_{\mu_c}=(1)$.
Therefore,  the nonzero element $\widetilde{\Phi}\in \Ker F_{n+1}^{n+1}$ defined in \eqnref{EqnwPhi} now reads
$$\widetilde{\Phi}=B(n+1)=\Phi_{LZ}(n+1)\in B_{n+1}(-2n),$$
This $\widetilde{\Phi}$  exactly  coincides with the element $\Phi$ defined by Lehrer and Zhang in \cite[Section 5]{LZ4}.
We also immediately recover \cite[Lemma 5.3]{LZ4}.
\begin{lem}\cite[Lemma 5.3]{LZ4}\label{lemProPhi} The element $\widetilde{\Phi}\in \Ker F_{n+1}^{n+1} $ has the following properties:
	\begin{enumerate}[(1)]
		\item $e_i\widetilde{\Phi}=\widetilde{\Phi} e_i=0$ for all $e_i\in B_{n+1}(-2n)$;
		\item $\widetilde{\Phi}^2=(n+1)!\widetilde{\Phi}$, hence $E:=\frac{1}{(n+1)!}\widetilde{\Phi}$ is an idempotent.
	\end{enumerate}	
\end{lem}
\begin{proof}
	Part (1) follows from \propref{AnnThm}.  Part (2) is immediate by \lemref{propidemp}.
\end{proof}

We now easily recover the second fundamental theorem of invariant theory for symplectic groups proved in \cite[Theorem 5.9]{LZ4}.

\begin{thm}\cite[Theorem 5.9]{LZ4}
	The algebra homomorphism $F_r^r: B_r(-2n)\rightarrow \End_{\Sp(2n)}(V^{\ot r})$ is an isomorphism if $r\leq n$. If $r\geq n+1$, then $\Ker F_r^r$  is generated as a 2-sided ideal of $B_r(-2n)$ by the idempotent $E=\frac{1}{(n+1)!}\widetilde{\Phi}$.
\end{thm}
\begin{proof}
	It follows from \thmref{thm:SFT-refine} that $\Ker F_r^r$ is generated by the generators of $\Ker F_{n+1}^{n+1}$. By  \lemref{lemProPhi} we have nonzero idempotent $E=((n+1)!)^{-1}\widetilde{\Phi}\in \Ker F_{n+1}^{n+1}$, which generates the kernel $\Ker F_{n+1}^{n+1}$ since  $\dim \Ker F_{n+1}^{n+1}=1$ by \lemref{lemKermini}.
\end{proof}

\subsection{The orthogonal group}
We take $V=V_{\bar{0}}$ to be purely even with $\dim V_{\bar{0}}=m$. Then $n=0$,
and the nondegenerate bilinear form $(-,-)$ on $V$ is symmetric. The   isometry group of the form is $\Or(m)$.  Applying \corref{corofft} again, we obtain the surjective algebra homomorphism
$$
F_r^r: B_r(m)\longrightarrow \End_{\Or(m)}(V^{\ot r}).
$$
Now $r_c=m+1$. It follows from \thmref{thmKeriso} that
\begin{lem}\label{lemKermini2}
	$\Ker F_r^r\neq 0$  if and only if $r\geq m+1$. Furthermore, $\Ker F_{m+1}^{m+1}\cong S^{(2^{m+1})}$.
\end{lem}

In the present case, $\la_c=(2^{m+1})$, $\mu_c=(1^{m+1})$ and $w_{\mu_c}=(1)$. By \lemref{lemKermini2}, we have
$$\dim  \Ker F_{m+1}^{m+1}=\dim  S^{(2^{m+1})}=\frac{(2m+2)(2m+1)\cdots(m+3)}{(m+1)!}.
$$
This agrees with the result of \cite{DH,HX}.
The nonzero element $\widetilde{\Phi}$ in \eqref{EqnwPhi}  has the form
$$\widetilde{\Phi}=(m+1)!A(m+1)=(m+1)!\sum_{\sigma\in \Sym_{m+1}}\epsilon(\sigma)\sigma.$$
For convenience, we shall omit the scalar multiple and write $\widetilde{\Phi}:=A(m+1)$ instead.

Suppose that $\bi$ and $\bj$ are standard sequences with $\ell(\bi)=\ell(\bj)=k$ $(0\leq k\leq m+1)$,  which in the present case are given by $\bi=(1,2,\dots,k)$ and $\bj=(\overline{1},\overline{2},\dots,\overline{k})$.  Hence by using symmetry properties of $\widetilde{\Phi}$,  we can depict $\widetilde{\Phi}^{\bi}_{\bj}$ pictorially as
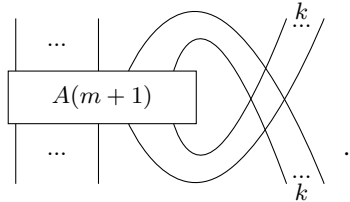
\begin{center}
	\begin{tikzpicture}[font=\scriptsize]
	\draw (0,0) rectangle (2.5,0.7);
	\draw (0.1,0) -- (0.1,-0.8);
	\draw (0.1,0.7) -- (0.1,1.4);
	\draw (1.2,0) -- (1.2,-0.8);
	\draw (1.2,0.7) -- (1.2,1.4);
	
	\node at (1.25,0.35){$A(m+1)$};
	\node at (0.65,1.05){$...$};  	
	\node at (0.65,-0.4){$...$};
	\node at (3.9,1.3){$...$};
	\node at (3.9,-0.7){$...$};
	\node at (3.9,1.5){$k$};
	\node at (3.9,-0.9){$k$};	 			
	
	\draw  plot  [smooth,tension=1] coordinates{(1.6,0.7) (2.8,1.4) (4.2,-0.8)};
	\draw  plot  [smooth,tension=1] coordinates{(2.2,0.7) (2.8,1.0) (3.7,-0.8)};
	
	\draw  plot  [smooth,tension=1] coordinates{(1.6,0) (2.8,-0.7) (4.2,1.4)};
	\draw  plot  [smooth,tension=1] coordinates{(2.2,0) (2.8,-0.3) (3.7,1.4)}; 		
	\node at (4.5,-0.4){.}; 		
	\end{tikzpicture}
\end{center}
Note that $\widetilde{\Phi}^{\bi}_{\bj}$ with $\ty(\bi)=\ty(\bj)=k$ is exactly the generator $E_{k}$ introduced in \cite[Definition 4.2]{LZ1} and \cite[Section 6]{LZ4}. We immediately recover the following key lemma in \cite[Proposition 6.1]{LZ1} by using \lemref{propPhigen}.
\begin{lem}\cite[Proposition 6.1]{LZ1}\label{lemOrbOrtho}
	If $r\geq m+1$, then $\Ker F_r^r$ is generated as a 2-sided ideal by $\widetilde{\Phi}^{\bi}_{\bj}$ with $\ty(\bi)=\ty(\bj)=(k)$ for all $k=0,1,\dots,[\frac{m+1}{2}]$.
\end{lem}

Denote by $B^{(1)}_{m+1}(m)$ the two-sided ideal generated by $e_1$ in $B_{m+1}(m)$, and define 
\[F_k:=A(m+1-k)\ot A(k).  \]
Let $\ty(\bi)=\ty(\bj)=(k)$, 
Then from the previous figure for $\widetilde{\Phi}^{\bi}_{\bj}$ we have
\begin{equation}\label{eqPhiortho}
\widetilde{\Phi}^{\bi}_{\bj}\equiv F_k  \pmod{B^{(1)}_{m+1}(m)}\quad \text{and}\quad
\widetilde{\Phi}^{\bi}_{\bj}\,\sigma=\sigma\widetilde{\Phi}^{\bi}_{\bj}=\epsilon(\sigma)\widetilde{\Phi}^{\bi}_{\bj},
\end{equation}
where  $\sigma\in \Sym_{m+1-k}\times  \Sym_k$. The following lemma is from \cite[Lemma 5.10, Corollary 5.14]{LZ1} and \cite[Lemma 6.2]{LZ4}

\begin{lem} \label{lemOrProPhi}
	Let $\bi$ and $\bj$ be standard sequences with $\ty(\bi)=\ty(\bj)=k$, then $\widetilde{\Phi}^{\bi}_{\bj}$ has the following properties:
	\begin{enumerate}
		\item $e_i\,\widetilde{\Phi}^{\bi}_{\bj}=\widetilde{\Phi}^{\bi}_{\bj}\,e_i=0$ for all $e_i\in B_{m+1}(m)$;
		\item $(\widetilde{\Phi}^{\bi}_{\bj})^2=k!(m+1-k)!\,\widetilde{\Phi}^{\bi}_{\bj}$, hence $E_k:=a_k\widetilde{\Phi}^{\bi}_{\bj}$  with $a_k=(k!(m+1-k)!)^{-1}$ is an idempotent for all $0\leq k\leq m+1$.
	\end{enumerate}	
\end{lem}
\begin{proof}
	Part (1) is immediate from  \propref{AnnThm}.
	Also, we deduce that $\widetilde{\Phi}^{\bi}_{\bj}\,D=0$ for any $D\in B^{(1)}_{m+1}(m)$ and it follows from \eqref{eqPhiortho} that
	$(\widetilde{\Phi}^{\bi}_{\bj})^2=\widetilde{\Phi}^{\bi}_{\bj}\,F_k=k!(m+1-k)!\,\widetilde{\Phi}^{\bi}_{\bj}.$
\end{proof}

Now we give a short proof of  the the second fundamental theorem of invariant theory for the orthogonal group obtained  in \cite[Theorem 4.3]{LZ1}.

\begin{thm} \cite[Theorem 4.3]{LZ1}
	The algebra homomorphism $F_r^r: B_r(m)\rightarrow \End_{\Or(m)}(V^{\ot r})$ is an isomorphism if $r\leq m$. If $r\geq m+1$, then $\Ker F_r^r$  is generated as a 2-sided ideal of $B_r(m)$ by the idempotent $E_{[\frac{m+1}{2}]}$.
\end{thm}
\begin{proof} We only need to consider the case with $r\ge m+1$.
	By \lemref{lemOrbOrtho}, it suffices to show the following chain of 2-sided ideals in $B_{m+1}(m)$
	$$\langle E_0\rangle_{m+1} \subseteq \langle E_1\rangle_{m+1} \subseteq \cdots \subseteq \langle E_{[\frac{m+1}{2}]}\rangle_{m+1}.$$
	
	To prove this, we claim that $\langle F_{k} \rangle_0\subseteq \langle F_{k+1}\rangle_0$ for all $k=0,1,\dots,[\frac{m+1}{2}]-1$, where $\langle F_k\rangle_0$ denotes the two-sided ideal generated by $F_k$ in $\Sym_{m+1}$. Actually, by Pieri's rule we know that $\langle F_k\rangle_0$ is the sum of 2-sided simple ideals such that each simple ideal corresponds to Young diagram  with at least $m+1-k$ boxes in its first column and $m+1$ boxes in the first and second columns, whence the claim follows.
	
	Now there exist $\sigma_{kl}, \tau_{kl}\in \Sym_{m+1}$ such that $F_{k}=\sum_{l}\sigma_{kl}F_{k+1}\tau_{kl}$. Note that $E_{k+1}=a_{k+1}\widetilde{\Phi}^{\bi}_{\bj}\equiv a_{k+1} F_{k+1} \pmod{B^{(1)}_{m+1}(m)}$ with $\ty(\bi)=\ty(\bj)=(k+1)$. By part (1) of \lemref{lemOrProPhi},  we have
	$$
	E_{k}\sum_{l}\sigma_{kl}E_{k+1}\tau_{kl}=a_{k+1}E_{k}\sum_{l}\sigma_{kl}F_{k+1}\tau_{kl}=a_{k+1}E_{k}F_{k}=k!(m+1-k)!a_{k+1}E_{k}.
	$$
	Hence $E_{k}\in \langle E_{k+1}\rangle_{m+1}$ for $0\leq k\leq [\frac{m+1}{2}]$ as required.
\end{proof}


\begin{thebibliography}{9999}
	

\bibitem{B} R. Brauer, On algebras which are connected with the semisimple continuous groups,  {\sl  Ann. of Math.  \bf 38(4)} (1937) 857-872.

\bibitem{CCF} C. Carmeli, L. Caston, R. Fioresi,   Mathematical Foundations of Supersymmetry, EMS Series of Lectures in Mathematics, European Mathematical Society (EMS), Z\"urich, 2011.
   
\bibitem{CdVP}  A. Cox, M. De Visscher, and P. Martin.  The blocks of the Brauer algebra in characteristic zero. {\sl Represent. Theory \bf 13} (2009) 272--308. 

\bibitem{DH} S. Doty, J. Hu,  Schur–Weyl duality for orthogonal groups. {\sl Proc. Lond. Math. Soc. \bf 98} (3) (2009) 679--713.

\bibitem{DLZ}   P. Deligne, G. I. Lehrer, and R. B. Zhang, The first fundamental theorem of invariant theory for the orthosymplectic super group. {\sl Adv. Math.}, in press; 	arXiv:1508.04202 [math.RT].

\bibitem{DM} P. Deligne, J. Morgan,  
Notes on Supersymmetry (following Joseph Bernstein),
Quantum fields and strings: a course for mathematicians, 
Vol.  1,  2 (Princeton, NJ, 1996/1997), 
Amer. Math. Soc. Providence, RI, 1999,  pp. 41-97.


\bibitem{EGS} K. Erdmann, J.A. Green, M. Schocker,  Polynomial Representations of $GL_n$: with an Appendix on Schensted Correspondence and Littelmann Paths, Lecture Notes in Mathematics,  Springer Berlin
Heidelberg, 2006.

\bibitem{ES} M. Ehrig, C. Stroppel, Schur-Weyl duality for the Brauer algebra and the ortho-symplectic Lie superalgebra.{\sl Math. Z. \bf 284} (2016), no. 1-2, 595–613. 

\bibitem{ES1} M. Ehrig, C. Stroppel, Koszul gradings on Brauer algebras. {\sl Int. Math. Res. Not.} (2016), no. 13, 3970–4011. 


\bibitem{FH} W. Fulton, J. Harris,  Representation Theory, A First Course,
Graduate Texts in Mathematics, 129, Readings in Mathematics, Springer-Verlag  New York, 1991.



\bibitem{H} J. Hu, Specht filtrations and tensor spaces for the Brauer algebra, {\sl J. Algebraic Combin.   \bf 28} (2) (2008) 281-312.

\bibitem{HX}  J. Hu,  Z.K. Xiao,  On a theorem of Lehrer and Zhang,  {\sl Doc. Math. \bf  17 }(2012) 245-270.

\bibitem{J} G. James, A. Kerber, The Representation Theory of the Symmetric Group, Encyclopedia of Mathematics and its Applications, 16, Addison-Wesley, 1981.

\bibitem{JS} A. Joyal, R. Street,  {\sl Braided tensor categories,  Adv. Math.  \bf 102(1)} (1993) 20-78.

\bibitem{K} V. Kac, Lie superalgebras,  {\sl  Adv.  Math.  \bf 26(1)} (1977)  8-96.

\bibitem{LZ1}  G.I. Lehrer, R.B. Zhang, The second fundamental theorem 
of invariant theory for the orthogonal group, 
{\sl Ann. of Math. \bf 176} (2012) 2031-2054.


\bibitem{LZ4} G.I. Lehrer, R.B. Zhang, The Brauer category and invariant theory, {\sl  J. Eur. Math. Soc.   \bf 17} (2015) 2311-2351.

\bibitem{LZ2} G.I. Lehrer, R.B. Zhang, The first fundamental theorem of invariant theory for the orthosymplectic supergroup. {\sl Comm. Math. Phys. \bf 349}(2017), no. 2, 661–702. 

\bibitem{LZ3} G.I. Lehrer, R.B. Zhang, The second fundamental theorem of invariant theory for the orthosymplectic supergroup, arXiv:1407.1058.

\bibitem{LZ5} G.I. Lehrer, R.B. Zhang, Invariants of the orthosymplectic Lie superalgebra and super Pfaffians. 
{\sl Math. Z. \bf 286} (2017), no. 3-4, 893–917. 

\bibitem{M} I.G. Macdonald, Symmetric Functions and Hall Polynomials,
      Oxford mathematical monographs, Clarendon Press, 1998.

\bibitem{R} C. Raicu,  Products of Young symmetrizers and ideals in the generic tensor algebra, {\sl J. Algebraic Combin. \bf 39(2)} (2014) 247-270.

%


\end{thebibliography}

\end{document}